\documentclass[12pt]{amsart}

\usepackage{amssymb,latexsym}
\usepackage{graphicx,epsfig,color}

\def\cal{\mathcal}
\def\frak{\mathfrak}
\def\Bbb{\mathbb}

\def\rank{\text{\rm rank\,}}

\def\supp{\text{\rm supp\,}}
\def\dist{\text{\rm dist\,}}

\def\la{\langle}

\def\ord{\text{\rm ord\,}}

\def\M{{\cal M}}
\def\N{{\cal N}}
\def\T{{\cal T}}
\def\S{{\cal S}}
\def\R{{\cal R}}
\def\bC{{\Bbb C}}

\def\bN{{\Bbb N}}
\def\bR{{\Bbb R}}

\def\bQ{{\Bbb Q}}
\def\vp{{\varphi}}
\def\al{{\alpha}}
\def\be{{\beta}}
\def\ga{{\gamma}}
\def\la{{\lambda}}
\def\om{{\omega}}
\def\x{(x_1,x_2)}
\def\pa{{\partial}}
\def\ve{{\varepsilon}}
\def\si{{\sigma}}

\def\de{{\delta}}

\def\ka{{\kappa}}
\def\th{{\theta}}

\def\tp{{\tilde\phi}}
\def\pad{{\phi}^{a}}
\def\thd{{\pa_2\th}}

\def\NN{\Bbb N}

\def\RR{\Bbb R}
\def\CC{\Bbb C}
\def\QQ{\Bbb Q}

\def\ve{\varepsilon}

\def\proof{{\noi {\bf Proof. }}}
\def\endproof{{\hfill Q.E.D.\medskip}}

\def\aol{{\left[ \begin{matrix} \al \\ l\end{matrix}\right ]}}

\def\dotol{{\left[ \begin{matrix} \cdot \\ l\end{matrix}\right ]}}
\def\dotola{{\left[ \begin{matrix} \cdot \\ \la\end{matrix}\right ]}}
\def\dotomu{{\left[ \begin{matrix} \cdot \\ \mu\end{matrix}\right ]}}

\def\bpm{\begin{pmatrix}}
\def\epm{\end{pmatrix}}

\def\qed{\smallskip\hfill Q.E.D.\medskip}

\def\bee{\begin{enumerate}}
\def\ee{\end{enumerate}}
\def\bi{\begin{itemize}}
\def\ei{\end{itemize}}

\def\Om{\Omega}

\def\dst{\displaystyle}

\def\noi{\noindent}

\newtheorem{thm}{Theorem}[section]
\newtheorem{namedthm}[thm]{Theorem}
\newtheorem{prop}[thm]{Proposition}
\newtheorem{cor}[thm]{Corollary}
\newtheorem{lemma}[thm]{Lemma}
\newtheorem{remark}[thm]{Remark}

\newtheorem{assumptions}[thm]{Assumptions}
\newtheorem{assumption}[thm]{Assumption}

\begin{document}

\title[maximal operators associated to  hypersurfaces]{Sharp $L^p$-estimates for maximal
operators associated to hypersurfaces in $\bR^3$ for $p>2.$ }

\author[D. M\"uller]{Detlef M\"uller}
\address{Mathematisches Seminar, C.A.-Universit\"at Kiel,
Ludewig-Meyn-Strasse 4, D-24098 Kiel, Germany}
\email{{\tt mueller@math.uni-kiel.de}}
\author[I. A. Ikromov]{Isroil A.Ikromov}
\address{Department of Mathematics, Samarkand State University,
University Boulevard 15,
703004, Samarkand,
Uzbekistan}
 
\email{{\tt ikromov1@rambler.ru}}
\author[M. Kempe]{Michael Kempe}
\address{Mathematisches Seminar, C.A.-Universit\"at Kiel,
Ludewig-Meyn-Strasse 4, D-24098 Kiel, Germany}
\email{{\tt kempe@math.uni-kiel.de}}

\thanks{We acknowledge the support for this work by the Deutsche
Forschungsgemeinschaft.}

\begin{abstract}
We study the boundedness problem   for maximal operators $\M$ associated to smooth hypersurfaces $S$  in 3-dimensional Euclidean space.
For $p>2,$ we prove that if no affine tangent plane to $S$ passes through the origin and  $S$ is analytic, then the associated maximal operator is  bounded on $L^p(\RR^3)$ if and only if  $p>h(S),$ where $h(S)$ denotes the  so-called height of the surface $S.$ For non-analytic finite type $S$ we obtain the same statement with the exception of the exponent $p=h(S).$ Our notion of  height $h(S)$ is closely related to A.~N.~Varchenko's  notion of height $h(\phi)$   for  functions $\phi$ such that  $S$ can be locally represented as the graph of $\phi$ after a rotation of coordinates. 

  Several  consequences of this result are discussed. In particular we verify  a conjecture by E.~M.~Stein  and its generalization by  A.~Iosevich and E.~Sawyer on the connection between the decay rate of the Fourier transform of the surface measure on $S$  and the $L^p$-boundedness of the associated maximal operator $\M$, and a conjecture by Iosevich and Sawyer which relates the  $L^p$-boundedness of $\M$ to an integrability condition on $S$  for the distance function to tangential hyperplanes,  in dimension three. 
  
  In particular,  we also give essentially sharp uniform estimates for  the Fourier transform of the surface measure on $S,$  thus  extending  a result by V.~N.~Karpushkin from the analytic to the smooth setting and implicitly verifying a conjecture by V.~I.~Arnol'd  in our context.
\end{abstract}
\maketitle


\tableofcontents

\thispagestyle{empty}

\section{Introduction}\label{introduction}

Let $S$ be a smooth hypersurface in $\RR^n$ and let  $\rho\in
C_0^\infty(S)$ be a smooth non-negative function with compact support.  Consider
the associated averaging operators
$A_t, t>0,$ given by 
$$
A_tf(x):=\int_{S} f(x-ty) \rho(y) \,d\si(y),
$$
where $d\si$ denotes the surface measure on $S.$   The 
associated maximal operator is given by 
\begin{equation}\label{1.1}
\M f(x):=\sup_{t>0}|A_tf(x)|, \quad ( x\in \RR^n).
\end{equation}
We remark that by testing $\M$ on the characteristic function of the unit ball in $\RR^n,$ it is easy to see  that a necessary condition for $\M$ to be  bounded  on $L^p(\RR^n)$  is that  $p> n/(n-1).$ 
\smallskip

In 1976,  E.~M.~Stein \cite{stein-sphere} proved that conversely, if $S$ is the Euclidean unit sphere in $\RR^n,\ n\ge 3, $  then the corresponding spherical maximal operator is bounded  on $L^p(\RR^n)$  for every $p> n/(n-1).$ The analogous result in dimension  $n=2$ was later proved by J.~Bourgain \cite{bourgain}. These results became  the starting point for intensive  studies of various classes of maximal operators associated to subvarieties. Stein's  monography  \cite{stein-book} is an excellent reference to many of these developments. From these early works,  the influence of geometric properties on the validity of $L^p$-estimates of the maximal operator $\M$ became evident.
For instance, A.~Greenleaf \cite{greenleaf} proved that $\M$ is bounded on $L^p(\RR^n)$ if $n\geq 3$ and $p>\frac{n}{n-1},$
provided $S$ has everywhere non-vanishing Gaussian curvature and in addition
$S$  is starshaped with respect to the origin.

In contrast, the case where the Gaussian curvature vanishes at some points is still wide open, with the exception of the two-dimensional case $n=2,$  i.e.,  the case of  finite type curves in $\RR^2$ studied by A.~Iosevich in  \cite{iosevich-curves}. 
As a partial result in higher dimensions,  C.~D.~Sogge and E.~M.~Stein showed  in \cite{sogge-stein}  that if the Gaussian curvature of $S$ does not vanish to infinite order at  any  point of $S,$ then $\M$ is bounded on $L^p$ in a certain range $p>p(S).$ However, the exponent $p(S)$ given in that article is in general far from being optimal, and   in dimensions $n\ge 3,$  sharp results  are known only for particular classes of hypersurfaces. 

The perhaps best understood class in higher dimensions is the class of convex hypersurfaces of finite line type (see in particular the early work in this setting by   M.~Cowling and G.~Mauceri in  \cite{cowling-mauceri1},   \cite{cowling-mauceri2},  the work by  A.~Nagel, A.~Seeger and S.~Wainger in \cite{nagel-seeger-wainger}, and the articles \cite{iosevich-sawyer1}, \cite{iosevich-sawyer2} and \cite{io-sa-seeger} by A.~Iosevich, E.~Sawyer and A.~Seeger).  In  \cite{nagel-seeger-wainger}, sharp results were for instance obtained for convex hypersurfaces which are given as the graph of a mixed homogeneous
convex function $\phi.$   Further results  were based on a result due to Schulz \cite{schulz}(see also \cite{vasilev}),
which states that, possibly after a rotation of  coordinates, any smooth convex
 function $\phi$ of finite line  type can be written in the form $\phi=Q+\phi_r$,
where $Q$ is a convex mixed homogeneous polynomial that vanishes only at the origin, and $\phi_r$
is a remainder term consisting of terms of higher homogeneous degree than the  polynomial $Q$.
By means of this result, Iosevich and Sawyer   proved in \cite{iosevich-sawyer2} sharp  $L^p$-estimates  for the maximal operator $\M$ for  $p>2.$ For further results in the case $p\le 2,$ see also \cite{stein-book}.
\smallskip
 
 As is well-known since the early work of E.~M.~Stein  on the spherical maximal operator, the estimates of the maximal operator $\M$ on Lebesgue spaces are intimately connected with the decay rate of the Fourier transform 
$$\widehat{\rho d\si}(\xi)=\int_S e^{-i\xi\cdot x}\rho(x)\, d\si(x),\quad \xi\in\RR^n,
$$
of the superficial measure $\rho d\si,$ i.e., to estimates of oscillatory integrals.
These in return are closely related to  geometric properties of the surface $S,$ and have  been considered by numerous authors ever since the early work by  B.~Riemann  on this subject (see \cite{stein-book} for further information).
Also the afore mentioned results for convex hypersurfaces of finite line type are based on such estimates. Indeed, sharp estimates for the Fourier tranform of superficial measures on  $S$ have been obtained by
J.~Bruna , A.~Nagel and S.~Wainger in \cite{bruna-n-w}, improving on previous results by B.~Randol \cite{randol} and I.~Svensson \cite{svensson}. They introduced a family of nonisotropic balls on $S$,
called "caps", by setting
$$
B(x,\de):=\{y\in S: \dist(y,x+T_xS)<\de\},\ \de>0.
$$
Here $T_xS$ denotes the tangent space to $S$ at $x\in S$. Suppose that $\xi$ is normal to $S$
at the point $x^0$. Then it was shown that$$
|\widehat{\rho d\si}(\xi) |\leq C|B(x^0,|\xi|^{-1})|,
$$
 where   $|B(x^0,\de)|$ denotes the surface area of $B(x^0,\de).$ These estimate became fundamental  also in the subsequent work on associated maximal operators.

\bigskip
However, such estimates fail to be true for non-convex hypersurfaces, which we shall be dealing with in this article.
More precisely,  we shall consider  general smooth hypersurfaces  in $\RR^3.$ 

\medskip
Assume that  $S\subset \RR^3$ is such a hypersurface, and  let   $x^0\in S$  be a fixed point in $S.$  We can then  find a Euclidean motion of   $\RR^3,$ so that in the new coordinates given by this motion, we can assume
that $x^0=(0,0,1) $ and $T_{x^0}=\{x_3=0\}.$ Then, in a  neighborhood $U$  of the origin,  the hypersurface  $S$ is given as the  graph
$$
U\cap S=\{(x_1,x_2,1+ \phi\x): \x\in \Om \}
$$
of a smooth function $1+\phi$ defined on an open neighborhood $\Om$ of $0\in\RR^2$ and satisfying the conditions
\begin{equation}\label{1.2}
\phi(0,0)=0,\, \nabla \phi(0,0)=0.
\end{equation}

 To $\phi$ we can then associate the so-called height $h(\phi)$ in the sense of A.~N.~Varchenko \cite{varchenko} defined in terms of the Newton polyhedra of $\phi$ when represented in smooth coordinate systems near the origin (see Section \ref{newton} for details). An important property of this height is that it is invariant under local smooth changes of coordinates fixing the origin.  We then define the {\it height} of $S$ at the point $x^0$  by
$h(x^0,S):=h(\phi).$ This notion can easily be seen to be invariant under affine linear changes of coordinates in the ambient space  $\RR^3$  (cf. Section \ref{sharpness}) because of the invariance property of $h(\phi)$ under local coordinate changes. 

\medskip
Now observe that unlike linear transformations, translations do not  commute with dilations, which is why Euclidean motions are no admissible coordinate changes for the study of the maximal operators $\M.$ We shall therefore study $\M$ under the following transversality assumption on $S.$

\begin{assumption}\label{s1.1}
The affine tangent plane $x+T_xS$ to $S$ through $x$ does not pass through the origin in $\RR^3$  for every $x\in S.$ Equivalently, $x\notin T_xS$ for every $x\in S,$ so that $0\notin S,$ and $x$ is transversal to $S$ for every point $x\in S.$ 
\end{assumption}

Notice that this assumption allows us to find a linear change of coordinates in $\RR^3$ so that in the new coordinates
 $S$ can locally be represented as the graph of a function $\phi$ as before, and that the norm of $\M$ when acting on $L^p(\RR^3)$ is invariant under such a linear change of coordinates.

\smallskip
If $\phi$ is flat, i.e., if all derivatives of $\phi$ vanish at the origin, and if $\rho(x^0)>0,$  then it is well-known and easy to see that the maximal operator $\M$ is $L^p$-bounded if and only if $p=\infty,$ so that this case is of no interest. Let us therefore always assume in the sequel that $\phi$ is non-flat, i.e., of finite type. Correspondingly, we shall always assume without further mentioning that the hypersurface $S$ is of finite type in the sense that every tangent plane has finite order of contact.

\medskip

  We can now state the main result of this article.

\begin{thm}\label{s1.2}
Assume that $S$ is a smooth hypersurface in $\RR^3$  satisfying  Assumption~\ref{s1.1}, and let $x^0\in S$ be a fixed point. Then
there exists a neighborhood $U\subset S $ of the point $x^0$ such
that for any $\rho\in C_0^\infty(U)$ the associated maximal
operator $\M$ is bounded on $L^p(\RR^3)$ whenever $p>\max\{h(x^0,S), 2\}.$
\end{thm}
Notice that even in the case where $S$ is convex this result is stronger than the known results, which always assumed that $S$ is of finite line type.

The following Theorem shows the sharpness of this theorem.

\begin{thm} \label{s1.3}
Assume that the maximal operator  $\M$ is bounded on $L^p(\RR^3) $ for some  $p>1,$ 
where $S$ satisfies Assumption \ref{s1.1}.
Then, for any point
$x^0\in S$  with  $\rho(x^0)>0,$ we have
$h(x^0,S)\le p.$ Moreover, if $S$ is analytic at such a point $x^0,$
then $h(x^0,S)<p.$
\end{thm}

As an  immediate consequence of these two results, we obtain
\begin{cor}\label{s1.4}
Suppose $S$ is a smooth hypersurface in $\RR^3$ satisfying Assumption \ref{s1.1}, and let  $x^0\in S$ be a fixed
point. Then there exists a neighborhood $U\subset S$ of this  point
such that $h(x, S)\le h(x^0,S)$  for every $x\in U.$ 
\end{cor}
 This shows in particular that  if $\Phi(x,s)=\phi\x+s_1x_1+s_2x_2$ is a smooth deformation by linear terms  of a smooth, finite type function $\phi$ defined near the origin in $\RR^2$ and satisfying \eqref{1.2}, then  the height of $\Phi(\cdot,s)$  at any  critical point  of the function $x\mapsto \Phi(x,s)$ is bounded by the height at $h(\Phi(\cdot,0))=h(\phi)$ for sufficiently small perturbation parameters $s_1$ and $s_2.$ This proves  a conjecture by  V.I.~Arnol'd \cite{arnold} in the smooth  setting at least for linear perturbations. For analytic functions $\phi$ of two variables, such a result has been proved for arbitrary analytic deformations by V.~N.~Karpushkin \cite{karpushkin}.
 \medskip

From these results, global results can be deduced easily. For instance, if $S$ is a compact hypersurface, then we define the {\it height} $h(S)$  of $S$ by $h(S):=\sup_{x\in S}h(x,S).$
 Corollary \ref{s1.4} shows that in fact 
 $$h(S):=\max_{x\in S}h(x,S)<\infty,
 $$
 and from Theorems \ref{s1.2}, \ref{s1.3} we obtain
 \begin{cor}\label{s1.5}
Assume that $S$ is a smooth, compact hypersurface  in $\RR^3$ satisfying   Assumption \ref{s1.1}, that  $\rho>0$ on $S$ and that $p>2.$  

 If $S$ is analytic, then the associated maximal
operator $\M$ is bounded on $L^p(\RR^3)$ if and only if  $p>h(S).$ If $S$ is only assumed to be smooth,  then for  $p\ne h(S)$ we still have that the maximal operator $\M$ is bounded on $L^p(\RR^3)$ if and only if  $p>h(S).$ 
\end{cor}

\medskip 

Let $H$ be an affine  hyperplane in $\RR^3.$  Following A.~Iosevich and E.~Sawyer \cite{iosevich-sawyer1}, we consider the  distance   $d_H(x):=\dist(H,x)$  from $x\in S$ to  $H.$ In particular,  if $x^0\in S,$ then $d_{T,x^0}(x):=\dist (x^0+T_{x^0}S, x)$ will denote  the distance from $x\in S$ to the affine tangent plane to $S$ at the  point $x^0.$ The following result has been proved in \cite{iosevich-sawyer1} in arbitrary dimensions $n\ge 2$  and without requiring Assumption \ref{s1.1}.

\begin{namedthm}[Iosevich-Sawyer]\label{s1.6}
If the maximal operator $\M$ is bounded on $L^p(\RR^n),$ where $p>1,$ then
\begin{equation}\label{1.3}
\int_S d_H(x)^{-1/p}\,\rho(x)\,  d\si(x)<\infty
\end{equation}
for every affine hyperplane $H$ in $\RR^n$ which does not pass through the origin.
\end{namedthm}
Moreover, they conjectured that for $p>2$ the condition \eqref{1.3} is indeed necessary  and sufficient for the boundedness of  the maximal operator $\M$ on $L^p,$ at least if for instance $S$ is compact and $\rho>0.$
\begin{remark}\label{s1.7}
Notice that  condition \eqref{1.3} is easily seen to be true for every affine hyperplane $H$ which is nowhere tangential to $S,$ so that it is in fact  a condition on affine tangent hyperplanes to $S$ only. Moreover, if Assumption \ref{s1.1} is satisfied, then there are no affine tangent hyperplanes which pass through the origin, so that in this case it is a condition on all affine tangent hyperplanes.
\end{remark}

In Section \ref{sharpness}, we shall prove 
\begin{prop}\label{s1.8}
Suppose $S$ is a smooth hypersurface in $\RR^3,$ and let  $x^0\in S$ be a fixed point. Then, for every $p<h(x^0,S),$ we have  
\begin{equation}\label{1.4}
\int_{S\cap U} d_{T,x^0}(x)^{-1/p}\,d\si(x)=\infty
\end{equation}
for every  neighborhood $U$ of $x^0.$ Moreover, if $S$ is analytic near $x^0,$  then \eqref{1.4}  holds true  also for $p=h(x^0,S).$
\end{prop}

Notice that  this result does not require Assumption \ref{s1.1}.

As an immediate consequence of Theorem \ref{s1.2}, Theorem \ref{s1.6} and Proposition \ref{s1.8}  we obtain
\begin{cor}\label{s1.9}
Assume that $S\subset \RR^3$ satisfies Assumption \ref{s1.1}, and let $x^0\in S$ be a fixed point. Moreover, let $p>2.$ 

Then, if $S$ is analytic near $x^0,$ 
there exists a neighborhood $U\subset S $ of the point $x^0$ such
that for any $\rho\in C_0^\infty(U)$ with $\rho(x^0)>0$ the associated maximal
operator $\M$ is bounded on $L^p(\RR^3)$ if and only if condition \eqref{1.3} holds for every affine hyperplane $H$ in $\RR^3$ which does not pass through the origin. 

If $S$ is only assumed to be smooth near $x^0,$ then the same conclusion holds true, with the possible exception of the exponent $p=h(x^0,S).$ 
\end{cor}
This confirms the conjecture by Iosevich and Sawyer in our setting for analytic  $S$, and for smooth $S$ with the possible exception of the exponent $p=h(x^0,S).$ For the critical exponent $p=h(x^0,S),$ if $S$ is not analytic near $x^0,$ examples show that unlike in the analytic case it may happen that $\M$ is bounded on $L^{h(x^0,S)}$ (see, e.g., \cite{iosevich-sawyer2}), and  the conjecture remains open for this value of $p.$ For further details, we refer to Section \ref{sharpness}.

\bigskip
As mentioned before,  the estimates of the maximal operator $\M$ on Lebesgue spaces are intimately connected with the decay rate of the Fourier transform 
$$\widehat{\rho d\si}(\xi)=\int_S e^{-i\xi\cdot x}\rho(x)\, d\si(x),\quad \xi\in\RR^n,
$$
of the superficial measure $\rho d\si. $ Estimates of such oscillatory integrals will naturally  play a central role also in our  proof Theorem \ref{s1.2}. Indeed our proof of Theorem \ref{s1.2} will provide enough information that it will also be easy to derive from it the following uniform estimate for the Fourier transform of surface carried measures on $S.$

\begin{thm}\label{s1.10n}
Let $S$ be  a smooth hypersurface of finite type in $\RR^3$ and let $x^0$ be a fixed point in $S.$ Then there exists a neighborhood $U\subset S $ of the point $x^0$ such
that for every  $\rho\in C_0^\infty(U)$ the following estimate holds true:
\begin{equation}\label{1.5n}
|\widehat{\rho d\si}(\xi)|\le C\,||\rho||_{C^3(S)}\,\log(2+|\xi|)(1+|\xi|)^{-1/h(x^0,S)}\  \mbox{ for every } \xi \in\RR^3.
\end{equation}
\end{thm}
This estimate  generalizes Karpushkin's estimates in \cite{karpushkin} from the analytic to the finite type setting, at least for linear perturbations. 

\medskip
The next result establishes a direct link between the  decay rate of $\widehat{\rho d\si}(\xi)$ and Iosevich-Sawyer's condition \eqref{1.3}. In combination with Proposition \ref{s1.8} it shows in particular that the exponent $-1/h(x^0,S)$ in estimate \eqref{1.5n} is sharp (for the case of analytic hypersurfaces, the latter follows also from Varchenko's asymptotic expansions of oscillatory integrals in \cite{varchenko}).

\begin{thm}\label{s1.10}
Let $S$ be  a smooth hypersurface in $\RR^n,$ and let $\rho\in C_0^\infty(S)$ be a smooth cut-off function $\rho\ge 0,$ and assume that 
\begin{equation}\label{1.5}
|\widehat{\rho d\si}(\xi)|\le C_\be\,(1+|\xi|)^{-\be} \mbox{ for every } \xi \in\RR^n,
\end{equation}
for some $\be>0.$  Then for every  $p>1$ such that $p>1/\be,$
\begin{equation}\label{1.6}
\int_S d_H(x)^{-1/p}\,\rho(x)\,  d\si(x)<\infty,
\end{equation}
for every affine hyperplane $H$ in $\RR^n.$
\end{thm}
In combination with Proposition \ref{s1.8} 
this result easily implies (see Section \ref{sharpness})
\begin{cor}\label{s1.11}
Suppose $S$ is a smooth hypersurface in $\RR^3,$   let  $x^0\in S$ be a fixed point  and assume that  the estimate \eqref{1.5} holds true for some $\be>0.$    If $\rho(x^0)>0,$  and if  $\rho$ is supported in a sufficiently small neighborhood of $x^0,$ then necessarily $\be\le 1/h(x^0,S).$

\end{cor}

Indeed, more is true. Let us introduce the following quantities. In analogy with V.~I.~Arnol'd's notion of the ''singularity index'' \cite{arnold},  we define the {\it uniform oscillation index} $\be_u(x^0,S)$ of the hypersurface $S\subset \RR^n$ at the point $x^0\in S$ as follows:
\smallskip

Let $\frak B_u(x^0,S)$ denote the set of all $\be\ge 0$ for which   there exists an open  neighborhood $U_\be$ of $x^0$ in $S$  such   that estimate \eqref{1.5} holds true for every function $\rho\in C^\infty_0(U_\be).$ Then 
$$\be_u(x^0,S):=\sup\{\be:\be\in \frak B_u(x^0,S)\}.
$$
If we restrict our attention to the normal direction to $S$ at $x^0$ only, then  we can define analogously the notion of {\it oscillation index} of the hypersurface $S$ at the point $x^0\in S.$ More precisely, if $n(x^0)$ is a unit normal to $S$ at $x^0,$ then we let  $\frak B(x^0,S)$ denote the set of all $\be\ge 0$ for which   there exists an open  neighborhood $U_\be$ of $x^0$ in $S$  such   that estimate \eqref{1.5} holds true along the line $\RR\,  n(x^0)$ for every function $\rho\in C^\infty_0(U_\be),$ i.e., 
\begin{equation}\label{1.7}
|\widehat{\rho d\si}(\la n(x^0))|\le C_\be\,(1+|\la|)^{-\ga} \mbox{ for every } \la \in\RR.
\end{equation}
Then 
$$\be(x^0,S):=\sup\{\be:\be\in \frak B(x^0,S)\}.
$$
If we regard $S$ locally as the graph of a function $\phi,$ then we can introduce related notions $\be_u(\phi)$ and $\be(\phi)$ for  $\phi,$ regarded as the phase function of an oscillatory integral (cf. \cite{ikromov-m}, and also Section \ref{sharpness}).

\medskip  
We also define  the {\it uniform contact index} $\ga_u(x^0,S)$ of the hypersurface $S$ at the point $x^0\in S$ as follows:
\smallskip

Let $\frak C_u(x^0,S)$ denote the set of all $\ga\ge 0$ for which   there exists an open  neighborhood $U_\ga$ of $x^0$ in $S$  such   that the  estimate
\begin{equation}\label{1.8}
\int_{U_\ga} d_H(x)^{-\ga}\,  d\si(x)<\infty
\end{equation}
holds true for every affine hyperplane $H$ in $\RR^n.$ Then we put
$$\ga_u(x^0,S):=\sup\{\ga:\ga\in \frak C_u(x^0,S)\}.
$$
Similarly, we let $\frak C(x^0,S)$ denote the set of all $\ga\ge 0$ for which   there exists an open  neighborhood $U_\ga$ of $x^0$ in $S$  such  
\begin{equation}\label{1.9}
\int_{U_\ga} d_{T,x^0}(x)^{-\ga}\,  d\si(x)<\infty,
\end{equation}
and call
$$\ga(x^0,S):=\sup\{\ga:\ga\in \frak C(x^0,S)\}
$$
the {\it  contact index} $\ga(x^0,S)$ of the hypersurface $S$ at the point $x^0\in S.$ Then clearly 
\begin{equation}\label{1.10}
\be_u(x^0,S)\le \be(x^0,S),\quad \ga_u(x^0,S)\le \ga(x^0,S).
\end{equation}

At least for hypersurfaces in $\RR^3,$   a lot more is true.

\begin{thm}\label{s1.12}
Let $S$ by a smooth, finite type hypersurface in  $\RR^3,$ and let $x^0\in S$ be a fixed point. Then 
$$
\be_u(x^0,S)= \be(x^0,S)=\ga_u(x^0,S)= \ga(x^0,S)=1/h(x^0,S).
$$
\end{thm}

\medskip
Let us recall at this point  a result by A.~Greenleaf. In  \cite{greenleaf} he proved that if $\widehat{\rho d\si}(\xi)=O(|\xi|^{-\be}) \ \mbox{as}\, |\xi|\to \infty$ and if  $\be>1/2,$ then the maximal operator is bounded on $L^p$ whenever
$p>1+\frac 1{2\be}.$ The case $\be\le 1/2$ remained open.

For $\be=1/2$  E.~M.~Stein and  later  for the full range 
$\be\le 1/2$   A.~Iosevich and E.~Sawyer  \cite{iosevich-sawyer2} {\it conjectured that  if $S$ is a smooth, compact hypersurface in $\RR^n$ such that 
$$
|\widehat{\rho d\si}(\xi)|=O(|\xi|^{-\be}) \,\mbox{for some }\,
0<\be\le1/2,
$$
 then the maximal operator $\M$ is bounded on $L^p(\RR^n)$ for  every $p>1/\be,$ at least if we assume $ \rho>0.$ }
 
 \medskip

A partial confirmation of Stein's conjecture  has been given by C.~D.~Sogge \cite{sogge} who proved that 
if the surface has at least one non-vanishing principal curvature everywhere, then the maximal operator is
 $L^p$-bounded for every $p>2.$ Certainly, if the surface has at least one non-vanishing principal curvature then the 
 estimate above holds  for $\be=1/2$.

\medskip

Now, if $n=3,$ and if $0<\be\le 1/2,$ then $\be_u(x^0,S)\ge \be$ for every point $x^0\in S,$ so that our Theorem \ref{s1.12} implies that $1/\be\ge h(x^0,S).$ Then, if $p>1/\be,$ we have $p>\max\{2,h(x^0,S)\}.$ Therefore, by means of a partition of unity argument, we obtain from Theorem \ref{s1.2}  the following confirmation of the Stein-Iosevich-Sawyer conjecture in this case.

\begin{cor}\label{s1.13}
Let  $S$ be a smooth compact  hypersurface in $\RR^3$  satisfying  Assumption~\ref{s1.1}, and let $\rho>0$ be a smooth density on $S.$ We assume that there is some $0<\be\le1/2$ such that 
$$
|\widehat{\rho d\si}(\xi)|=O(|\xi|^{-\be}).
$$
Then the associated maximal operator $\M$ is bounded on $L^p(\RR^3)$ for every  $p>1/\be.$
\end{cor}

  \bigskip
We finally remark that the case $p\le 2$ behaves quite differently, and  examples show that neither condition \eqref{1.3} nor the notation of height will be suitable to  determine the range of exponents $p$ for which the maximal operator $\M$ is $L^p$-bounded (see, e.g., \cite{io-sa-seeger}). The study of this range for $p\le 2$ is  work in progress.

\bigskip
\subsection{Outline of the proof of Theorem \ref{s1.2} and  organization of the article}

The  proof of our main result, Theorem \ref{s1.2}, will strongly make use of the results in \cite{ikromov-m} on the existence of a so-called ''adapted'' coordinate system for a smooth, finite type function $\phi$ defined near the origin in 
$\RR^2$ (see Section \ref{newton} for some basic notation).
These results generalize the corresponding results for analytic $\phi$ by A.~N.~Varchenko \cite{varchenko}, by means of a simplified  approach inspired by the work of Phong and Stein \cite{phong-stein}.  According to these results, one  can always find a change of coordinates of the form
$$y_1:= x_1, \ y_2:= x_2-\psi(x_1)
$$
which leads to adapted coordinates $y.$  The function $\psi$ can be constructed from the Pusieux series expansion of roots of $\phi$ (at least if $\phi$ is analytic) as the so-called principal root jet (cf.\cite{ikromov-m}).  Somewhat simplifying, it  agrees with  a real-valued leading part of the (complex) root  of $\phi$ near which $\phi$ is ''small of highest order'' in an averaged sense.
  One would preferably like to work in these adapted coordinates $y,$ since the height of $\phi$ when expressed in these adapted coordinates,  can be read off directly from the Newton polyhedron of $\phi$ as the so-called ''distance.''  However, this change of coordinates  leads to substantial problems, since it is in general non-linear. 

Now, away from the curve $x_2=\psi(x_1),$ it turns out that one can find some $k$ with $2\le k\le h(\phi)$ such that 
$\pa_2^k\phi\ne 0.$ This suggests that one may apply the results on maximal functions on curves in \cite{iosevich-curves}. Indeed this is possible, but we need  estimates for such maximal operators along curves which are stable under small perturbations of the given curve. Such results, which will be based on the local smoothing estimates by G.~Mockenhaupt, A.~Seeger and C.~Sogge in \cite{mockenhaupt-seeger-sogge}, and related estimates for maximal operators along surfaces, are derived in Section \ref{uniform estimates}. The necessary control on partial derivatives $\pa_2^k\phi$ will be obtained from the study of mixed homogeneous polynomials in Section \ref{multiplicities}. Indeed, in a similar way as the Schulz polynomial is used in the convex case to approximate the given function $\phi,$ we  shall approximate the function $\phi$ in domains close to a given root of $\phi$ by a suitable mixed homogeneous 
polynomial, following here some ideas in \cite{phong-stein}.

The case where our original coordinates $x$ are adapted or where the height $h(\phi)$ is strictly less than $2$ is the simplest one, since we can here avoid  non-linear changes of coordinates. This case is dealt with in Section \ref{adapted}.

We then concentrate on the situation where $h(\phi)\ge 2$ and where the coordinates are not adapted. The contributions to  the maximal operator $\M$  by  a suitable homogeneous domain away from the curve $x_2=\psi(x_1)$ require a lot more effort and are  estimated in Section \ref{away-root} by means of the results in Sections \ref{uniform estimates} and \ref{multiplicities}. 

There remains the domain near the curve $x_2=\psi(x_1).$ For this  domain, it is in general no longer possible to reduce its  contribution to the maximal operator $\M$ to maximal operators along curves, and we have to apply two-dimensional oscillatory integral technics. Indeed, we shall need estimates for certain classes of oscillatory integrals with small parameters, which will be given in Section \ref{oscint}. These results will be applied in Sections \ref{near-root} and \ref{proof} in order to complete the proof of Theorem \ref{s1.2}.
\medskip

We remark that  our proof does not make use of any  damping technics, which had been crucial to many other approaches.

\medskip The proof of Theorem \ref{s1.10n}, which will be given in Section \ref{uniest}, can easily be obtained from the results established in the course of the proof of  Theorem \ref{s1.2}, except for the case $h(x^0,S)<2,$ which, however, has been studied in a complete way by Duistermaat \cite{duistermaat}. The main difference  is that we have to replace the estimates for maximal operators in Section \ref{uniform estimates} by van der Corput type estimates due to J.~E.~Bj\"ork and G.~I.~Arhipov.
\medskip

In the last Section \ref{sharpness}, we shall give proofs of all the other results stated above.

\bigskip

\setcounter{equation}{0}
\section{Newton diagrams and adapted coordinates}\label{newton}

We recall here some basic notation (compare, e.g., \cite{ikromov-m} for further information). 
Let $\phi$ be a smooth real-valued  function defined on a neighborhood of the origin in $\bR^2$ with $\phi(0,0)=0,\, \nabla \phi(0,0)=0,$ and consider the associated Taylor series 
$$\phi(x_1,x_2)\sim\sum_{j,k=0}^\infty c_{jk} x_1^j x_2^k$$
of $\phi$ centered at  the origin. 
The set
$$\T(\phi):=\{(j,k)\in\bN^2: c_{jk}=\frac 1{j!k!}\partial_{ x_1}^j\partial_{ x_2}^k \phi(0,0)\ne 0\}
$$ 
will be called the {\it Taylor support } of $\phi$ at $(0,0).$  We shall always assume that 
$$\T(\phi)\ne \emptyset,$$
i.e., that the function $\phi$ is of finite type at the origin. If $\phi$ is real analytic, so that the Taylor series converges to $\phi$ near the origin, this just means that $\phi\ne 0.$ 
The {\it Newton polyhedron} $\N(\phi)$ of $\phi$ at the origin is defined to be the convex hull of the union of all the quadrants $(j,k)+\bR^2_+$ in $\bR^2,$ with $(j,k)\in\T(\phi).$  The associated {\it Newton diagram}  $\N_d(\phi)$ in the sense of Varchenko \cite{varchenko}  is the union of all compact faces  of the Newton polyhedron; here, by a {\it face,} we shall  mean an edge or a vertex.

We shall use coordinates $(t_1,t_2)$ for points in the plane containing the Newton polyhedron, in order to distinguish this plane from the $(x_1,x_2)$ - plane. 

The {\it distance} $d=d(\phi)$  between the Newton
polyhedron and the origin in the sense of Varchenko is given by the coordinate $d$ of the point $(d,d)$ at which the bisectrix   $t_1=t_2$ intersects the boundary of the Newton polyhedron.

The {\it principal face} $\pi(\phi)$  of the Newton polyhedron of $\phi$  is the face of minimal dimension  containing the point $(d,d)$. Deviating from the notation in \cite{varchenko}, we shall call the series
$$\phi_p(x_1,x_2):=\sum_{(j,k)\in \pi(\phi)}c_{jk} x_1^j x_2^k
$$
the {\it principal part} of $\phi.$ In case that $\pi(\phi)$ is compact,  $\phi_p$ is a mixed homogeneous polynomial; otherwise, we shall  consider $\phi_p$ as a formal power series. 

Note that the distance between the Newton polyhedron and
the origin depends on the chosen local coordinate system in which $\phi$ is expressed.  By a  {\it local analytic (respectively smooth) coordinate system at the origin} we shall mean an analytic (respectively smooth)   coordinate system defined near the origin which preserves $0.$ If we work in the category of smooth functions $\phi,$  we shall always consider smooth coordinate systems, and if $\phi$ is analytic, then one usually restricts oneself to analytic coordinate systems (even though this will not really be necessary for the questions we are going to study, as we will see). The {\it height } of the analytic (respectively smooth) function $\phi$ is defined by 
$$h(\phi):=\sup\{d_x\},$$
 where the
supremum  is taken over all local analytic (respectively smooth) coordinate systems $x$ at the origin, and where $d_x$
is the distance between the Newton polyhedron and the origin in the
coordinates  $x$.

A given   coordinate system $x$ is said to be
 {\it adapted} to $\phi$ if $h(\phi)=d_x.$

\medskip
\subsection{The principal part of $\phi$ associated to a supporting line of the Newton polyhedron as a mixed homogeneous polynomial }

Let $\ka=(\ka_1,\ka_2)$ with $\ka_1,\ka_2>0$ be a given weight, with associated one-parameter family of dilations $\delta_r(x_1,x_2):= (r^{\ka_1}x_1,r^{\ka_2}x_2),\ r>0.$  A function  $\phi$ on $\bR^2$ is  said to be   {\it $\ka$-homogeneous of degree $a,$} if $\phi(\delta_rx)=r^a \phi(x)$ for every $r>0, x\in\bR^2.$ Such functions will also be called {\it mixed homogeneous.} 
The exponent $a$ will be denoted as the  {\it $\ka$-degree} of $\phi.$ For instance, the monomial $x_1^j x_2^k$ has $\ka$-degree $\ka_1 j+\ka_2k.$

If $\phi$ is an arbitrary smooth function near the origin,  consider its Taylor series $\sum_{j,k=0}^\infty c_{jk} x_1^j x_2^k$  around the origin.  We choose $a$ so that the line  $L_\ka:=\{(t_1,t_2)\in\bR^2:\ka_1 t_1+\ka_2t_2=a\}$ is the supporting  line to the Newton polyhedron $\N(\phi) $ of $\phi.$ Then the non-trivial polynomial 
$$\phi_\ka(x_1,x_2):=\sum_{(j,k)\in L_\ka} c_{jk} x_1^j x_2^k
$$
is $\ka$-homogeneous of degree $a;$ it will be called the {\it $\ka$-principal part} of $\phi.$  By definition, we then have 
\begin{equation}\label{princ1}
\phi(x_1,x_2)=\phi_\ka(x_1,x_2) +\ \mbox{terms of higher $\ka$-degree.}
\end{equation}
More precisely, we mean by this that  every point $(j,k)$ in the Taylor support of the   remainder term $\phi_r:=\phi-\phi_\ka$ lies  on a line $\ka_1t_1+\ka_2t_2=d$ with $d>a$   parallel to, but above the line $L_\ka,$  i.e., 
we have $\ka_1j+\ka_2k>a.$ 
Moreover, clearly
$$
\N_d(\phi_\ka)\subset \N_d(\phi).
$$

\bigskip

\setcounter{equation}{0}
\section{Uniform estimates for maximal operators associated to families of
finite type curves and related surfaces}\label{uniform estimates}

\subsection{Finite type curves}\label{curves}

In this subsection, we shall prove an extension of some results by
Iosevich [Ios], which allows for uniforms estimates for maximal
operators associated to families of curves which arise as small
perturbations of a given curve.

We begin with a result whose proof is based on Iosevich's approach
in [Ios].

\begin{prop}\label{prop3.1}
Consider averaging operators along curves in the plane  of the form
$$A_tf(x)=A_t^{(\rho,\eta,\tau)}f(x):=\int_{\RR} f\Big(x_1-t(\rho_1s+\eta_1),x_2-t(\eta_2+\tau
s+\rho_2g(s))\Big)\, \psi(s)\,ds\, ,
$$
where
$\rho=(\rho_1,\rho_2),\eta=(\eta_1,\eta_2)\in\RR^2,\rho_1>0,\rho_2>0,\tau\in\RR$,
$\psi\in C^{\infty}_0(\RR)$ is supported in a bounded interval
I containing the origin, and where
\begin{equation}\label{3.1}
g(s)=s^m\Big(b(s)+R(s)\Big), \quad s\in I, m\in\NN, m\ge 2\, ,
\end{equation}
with $b\in C^{\infty}(I,\RR)$ satisfying $b(0)\ne 0$. Moreover,
$R\in C^{\infty}(I,\RR)$ is a smooth perturbation term.

By $\mathcal{M}^{(\rho,\eta,\tau)},$ we denote the associated
maximal operator
$$\mathcal{M}^{(\rho,\eta,\tau)}f(x):=\sup_{t>0}|A_t^{(\rho,\eta,\tau)}f(x)|\, .$$

Then there exist a neighborhood $U$ of the origin in $I$ and
$M\in\NN$, $\delta>0$, such that for $p>m$,
\begin{equation}\label{3.2}
\parallel\mathcal{M}^{(\rho,\eta,\tau)}f\parallel_p\le
C_p\left(\frac{|\eta_1|}{\rho_1}+\frac{|\eta_2-\tau\eta_1/\rho_1|}{\rho_2}+1
\right)^{1/p}|| f||_p\, ,\quad f\in\mathcal{S}(\RR^2)\, ,
\end{equation}
for every $\psi$ supported in $U$ and every $R$ with $||
R||_{C^M}<\delta$, with a constant $C_p$ depending only on
$p$ and the $C^M$-norm of $\psi$ (such constants will be called
``admissible'').

\end{prop}

\noi {\bf Proof.} Consider the linear operator
$$Tf(x_1,x_2)=(\rho_1\rho_2)^{-1/p} f \Big( \rho^{-1}_1x_1,\rho^{-1}_2 (x_2-\frac
\tau{\rho_1}x_1)\Big)\, .$$
Then $T$ is isometric on $L^p(\RR^2)$, and one computes that
$\tilde A_t:=T^{-1}A_tT$ is given by
$$\tilde A_t f(x)=\tilde A_t^{\sigma}f(x)=\int
f\Big(x_1-t(s+\sigma_1),x_2-t(\sigma_2+g(s))\Big)\psi(s)\,ds\, ,$$
where $\sigma =(\sigma_1,\sigma_2)$ is given by
$$\sigma_1=\frac{\eta_1}{\rho_1},\quad\sigma_2=\frac{\eta_2}{\rho_2}-\frac{\tau}{\rho_2}\frac{\eta_1}{\rho_1}\, .$$

Put
$$\tilde{\mathcal{M}}f(x)=\sup_{t>0}|\tilde A_tf(x)|\, .$$
Then \eqref{3.2} is equivalent to the following estimate for
$\tilde{\mathcal{M}}$:
\begin{equation}\label{3.3}
||\tilde{\mathcal{M}}f||_p \quad\le\quad
C_p(|\sigma|+1)^{1/p}|| f||_p,\quad
f\in\mathcal{S}(\RR^2)\, ,
\end{equation}
for every $\sigma\in\RR^2$, where $C_p$ is an admissible
constant.\\

\noi  a) We first consider the case $m=2$.
\medskip

\noi By means of the Fourier inversion formula, we can write
$$\tilde A_t f(x)=\frac 1{(2\pi)^2}\int_{\RR^2} e^{i(x-t\sigma)\cdot\xi}H(t\xi)\hat
f(\xi)d\, \xi\, ,
$$
where
$$H(\xi_1,\xi_2):=\int_{\RR}e^{-i(\xi_1s+\xi_2 g(s))}\psi(s)\, ds\, .
$$
If $s_0$ is a critical point of the phase $\xi_1s+\xi_2g(s)$ of
$H$, then $g'(s_0)=-\xi_1/\xi_2$, where by our assumptions on $g$
we have $|g'(s_0)|\sim |s_0|$.

This shows that we can choose a neighborhood $U$ of $s=0$ in $\RR$
and some $\varepsilon_1 >0$ such that for any $\xi$ with
$|\xi_1/\xi_2|<\varepsilon_1$ the phase function has a unique
non-degenerate critical point
$$s_0(\xi_1/\xi_2)=-\tfrac{\xi_1}{\xi_2}\omega(\tfrac{\xi_1}{\xi_2},R)\in U\, ,$$
where $\omega$ depends smoothly on $\xi_1/\xi_2$ and the error
term $R$, and $\omega(\xi_1/\xi_2,0)\neq 0$. Moreover, if
$|\xi_1/\xi_2|\ge\varepsilon_1$, we may assume that no critical
point belongs to $U$. In the last case, we may integrate by parts
to see that
$$|D^{\alpha}_\xi H(\xi)|\le C_{\alpha,N}(1+|\xi|)^{-N},\quad |\xi_1/\xi_2|\ge\varepsilon_1\,
,$$
for every $\alpha\in\NN^2$ with $|\alpha|\le 3$
 and $N=0,\dots,3$, where the constants $C_{\alpha,N}$ are
 admissible.

 A similar estimate holds obviously true for $|\xi|\le C$. Applying
 then the stationary phase method to the remaining frequency
 region, and combining these estimates, we get:
$$
H(\xi)=e^{iq(\xi)}\frac{\chi\left(\frac{\xi_1}{\xi_2}\right)A(\xi)}{(1+|\xi|)^{1/2}}+
B(\xi)\, ,
$$
where $\chi$ is a smooth function supported on a small
neighborhood of the origin,
$$
q(\xi)=q(\xi,R)
$$
is a smooth function of $\xi$ and $R$ which is homogenous of
degree 1, and   which can be considered as a small perturbation of
$q(\xi,0)$, if $R$ is contained in a sufficiently small
neighborhood of 0 in $C^{\infty}(I,\RR)$. It is also important to
notice that the Hessian $D^2_{\xi}q(\xi,0)$ has rank 1, so that
the same applies to $D^2_{\xi}q(\xi,R)$ for small perturbations
$R$. Moreover, $A$ is a symbol of order zero such that
$$
A(\xi)=0,\quad \mathrm{if} \quad|\xi|\le C\, ,
$$
and 
\begin{equation}\label{3.4}
|\xi^{\alpha}D^{\alpha}_{\xi}A(\xi)|\le C_{\alpha},\quad
\alpha\in\NN^2,|\alpha|\le 3\, ,
\end{equation}
where the $C_{\alpha}$ are admissible constants. Finally, $B$ is a
remainder term satisfying
\begin{equation}\label{3.5}
|D^{\alpha}_{\xi}B(\xi)|\le C_{\alpha,N}(1+|\xi|)^{-N}\, ,\quad
|\alpha|\le 3,\, 0\le N\le 3\, ,
\end{equation}
again with admissible constants $C_{\alpha,N}$.

If we put
$$
\tilde A^0_tf(x):=\frac 1{(2\pi)^2}\int_{\RR^2}e^{i(x-t\sigma)\cdot\xi}
\, B(t\xi)\hat f(\xi)\, d\xi\, ,
$$
then by \eqref{3.5},
$$\tilde A^0_t f(x)=f* k^{\sigma}_t(x)\, ,$$
where
$$
k^{\sigma}_t(x)=t^{-2}k\left(\frac x t\right)
$$
and where $k^{\sigma}$ is the translate
\begin{equation}\label{3.6}
k^{\sigma}(x):=k(x-\sigma)
\end{equation}
of $k$ by the vector $\sigma$ of a fixed function $k$ satisfying
an estimate of the form
\begin{equation}\label{3.7}
|k(x)|\le C(1+|x|)^{-3}\, .
\end{equation}

Let $\tilde{\cal M}^0f(x):={\dst\sup_{t>0}}|\tilde A^0_t (x)|$ denote the corresponding
maximal operator.
\eqref{3.6} and \eqref{3.7} show that
$$
||\tilde{\cal M}^0||_{L^{\infty}\rightarrow L^{\infty}}\le C,
$$
with a constant $C$ which does not depend on $\sigma$.

Moreover, scaling by the factor $(|\sigma|+1)^{-1}$ in direction
of the vector $\sigma,$ we see that
$$
||\tilde{\cal M}^0||_{L^1\rightarrow L^{1,\infty}}\le C\; (|\sigma|+1)\, ,
$$
since we then can compare with $(|\sigma|+1)M$, where $M$ is the
Hardy-Littlewood maximal operator. By interpolation, these
estimates imply that
$$
||\tilde{\cal M}^0||_{L^p\rightarrow L^p}\le C_p(|\sigma|+1)^{1/p}\, ,
$$
if $p>1$.

\medskip

There remains the maximal operator ${\tilde{\cal M}}^1$
corresponding to the family of averaging operators
$${\tilde{\cal A}}^1_t(x):=\frac 1{(2\pi)^2}\int_{\RR^2}e^{i[\xi\cdot x-t(\sigma\cdot\xi
+q(\xi))]}\,\tfrac{\chi(\xi_1/\xi_2)A(t\xi)}{(1+|t\xi|)^{1/2}}\hat f(\xi) \,d\xi
$$
(notice that $q(t\xi)=tq(\xi)$).

As usually we choose a non-negative function $\beta\in
C^{\infty}_0(\RR)$ such that
$$
\mathrm{supp}\beta \subset [1/2,2],\quad\sum^{\infty}_{j=-\infty}\beta(2^{-j}r)=1
\quad\mathrm{for}\ r>0\, ,
$$
and put
$$
A_{j,t}f(x):=\int_{\RR^2}e^{i[\xi\cdot x-t(\sigma\cdot\xi
+q(\xi))]}\, \tfrac{\chi(\xi_1/\xi_2)A(t\xi)}{(1+t|\xi|)^{1/2}}\,\beta(2^{-j}|t\xi|)\hat
f(\xi)\, d\xi\, .
$$
Since we may assume that $A$ vanishes on a sufficiently large
neighborhood of the origin, we have $A_{j,t}f=0,\,\mathrm{if}\,
j\le 0$, so that
\begin{equation}\label{3.8}
\tilde A^1_t f(x)=\sum^{\infty}_{j=1}A_{j,t}f(x)\, .
\end{equation}
Denote by ${\cal M}_j$ the maximal operator associated to the
averages $A_{j,t},\ t>0$.

Since $A_{j,t}$ is localized to frequencies
$|\xi|\sim\frac{2^j}{t}$, we can use Littlewood-Paley theory (see
\cite{stein-book}) to see that
\begin{equation}\label{3.9}
||{\cal M}_j||_{L^p\rightarrow
L^p}\;\lesssim\;||{\cal
M}_{j,\mathrm{loc}}||_{L^p\rightarrow L^p}\, ,
\end{equation}
where
$$
{\cal M}_{j,\mathrm{loc}}f(x):=\sup_{1\le t\le 2}|A_{j,t}f(x)|\, .
$$

Choose a bump function $\rho\in C^{\infty}_0(\RR)$ supported in
$[1/2,4]$ such that $\rho(t)=1$, if $1\le t\le 2$.
 In order to estimate ${\cal M}_{j,\mathrm{loc}}$, we use the
following well-known estimate (see, e.g. \cite{iosevich-curves}, Lemma 1.3)
\begin{eqnarray*}
&&\sup_{t\in\RR}|\rho(t)A_{j,t}f(x)|^p\\
&\le& \quad p\left(\int^{\infty}_{-\infty}\Big|\rho(t)A_{j,t}f(x)\Big|^p
dt\right)^{1/p'}\left(\int^{\infty}_{-\infty}\Big|\frac{\partial}{\partial
t}\Big(\rho(t)A_{j,t}f(x)\Big)\Big|^p dt\, ,\right)^{1/p}
\end{eqnarray*}
which follows by integration by parts. By H\"older's inequality,
this implies
\begin{eqnarray}\label{3.10}
&&||{\cal M}_{j,\mathrm{loc}}f||^p_p \nonumber\\
&\le &
C\left(\int_{\RR^2}\int^{4}_{1/2}\Big|A_{j,t}f(x)\Big|^p dt
dx\right)^{\frac{p-1}{p}}
  \left(\int_{\RR^2}\int^{4}_{1/2}\Big|\frac{\partial}{\partial
t}A_{j,t}f(x)\Big|^p dt dx\right)^{\frac 1 p}\\ \nonumber
 &+&  C\int_{\RR}\int_{1/2}^4|A_{j,t}f(x)|^p dt dx\, . 
\end{eqnarray}
Moreover,
$$\frac{\partial}{\partial t}A_{j,t}f(x)=\int_{\RR^2}e^{i[\xi\cdot
x-t(\sigma\cdot\xi+q(\xi))]} 
\chi \left(\tfrac{\xi_1}{\xi_2}\right) h(t,j,\xi)\,d\xi\, ,$$
where
\begin{eqnarray*}
h(t,j,\xi) & = &
-i\frac{\sigma\cdot\xi+q(\xi)}{(1+t|\xi|)^{1/2}} A(t\xi)\,\beta(2^{-j}t|\xi|)
 +\frac{\partial}{\partial t}\left[\frac{A(t\xi)}
{(1+t|\xi|)^{1/2}}\right]\beta(2^{-j}t|\xi|)\\
 & + & \frac{A(t\xi)}{(1+t|\xi|)^{1/2}}\; 2^{-j}|\xi|\, \beta'(2^{-j}t|\xi|)\, .
\end{eqnarray*}
Now, if $t\sim 1$, since $A$ vanishes near the origin, it is easy
to see that the amplitude of $A_{j,t}$ can be written as
$2^{-j/2}a_{j,t}(\xi)$, where $a_{j,t}$ is a symbol of order 0
localized where $|\xi|\sim 2^j$. Similarly, the amplitude of
$\frac{\partial}{\partial t}A_{j,t}$ can be written as
$2^{j/2}(|\sigma|+1)b_{j,t}$, where $b_{j,t}$ is a symbol of order
0 localized where $|\xi|\sim 2^j$, and $a_{j,t}$, $b_{j,t}$
satisfy estimates of the form
$$
\Big|(1+|\xi|)^{|\alpha|}\Big(|D^{\alpha}a_{j,t}(\xi)|+|D^{\alpha}b_{j,t}
(\xi)|\Big)\Big|\le C_{\alpha}\, ,
$$
with admissible constants $C_{\alpha}$.

We can then apply the local smoothing estimates by Mockenhaupt,
Seeger and Sogge from \cite{mockenhaupt-seeger-sogge}, \cite{mockenhaupt-seeger-sogge2} for operators of the form
$$P_jf(x,t)=\int e^{i(\xi\cdot x-tq(\xi))}a(t,\xi)\beta(2^{-j}|\xi|)\hat f(\xi)\, d\xi\, ,$$
where $a(t,\xi)$ is a symbol of order 0 in $\xi$, and the Hessian
matrix of $q$ has  rank 1 everywhere. Their results imply in particular
that for $2<p<\infty$,
\begin{eqnarray}\label{3.11}
\left(\int^4_{1/2}\int_{\RR^2}|P_jf(x)t|^p dx\,dt\right)^{1/p} \le
C_p \,2^{j\Big(\frac 1 2-\frac 1 p-\delta(p)\Big)} ||
f||_{L^p(\RR^2)}\, ,
\end{eqnarray}
for some $\delta(p)>0$.

Since $2^{j/2}A_{j,t}f(x)$ and
$2^{-j/2}(|\sigma|+1)^{-1}\frac{\partial}{\partial t}A_{j,t}f(x)$
are of the form $P_j f(x-t\sigma)$, for suitable operators $P_j$
of this type, we can apply \eqref{3.10} and \eqref{3.11} to
obtain, if $R=0$,
\begin{eqnarray*}
||{\cal M}_{j,\mathrm{loc}}f||_p  \le  C_p\; 2^{j(\frac 1 2- \frac 1
p-\delta(p))} \cdot  2^{-(j/2(p-1)+j/2)/p}\;(|\sigma|+1)^{1/p}|| f||_p\, ,
\end{eqnarray*}
i.e.,
\begin{equation}\label{3.12}
||{\cal M}_{j,\mathrm{loc}}f||_p\le
C_p\,(|\sigma|+1)^{1/p}\, 2^{-\delta(p)j}|| f||_p\, ,
\end{equation}
if $2<p<\infty$, where $\delta(p)>0$.

However, as observed in \cite{iosevich-curves}, the estimate \eqref{3.11}
remains valid under small, sufficiently smooth perturbations, and
the constant $C_p$ depends only on a finite number of derivatives
of the phase function and the symbol of $P_j$. Therefore, if
$\delta$ is sufficiently small and $||
R||_{C^M}<\delta$, then estimate \eqref{3.12} holds true 
also for $R\neq 0$, with an admissible constant $C_p$.

Summing over all $j\ge 1$ (compare \eqref{3.8}), we thus get
$$||\tilde{\cal M}^1f||_p\;\le C_p\,(|\sigma|+1)^{1/p}\,|| f||_p,\quad\mathrm{if}\ p>2\, ,
$$
with an admissible constant $C_p$.

This finishes the proof of the proposition in the case $m=2$ .\\

\noi b) Let next $m\in\NN,m\ge 2$, be arbitrary. Following \cite{iosevich-curves},
we shall reduce this case to the previous case $m=2$ by means of a
dyadic decomposition and scaling in $s$. Given $K\in\NN$ and
$d>0$, we choose a bump function $\beta\in C^{\infty}_0(\RR)$
supported in $(d/2,2d)$ such that
$$\sum^{\infty}_{k=K}\beta(2^ks)+\sum^{\infty}_{k=K}\beta(-2^ks)=1
\quad\mbox{for every}\ s\in\mathrm{supp}\,\psi\setminus\{0\}$$
(this is possible, if supp $\psi$ is assumed to be sufficiently
small). Accordingly, we decompose the averaging operator $\tilde
A_t=\tilde A^{\sigma}_t$:
\begin{equation}\label{3.13}
\tilde A_tf(x)=\sum^{\infty}_{k=K}\tilde
A^k_tf(x)+\sum^{\infty}_{k=K}\tilde B^k_tf(x)\, ,
\end{equation}
where
$$
\tilde A^k_tf(x):=\int_{\RR}f\Big(x_1-t(s+\sigma_1),x_2-t(\sigma_2+g(s)\Big)
\psi(s)\,\beta(2^ks)\,ds\, ,$$
and where $\tilde B^k_tf(x)$ is defined in the same way, only with
$\beta(2^ks)$ replaced by $\beta(2^{-k}s)$.

Let us consider the maximal operator
$${\tilde{\cal M}}^kf(x):=\sup_{t>0}|\tilde A^k_tf(x)|,\quad k\ge K\, .$$
Changing coordinates $s\mapsto 2^{-k}(d+s)$, we obtain
\begin{eqnarray*}
\tilde
A^k_tf(x)=2^{-k}\int_{\RR}&&f\Big(x_1-t(2^{-k}s+2^{-k}d+\sigma_1),
x_2-t(\sigma_2+ g(2^{-k}(d+s))\Big) \\
&&\cdot\psi(2^{-k}(d+s)) \beta(d+s)\,ds\, ,
\end{eqnarray*}
where $s\mapsto\beta(d+s)$ is supported where $-\frac d 2<s<d$.
 Observe that, by \eqref{3.1},
$$g(2^{-k}(d+s))=2^{-mk}(d+s)^m\, \Big(b(2^{-k}(d+s))+R(2^{-k}(d+s))\Big)\, ,$$
where
$$(d+s)^m=(d^m+md^{m-1}s)+s^2Q(s)\, ,$$
for some polynomial $Q$ with $Q(0)={m\choose 2}d^{m-2}\neq 0$,
whose coefficients are polynomials in $d$. Moreover, by Taylor's
formula
$$b(2^{-k}(d+s))=b(2^{-k}d)+2^{-k}b'(2^{-k}d)s+\tilde b_k(s)s^2\, ,$$
where $b(0)\neq 0$ and $||\tilde b_k||_{C^M}\lesssim
2^{-2k}$, if $K$ is assumed to be sufficiently large. Similarly,
$$R(2^{-k}(d+s))=R(2^{-k}d)+2^{-k}R'(2^{-k}d)s+\tilde R_k(s)s^2\, ,$$
where
$$
|R'(2^{-k}d)|\lesssim\delta\quad\mathrm{and}\quad||\tilde R_k||_{C^M}\lesssim 2^{-2k}\delta\,
,$$
if $|| R||_{C^{M+2}}\le\delta$, say.
\smallskip

We may assume that $2^{-K}\le\delta$. Then we see that
$$g(2^{-k}(d+s))=\alpha +\beta s+2^{-mk}s^2g_k(s)\, ,$$
where
\begin{eqnarray*}
\alpha &=& 2^{-mk}d^m(b+R)(2^{-k}d),\\
\beta  &=& 2^{-mk}md^{m-1}(b+R)(2^{-k}d)+d^m2^{-k}(b+R)'(2^{-k}d)\\
g_k(s) &=& Q(s)b(2^{-k}(d+s))+d^m(\tilde b_k+\tilde
R_k)(s)
 +md^{m-1}2^{-k}\Big(b'(2^{-k}d)+R'(2^{-k}d)\Big).
\end{eqnarray*}
Notice that, since $2^{-k}\le\delta$,
$$g_k(s)=Q(s)b(0)+r_k(s)\, ,$$
where $|| r_k||_{C^M}\lesssim\delta$. This shows
that $g_k(s)$ is a perturbation of the fixed function $b(0)Q(s)$,
and thus we can apply the estimate \eqref{3.2} for the case
$m=2$ to the maximal operator
$${\tilde{\cal M}}_kf(x):=\sup_{t>0}|\tilde A^k_tf(x)|\, ,$$
with
\begin{eqnarray*}
&&\rho_1 = 2^{-k},\quad \eta_1=2^{-k}d+\sigma_1\, ,\\
&&\rho_2 = 2^{-mk},\ \eta_2=\sigma_2+\alpha,\ \tau =\beta\, .
\end{eqnarray*}
I.e., if we fix some sufficiently small $d>0$, then for $p>2,$
\begin{eqnarray*}
||{\tilde{\cal M}}_k||_{L^p\rightarrow L^p}\le
C_p\, 2^{-k}\Big(2^k|\sigma_1|+2^{mk}(|\sigma_2|+|\alpha|)
+2^{mk}|d+2^{k}\sigma_1|\;|\beta|+1\Big)^{1/p}\, .
\end{eqnarray*}
Since clearly $|\alpha|=O(2^{-mk}),\, |\beta|=O(2^{-k})$, this
implies
$$||{\tilde{\cal M}}_k||_{L^p\rightarrow L^p}\le
C_p(|\sigma_1|+|\sigma_2|+1)^{1/p}\,
2^{-k(1-\frac m p)}\, .$$

For $p>m$, we can thus sum over all $k\ge K$ and obtain
\eqref{3.3} (notice that the maximal operators associated to the
$\tilde B^k_t$ can be estimated in the same way by means of the
change of coordinates $s\mapsto -s$).

\qed
\bigskip

Consider now a smooth function $a:I\rightarrow \RR$, where $I$ is
a compact interval of positive length. We say that $a$ is a
function of {\it polynomial type} $m\ge 2\ (m\in\NN)$, if there is
a positive constant $c>0$ such that

\begin{equation}\label{3.14}
c\le\sum^m_{j=2}|a^{(j)}(s)|\quad\mbox{for every}\  s\in I\, ,
\end{equation}
and if $m$ is minimal with this property. Oscillatory integrals
with phase functions $a$ of this type have been studied, e.g., by
J.~E.~Bj\"ork  (see \cite{domar}) and G.~I.~Arhipov \cite{arhipov}, and it is our goal here to estimate related maximal
operators, allowing even for small perturbations of $a$. More
precisely, consider averaging operators
\begin{eqnarray*}
A^{\varepsilon}_tf(x):=\int_{\RR}f\Big(x_1-ts,x_2-t(1+\varepsilon(a(s)+r(s)))\Big)
\psi(s)\,ds,\quad f\in{\cal S}(\RR^2)\, ,
\end{eqnarray*}
along dilates by factors $t>0$ of the curve
$$
\gamma(s):=\Big(s, 1+\varepsilon(a(s)+r(s))\Big), \quad s\in I\, ,
$$
where $\varepsilon>0$, $\psi\in C^{\infty}(I)$ is a smooth,
non-negative density and $r\in C^{\infty}(I)$ will be a
sufficiently small perturbation term. By $\cal M^{\varepsilon}$
we denote the corresponding maximal operator
$${\cal M}^{\varepsilon}f(x):=\sup_{t>0}|A^{\varepsilon}_tf (x)|\, .$$

\begin{thm}\label{s3.2}
Let $a$ be a function of polynomial type $m\ge 2$. Then there
exist numbers $M\in\NN$, $\delta>0$, such that for every $r\in
C^{\infty}(I,\RR)$ with $|| r||_{C^M}<\delta,\, 
0<\varepsilon <<1$ and $p>m$, the following a priori estimate is
satisfied:
\begin{equation}\label{3.15}
||{\cal M}^{\varepsilon}f||_p\le
C_p\, \varepsilon^{-1/p}|| f||_p,\quad f\in{\cal
S}(\RR^2)\, ,
\end{equation}
with a constant $C_p$ depending only on $p$.
\end{thm}

By means of an induction argument (based on an idea of J.~J.~Duistermaat
\cite{duistermaat}), we shall reduce this theorem to Proposition \ref{prop3.1}.

Let us fix a smooth function $a:I\rightarrow\RR$ of polynomial
type $m\ge 2$. We shall proceed by induction on the type $m$.

Observe first that it suffices to find for every fixed $s_0\in I$
a subinterval $I_0\subset I$ which is relatively open in $I$ and contains $s_0$ such that
\eqref{3.15} holds for every $\psi$ supported in $I_0$. For,
then we can cover $I$ be a finite number of such subintervals
$I_j$, decompose $\psi$ by means of a subordinate smooth
partition of unity into $\psi=\sum_j\psi_j$, where
$\psi_j$ is supported in $I_j$, and apply the
estimate \eqref{3.15} for each of the pieces.

So, fix $s_0\in I$. Extending the function $a$ in a suitable way
to a $C^{\infty}$-function beyond the boundary points of $I$, we
may assume that $s_0$ lies in the interior of $I$. Translating by
$s_0,$  we may furthermore assume that $s_0=0$.
Then, by \eqref{3.14}, there is some $k\in\NN$,\ $2\le k\le m$,
such that
\begin{equation}\label{3.16}
a^{(j)}(0)=0\ \mathrm{for} \ 2\le j\le k-1,\ \mathrm{and}\; a^{(k)}(0)\neq
0\, .
\end{equation}
\bigskip

 Assume first that $k=2$. Then we may write
$$a(s)=\alpha_0+\alpha_1 s+s^2b(s)\quad\mathrm{near}\; s=0\, ,$$
where $b\in C^{\infty}(I)$, $b(0)\neq 0$. Consequently, if $r\in
C^{\infty}(I)$ with $|| r||_{C^{M+2}}<\delta$, then,
by Taylor's formula,
\begin{eqnarray*}
a(s)+b(s)=(\alpha_0+r(0))+(\alpha_1+r'(0))s
+s^2(b(s)+R(s))\quad\mathrm{near}\; s=0\, ,
\end{eqnarray*}
where $|| R||_{C^M}\lesssim\delta$. Estimate
\eqref{3.15} thus follows from Proposition \ref{prop3.1}.

\bigskip
Let next $k\ge 2$.

\begin{lemma}\label{lem3.3}
Assume a satisfies \eqref{3.16} with $k\ge 3,$ and let $N\in\NN.$ Then there
is some
$\delta>0$, and for every function $r\in C^{\infty}(I)$ with
$|| r||_{C^{k+N}(I)}<\delta$ a number $\sigma(r)\in
I$ with $|\sigma(r)|\lesssim\delta$, depending smoothly on $r$,
such that
\begin{equation}\label{3.17}
(a+r)^{(k-1)}(\sigma(r))=0\, .
\end{equation}

In particular, if we put $I_r:=-\sigma(r)+I$ and $\mu:=(\mu_0,\ldots ,\mu_{m-2}),$
then
\begin{equation}\label{3.18}
(a+r)(s+\sigma(r))=(b(s)+R(s))s^k+\mu_0+\mu_1s+\cdots
+\mu_{k-2}s^{k-2}\, ,
\end{equation}
where $b\in C^{\infty}(I_r)$ with $b(0)\neq 0,\, R\in
C^{\infty}(I_r)$
 with $|| R||_{C^N}\lesssim\delta$ and
$|\mu|\lesssim\delta.$
\end{lemma}

\noi {\bf Proof.} \eqref{3.17} follows from the implicit function
theorem, applied to the mapping $f:I\times C^{k+N}(I)\rightarrow
\RR$, $f(s,r):=(a+r)^{(k-1)}(s)$, and \eqref{3.18} is then a
consequence of Taylor's formula.

\endproof

The case $k=3$ can now be treated by means of \eqref{3.18} in a
similar way as the case $k=2$ (notice that $I_r$ and $I$ overlap
in a neighborhood $U$ of 0 not depending on $r$, if $\delta$ is
sufficiently small, so that we can again assume that $\psi$ is
supported in a fixed interval contained in $U$).

\bigskip

We may thus from now on assume that $k\ge 4$. Since we have seen
that the cases $m=2$ and $m=3$ of Theorem \ref{s3.2} are true, we may
assume that $m\ge 4$, and, by  induction hypothesis, that the
statement of Theorem \ref{s3.2} is true for all $m'\le m-1$. Then, we
may also assume that $k=m$ in \eqref{3.16}, so that, by Lemma
\ref{lem3.3},
$$(a+r)(s+\sigma(r))={\tilde b}(s)s^m+\mu_2s^2+\cdots +\mu_{m-2}s^{m-2}
$$
on $I_r$, where $m-2\ge 2$ (the affine linear term $\mu_0 +\mu_1s$
can again be omitted by means of a linear change of coordinates).
Here we have set ${\tilde b}=b+R$, where, by Lemma \ref{lem3.3}, $||
R||_{C^M}\lesssim\delta$. 
\medskip

Let us put now $\mu=(\mu_2,\ldots ,\mu_{m-2})$.
The case $\mu=0$ can again be treated by Proposition
\ref{prop3.1}, so assume $\mu\neq 0$.

If we scale in $s$ by a factor $\rho^{1/m},\, \rho>0$, we obtain
\begin{eqnarray*}
(a+r)\Big(\rho^{1/m}s+\sigma(r)\Big)
=\rho\left[{\tilde b}(\rho^{1/m}s)s^m+\frac{\mu_2}{\rho^{\frac{m-2}{m}}}
s^2+\cdots+\frac{\mu_{m-2}}{\rho^{\frac 2 m}}s^{m-2}\right].
\end{eqnarray*}
This suggests to introduce a quasi-norm
$$N(\mu):=\left[\mu_2^{\frac{m}{m-2}\nu}+\cdots +
\mu_{m-2}^{\frac m 2\nu}\right]^{1/\nu}\, ,$$
say with $\nu:=2(m-2)!$\,. For then $N$ is smooth away from the
origin, and if we put $\rho:=N(\mu)$, i.e., if we define
$\xi=(\xi_2,\dots,\xi_{m-2})$ by 
$$\xi_2:=\frac{\mu_2}{N(\mu)^{\frac{m-2}{m}}} \ ,\dots, \ \xi_{m-2}
:=\frac{\mu_{m-2}}{N(\mu)^{\frac 2 m}}\, ,$$
then $N(\xi)=1$, and
\begin{eqnarray*}
g(s)=g(s,\rho,\xi):=\frac{1}{\rho}(a+r)\Big(\rho^{1/m}s+\sigma(r)\Big)
    ={\tilde b}(\rho^{1/m}s)s^m+\xi_2 s^2+\cdots
+\xi_{m-2}s^{m-2}\, .
\end{eqnarray*}

Then, putting $\eta:=\sigma(r)$, we have 
\begin{eqnarray*}
A_tf(x)
=\rho^{1/m}\int_{\RR}f\Big(x_1-t(\rho^{1/m}s+\eta),x_2-t(1+\varepsilon\rho
g(s)\Big)\, \psi(\rho^{1/m}s+\eta)\, ds\, .
\end{eqnarray*}
Recall at this point  that $\eta\rightarrow 0$ and $\rho\rightarrow 0$ as
$\delta\rightarrow 0$. In particular we may consider
$g(s,\rho,\xi)$ as a $C^{\infty}$-perturbation of $g(s,0,\xi),$ where 
$$g(s,0,\xi)={\tilde b}(0)s^m+\xi_2s^2+\cdots+\xi_{m-2}s^{m-2}\, .$$
Denote by $\Sigma$ the unit sphere
$$\Sigma:=\{\xi\in\RR^{m-3}:N(\xi)=1\}$$
with respect to the quasi-norm $N$, and choose $B>0$ so large
that
\begin{equation}\label{3.19}
| g''(s)|\ge c|s|^{m-2}\ \mathrm{whenever}\  |s|\ge B, \, \xi\in
I,\, \rho<\delta\, ,
\end{equation}
where $c>0$. This is possible, since $b(0)\neq 0$, provided
$\delta$ is sufficiently small. We then choose $\chi_0,\chi\in
C^{\infty}_0(\RR)$ such that
$\mathrm{supp}\,\chi\subset(-2B,-B/2)\cup(B/2,2B)$ and
$$1=\chi_0(s)+\sum^{\infty}_{k=1}\chi\left(\frac
s{2^k}\right):=\chi_0(s)+\sum^{\infty}_{k=1}\chi_k(s)\ \mbox{for every}\
s\in\RR .$$
Accordingly, we decompose
$$A^{\varepsilon}_tf=\sum^{\infty}_{k=0}A^{\varepsilon,k}_tf\, ,$$
where
\begin{eqnarray*}
A^{\varepsilon,k}_tf(x)
:=\rho^{1/m}\int_{\RR}f\Big(x_1-t(\rho^{1/m}s+\eta),x_2-t(1+\varepsilon\rho
g(s))\Big )\,\psi(\rho^{1/m}s+\eta\ )\chi_k(s)\, ds.
\end{eqnarray*}
Assume first that $k\ge 1.$ Then this can be re-written as 
\begin{eqnarray*}
&&A^{\varepsilon,k}_tf(x)\\
&=&2^k\rho^{1/m}\int_{\RR}f\Big(x_1-t(\rho^{1/m}2^ks+\eta),x_2-t(1+\varepsilon
\rho 2^{mk}g_k(s))\Big)\, \psi(\rho^{1/m}2^ks+\eta)\,\chi(s)\,ds,
\end{eqnarray*}
where
$$g_k(s)=g_k(s,\rho,\xi):=2^{-mk}g(2^ks,\rho,\xi)\, .$$
And, by \eqref{3.19},
$$|g''_k(s)|\ge c>0\quad \mbox{for every}\
s\in \supp\chi\,,\xi\in\Sigma,\rho<\delta\, .$$

More precisely, since
$$g_k(s)={\tilde b}(\rho^{1/m}2^ks)s^m+\frac{\xi_2}{2^{(m-2)k}}s^2+\cdots+
\frac{\xi_{ m-2}}{2^{2k}}s^{m-2}\, ,
$$
where $|s|\sim B$, and where $\rho^{1/m}2^k\le\delta$, unless
$A^{\varepsilon,k}_t=0$, if we choose $\mathrm{supp\,\psi}$
sufficiently close to 0, we see that $g_k(s)$ is a small
$\delta$-perturbation of $g_k(s,0,\xi)$.

Moreover, covering $\Sigma$ by a finite number of
$\delta$-neighborhoods $\Sigma_j$ of points $\xi^{(j)}\in\Sigma$,
for every $\xi\in\Sigma_j$ we may regard $g_k(s,0,\xi)$ as a
$\delta$-perturbation of $g_k(s,0,\xi^{(j)})$. Thus, for
$\xi\in\Sigma_j$, Proposition \ref{prop3.1} can be applied for
$m=2$, in a similar way as in our discussion of the case $k=2$, in
order to estimate the maximal operator

$$
{\cal M}^{\varepsilon, k}f(x) =
\sup_{t>0}|A^{\varepsilon,k}_tf(x)|
$$
by 
\begin{eqnarray*}
||\mathcal{M}^{\varepsilon,k}f||_p &\le&
C'_p\, 2^k\rho^{1/m}\Big(|\eta|(2^k\rho^{1/m})^{-1}+
(\varepsilon\rho2^{mk})^{-1}+1\Big)^{1/p}||
f||_p\\ &\le& C_p\left[(2^k\rho^{1/m})^{1-\frac 1
p}+\varepsilon^{-1/p}(2^k\rho^{1/m})^{1-\frac m p}\right]||
f||_p\, .
\end{eqnarray*}
Since $\mathcal{M}^{\varepsilon,k}=0$ if
$2^k \rho^{1/m}>\delta$, we then obtain for  $p>m$
$$\sum_{k\ge 1}||{\cal
M}^{\varepsilon,k}f||_p=\sum_{k\ge 1,\, 2^k\rho^{1/m}\le\delta}||{\cal
M}^{\varepsilon,k}f||_p\le C_p\varepsilon^{-1/p}|| f||_p\, .$$
\medskip

There remains the operator ${\cal M}^{\varepsilon,0}$.
 Conjugating $A^{\varepsilon,0}_t$ with the scaling operator
$$T_{\rho}f(x_1,x_2):=\rho^{-1/(mp)}f(\rho^{-1/m}x_1,x_2)\, ,$$
which acts isometrically on $L^p(\RR^2)$, we can reduce our
considerations to the averaging operator
\begin{eqnarray*}
&&T^{-1}_{\rho}A^{\varepsilon,0}_t T_{\rho}f(x)\\
&:=&\rho^{1/m}\int_{\RR}f\Big(x_1-t(s+\rho^{-1/m}\eta),x_2-t(1+\varepsilon\rho
 g(s))\Big) \,\psi(\rho^{1/m}s+\eta)\,\chi_0(s)\,ds\, .
\end{eqnarray*}

Fixing again $\xi^0\in\Sigma$, for $\xi$ in a
$\delta$-neighborhood $\Sigma_0$ of $\xi^0$, we can consider
$g(s,\rho,\xi)$ as a $\delta$-perturbation of the polynomial
function
$$
P(s):=g(s,0,\xi^0)={\tilde b}(0)s^m+\xi^0_2s^2+
\cdots+\xi^0_{m-2}s^{m-2}\, .
$$
Since there is no term $\xi^0_{m-1}s^{m-1}$ in $P(s),$ and since $\xi^0\neq
0$, it follows that for every $s_0$ one has
$$\sum^{m-1}_{j=2}|(\tfrac\partial{\partial s})^{j} g(s_0,0,\xi^0)|\neq
0\, ,$$ for otherwise we had
$$P(s)-P(s_0)-P'(s_0)(s-s_0)={\tilde b}(0)(s-s_0)^m = 
{\tilde b}(0)(s^m-ms_0s^{m-1}+\cdots\,)\, ,
$$
hence $s_0=0$, and so $\xi^0=0$.

We can thus apply our induction hypothesis, and obtain for $p>m-1$
$$|| T^{-1}_{\rho}{\cal M}^{\varepsilon,0}T_{\rho}f||_p
\le C_p\,
\rho^{1/m}\left[\rho^{-1/m}|\eta|+(\varepsilon\rho)^{-1}\right]^{1/p}||
f||_p,$$
hence
$$
||{\cal M}^{\varepsilon,0}f||_p\le C_p\, \varepsilon^{-1/p}\rho^{1/m-1/p}||
f||_p\, ,
$$
first for $\xi\in\Sigma_0$, and then, by covering $\Sigma$ again
by a finite number of $\delta$-neighborhoods of points $\xi_j$, for
every $\xi\in\Sigma$. In particular, for $p>m$ we get the uniform
estimate
$$||{\cal M}^{\varepsilon,0}f||_p\le C_p\,\varepsilon^{-1/p}|| f||_p,$$
which concludes the proof of Theorem \ref{s3.2}.

\endproof

In the next subsection, we shall need a slight generalization of this theorem, namely for averaging operators of the form 
\begin{eqnarray*}
A^{\varepsilon,\si_1}_tf(x):=\int_{\RR}f\Big(x_1-t(s+\si_1),x_2-t(1+\varepsilon(a(s)+r(s)))\Big)
\psi(s)\,ds,\quad f\in{\cal S}(\RR^2)\, ,
\end{eqnarray*}
where $\si_1$ is a second real parameter which can be arbitrarily large. The corresponding maximal operator 
$${\cal M}^{\varepsilon,\si_1}f(x):=\sup_{t>0}|A^{\varepsilon,\si_1}_tf (x)|\, $$
can be estimated exactly is before, if we simply replace the shift term $\eta$ in the proof of Theorem \ref{s3.2} by $\eta+\si_1,$ and one easily obtains

\begin{cor}\label{s3.2c}
Let $a$ be a function of polynomial type $m\ge 2$. Then there
exist numbers $M\in\NN$, $\delta>0$, such that for every $r\in
C^{\infty}(I,\RR)$ with $|| r||_{C^M}<\delta,\, 
0<\varepsilon <<1$ and $p>m$, the following a priori estimate is
satisfied:
$$
||{\cal M}^{\varepsilon,\si_1}f||_p\le
C_p\, (|\si_1|+\varepsilon^{-1})^{1/p}|| f||_p,\quad f\in{\cal
S}(\RR^2)\, ,
$$
with a constant $C_p$ depending only on $p$.
\end{cor}

\bigskip

\subsection{Related results for families of surfaces}\label{surfaces}
By decomposing a given surface in $\RR^3$ by means of a "fan" of
hyperplanes into a family of curves, we can easily derive suitable
estimates for certain families of surfaces from the maximal estimates in
the previous subsection. 

Let $U$ be an open neighborhood of the point $x^0 \in \RR^2,$ and let
$\phi_p\in C^\infty(U,\RR)$ such that 
\begin{equation}\label{3.20}
\pa_2^m\phi_p(x^0_1,x^0_2)\neq 0,
\end{equation}
where $m\ge 2.$ Let 
$$\phi=\phi_p+\phi_r,$$
where $\phi_r\in C^\infty(U,\RR)$ sufficiently small.
 Denote by 
$S_\ve$ the surface  in $\bR^3$ given  by
$S_\ve:=\{(x_1,x_2,1+\ve \phi(x_1,x_2)): (x_1,x_2)\in U\},$  with $\ve>0,$
and
 consider the averaging operators
$$
A_tf(x)=A^\ve_tf(x):=\int_{S_\ve} f(x-ty) \psi(y)\,d\si(y),
$$
where $d\si$ denotes the surface measure and  $\psi\in C^\infty_0(S_\ve)$ is
a non-negative cut-off function.  Define the  associated maximal operator by 
$$
\M^\ve f(x):=\sup_{t>0}|A^\ve_tf(x)|.
$$

\begin{prop}\label{s3.4}
 Assume that $\phi_p$ satisfies \eqref{3.20} and that the neighborhood $U$ of the
point $x^0$  is sufficiently small. Then  there
exist numbers $M\in\NN$, $\delta>0$, such that for every $\phi_r\in
C^{\infty}(U,\RR)$ with $|| \phi_r||_{C^M}<\delta$ and any 
$p>m$  there exists a positive constant
$C_p$ such that for $\varepsilon>0 $ sufficiently small the maximal operator $\M^\ve$
satisfies  the following a priori estimate:
\begin{equation}\label{3.21}
||{\cal M}^{\varepsilon}f||_p\le
C_p\, \varepsilon^{-1/p}|| f||_p,\quad f\in{\cal S}(\RR^3)\, .
\end{equation}
\end{prop}

\noi {\bf Proof.} Let us write the averaging operator $A_t$ in the form
\begin{equation}\label{3.22}
A_tf(y)=\int_{\bR^2} f\Big(y_1-tx_1, y_2-tx_2, y_3-t(1+\ve \phi(x_1,x_2))\Big)\,
\eta(x_1,x_2)\,dx,
\end{equation}
where $\eta\in C^\infty_0(U).$

Choose $\th_0$ such that $\sin(\th_0)+\cos(\th_0)x^0_1=0$
(notice that  we may assume that $\cos(\th_0)> 0$). For
small 
$\theta,$  consider the equation
\begin{equation}\label{3.23}
\sin(\th_0+\th)(1+\ve \phi\x)+\cos(\th_0+\th)x_1=0
\end{equation}
with respect to $x_1$. By the implicit function theorem, the last equation has a
unique smooth  solution 
$x_1(\th, x_2,\ve)$ for $|\th|,$
$|x_2-x^0_2|$ and $\ve$ sufficiently small such that $x_1(0, x^0_2,0)=x^0_1. $
Moreover
$\frac\pa {\pa\th} x_1(0,x^0_2,0)\ne 0$. In the integral \eqref{3.22} we can thus
use the  change of variables
$(\theta,x_2)\mapsto (x_1(\th, x_2,\ve), x_2)$ (assuming $U$ to be sufficiently small) and obtain
\begin{eqnarray}\label{3.24}
&&A_tf(y) \nonumber \\
&=&\int_{\bR^2} f\Big(y_1-tx_1(\th, x_2,\ve), y_2-tx_2, y_3-t(1+\ve \phi(x_1(\th,
x_2,\ve),x_2))\Big)\\ 
&&\hskip10cm\psi(\th,x_2,\ve)\,d\th dx_2, \nonumber
\end{eqnarray}
where $ \psi(\th,x_2,\ve):=\eta(x_1(\th,x_2,\ve),x_2)|J(\th,x_2,\ve)|$ and  
$J(\th,x_2,\ve)$ denotes the Jacobian of this change of coordinates.
Let us write the integral \eqref{3.24} as an iterated integral
$$
A_tf(y)=\int_{-b}^b A_t^\th f(y)d\th,
$$
where $b$ is some positive number and  $A_t^\th $ denotes the following 
averaging operator along a curve:
$$
A_t^\th f(y):=\int_{\bR^2} f\Big(y_1-tx_1(\th, s,\ve), y_2-ts, y_3-t(1+\ve \phi(x_1(\th,
s,\ve),s))\Big)\psi(\th,s,\ve) \,ds.
$$
Now, we define the rotation operator
$$
R^\th f(x):=f(x_1\sin (\th_0+\th)-x_3\cos(\th_0+\th), x_2,
x_1\cos(\th_0+\th)+x_3\sin(\th_0+\th)),
$$
which acts isometrically on every $L^p(\RR^3).$ 
Then we have
$$
R^{-\th}A_t^\th R^\th f(y)=\int_{\bR^2} f\Big(y_1+t\frac{1}{\cos(\th_0+\th)}(1+\ve
\phi(x_1(\th, s,\ve),s),\, y_2-ts,\, y_3))\Big)\psi(\th,s,\ve) \,ds.
$$

Observe  that the last operator "acts" only on the first two variables. Moreover, for
$\ve=0,$ by \eqref{3.23}, we have $x_1(\th, x_2,0)=-\tan(\th_o+\th),$ which is
independent of $x_2.$ This implies that 
$$\frac {d^m}{ds^m}\Big(\phi_p(x_1(0,
s,0),s)\Big)\Big|_{s=x^0_2}=\pa_2^m\phi_p(x^0_1,x^0_2)\neq 0.
$$
Notice also that for $\ve, \delta$ and  $U$ (hence also $\th$) sufficiently small,
$\phi(x_1(\th, s,\ve),s)$ can be regarded as a small perturbation of $\phi_p(x_1(0,
s,0),s).$  Therefore we can apply 
Theorem
\ref{s3.2} (in the first two variables) and  obtain that for
$p>m$
$$
\|\sup_{t>0}|R^{-\th}A_t^\th R^\th f|\|_p\le C_p\varepsilon^{-1/p}|| f||_p,
$$
hence
$$
\|\sup_{t>0}|A_t^\th f|\|_p\le C_p\varepsilon^{-1/p}|| f||_p,
$$
where $C_p$ is independent of $\th$ and $\ve.$ 
Integrating  finally in the $\th$ variable we obtain the required estimate.

\endproof
\medskip

In our later applications of this proposition, we shall also have to deal with  functions $\phi$ which depend  in fact  also on the parameter $\ve$ in such a way that they blow up as $\ve\to 0,$ however, in a particular way. 
More presisely, assume $\tilde \phi=\tilde\phi_p+\tilde\phi_r$  has the same properties as $\phi$ in the proposition, so that in particular \eqref{3.20} is satisfied by $\tilde \phi.$ We assume for simplicity that $\tilde \phi$ is defined on $\bR^2$ and supported in the neighborhood $V$ of the point $x^0.$ 
Let further  $\psi_\ve\in C^\infty(V_1)$ be a smooth function depending on the parameter $\ve$ so that there is some $0\le \delta<1$ such that 
\begin{equation}\label{3.25}
\psi_\ve=O(\ve^{-\delta})\ \mbox{in}\ C^\infty,
\end{equation}
in the sense that $||\psi_\ve||_{C^m(V_1)}=O(\ve^{-\delta})$ for every $m\in\bN,$ 
where $V_1$ denotes the orthogonal projection of the neighborhood $V$ onto the $x_1$-axis. Put then
\begin{equation}\label{3.26}
\phi_\ve\x:= \tilde\phi(x_1,x_2-\psi_\ve(x_1)).
\end{equation}
Notice that then 
\begin{equation}\label{3.27}
|\pa_1^j\pa_2^k\phi_\ve(x)|=O(\ve^{-j\delta}).
\end{equation}

\smallskip
This means that we cannot directly apply Proposition \ref{s3.4} to $\phi_\ve.$ We shall see that nevertheless the proof of this proposition can be extended  to $\phi_\ve.$ To this end, observe first 
that  $|\nabla (\ve \phi_\ve)(x)|\le C\ve^{1-\delta},$ uniformly in $x.$ Therefore, 
again by the implicit function theorem, we can  solve the equation 
$$
\sin(\th_0+\th)(1+\ve \phi_\ve\x)+\cos(\th_0+\th)x_1=0
$$
in  $x_1$ near the point $(x_1^0,x_2^0+\psi_\ve(x_1^0)),$ and obtain a smooth solution 
 $x_1(\th, x_2,\ve)$ for  sufficiently small values of $|\th|,$ 
$|x_2-(x^0_2+\psi_\ve(x_1^0))|$ and $\ve>0,$  satisfying  $x_1(0, x^0_2+\psi_\ve(x_1^0),0)=x^0_1. $ 

Let us also define $x_1^0(\theta)$ as the solution of the equation
$$\sin(\th_0+\th)+\cos(\th_0+\th)x_1^0(\th)=0,
$$
and put  $g(\th, x_2,\ve):=x_1(\th, x_2,\ve)-x_1^0(\th).$ Then $g$ satisfies the equation
$$
\sin(\th_0+\th)\ve \phi_\ve\Big(x_1^0(\th)+g(\th, x_2,\ve),x_2\Big)+\cos(\th_0+\th)g(\th, x_2,\ve)=0.
$$
Implicit differentiation shows that 
$$g_\ve'(x_2)=-\ve \frac{\sin(\th_0+\th)\pa_2 \phi_\ve\Big(x_1^0(\th)+g_\ve(x_2),x_2\Big)}
{\cos(\th_0+\th)+\sin(\th_0+\th)\ve\pa_1 \phi_\ve\Big(x_1^0(\th)+g_\ve(x_2),x_2\Big)},
$$
if we use the short-hand notation $g_\ve(x_2)=g(\th, x_2,\ve).$ By \eqref{3.27}, this implies that $|g_\ve'(x_2)|=O(\ve),$ and similarly $|g_\ve^{(j)}(x_2)|=O(\ve),$  for every $j\ge 1,$ uniformly in $x_2.$ But clearly this estimate is also true for $j=0,$ so that 
\begin{equation}\label{3.28}
g_\ve=O(\ve) \ \mbox{in }\ C^\infty.
\end{equation}
If put 
$$
\Phi_\ve(\th,s):=\phi_\ve\Big(x_1^0(\th)+g_\ve(s),s\Big),$$
then \eqref{3.27}, \eqref{3.28} show that $\Phi_\ve(\th,\cdot)=O(1)$ in $C^\infty.$ 
The averaging operators associated to $\phi_\ve$ will be of the form
\begin{equation}\label{3.29}
A_tf(y):=\int_{\bR^2} f\Big(y_1-tx_1, y_2-tx_2, y_3-t(1+\ve \phi_\ve(x_1,x_2))\Big)\,
\eta(x_1,x_2)\,dx,
\end{equation}
where $\eta\x=\tilde \eta(x_1,x_2-\psi_\ve(x_1)),$ with $\tilde \eta\in C^\infty_0(\bR^2)$ supported in a sufficiently small neighborhood $\tilde U\subset V$ of $x^0.$ The corresponding  operators $R^{-\th}A_t^\th R^\th$ are then  given  by 
$$
R^{-\th}A_t^\th R^\th f(y)=\int_{\bR^2} f\Big(y_1+t\frac{1}{\cos(\th_0+\th)}(1+\ve
\Phi_\ve(\th, s)),\, y_2-ts,\, y_3\Big)a(\th,s,\ve) \,ds,
$$
with
$$a(\th,s,\ve):=\eta(x_1(\th,s,\ve),s)|J(\th,s,\ve)|=\tilde \eta\Big(x_1^0(\th)+g_\ve(s), s-\psi_\ve(x_1^0(\th)+g_\ve(s))\Big)
|J(\th,s,\ve)|.
$$
The subsitution $s\mapsto s+\psi_\ve(x_1^0(\th))$ in the integral thus leads to 
$$
R^{-\th}A_t^\th R^\th f(y)=\int_{\bR^2} f\Big(y_1+t\frac{1}{\cos(\th_0+\th)}(1+\ve
\tilde \Phi_\ve(\th, s)),\, y_2-t(s+\psi_\ve(x_1^0(\th))),\, y_3\Big)\tilde a(\th,s,\ve) \,ds,
$$
with $\tilde \Phi_\ve(\th, s):=\tilde \phi\Big(x_1^0(\th)+\tilde g_\ve(s),s+\psi_\ve(x_1^0(\th))-\psi_\ve(x_1^0(\th)+\tilde g_\ve(s))\Big)$ and 
$$\tilde a(\th,s,\ve):=\tilde \eta\Big(x_1^0(\th)+\tilde g_\ve(s),s+\psi_\ve(x_1^0(\th))-\psi_\ve(x_1^0(\th)+\tilde g_\ve(s))\Big)
|J(\th,s+x_1^0(\th),\ve)|,
$$
where we have set $\tilde g_\ve(s):=g_\ve(s+\psi_\ve(x_1^0(\th))).$ From \eqref{3.25} and \eqref{3.28} it is clear that $\tilde g_\ve=O(\ve) \ \mbox{in }\ C^\infty $ and $\psi_\ve(x_1^0(\th))-\psi_\ve(x_1^0(\th)+\tilde g_\ve(s))=O(\ve^{1-\delta}) \ \mbox{in }\ C^\infty $

Consequently,  $\tilde a$ is supported in $\tilde V_1,$ if $\ve$ and $\th$ are sufficiently small, and  $\tilde a=O(1)$ in $C^\infty.$ In a similar way, we see that 
$$\tilde \Phi_\ve(\th, s)=\tilde \phi(x_1^0(\th),s)+\tilde \phi_r(\th,s,\ve),
$$
where the perturbation  term $\tilde \phi_r(\th,s,\ve)$ can be made small in $C^\infty$ by choosing  $\ve$ and $\th$ sufficiently small.

Notice finally that for $\ve<1,$ 
$$|\psi_\ve(x_1^0(\th))|\lesssim \ve^{-\delta}\le \ve^{-1}.
$$

We can therefore apply the maximal theorem for curves, Corollary \ref{s3.2c}, to each operator $R^{-\th}A_t^\th R^\th$ and obtain 

\begin{cor}\label{s3.5}
Let $V$ be an open neighborhood of the point $x^0 \in \RR^2,$ and let
$\tilde \phi_p\in C^\infty(V,\RR)$ be such that 
$$
\pa_2^m\tilde \phi_p(x^0_1,x^0_2)\neq 0,
$$
where $m\ge 2.$ Let 
$$\tilde \phi:=\tilde \phi_p+\tilde \phi_r,$$
where $\tilde \phi_r\in C^\infty(V,\RR)$ is sufficiently small, and assume that  $\psi_\ve\in C^\infty(V_1)$ satisfies \eqref{3.25} for some $0\le\delta<1.$ Put $\phi_\ve\x:= \tilde\phi(x_1,x_2-\psi_\ve(x_1))$ and $\eta\x=\tilde \eta(x_1,x_2-\psi_\ve(x_1)),$ with $\tilde \eta\in C^\infty_0(\bR^2)$ supported in a sufficiently small neighborhood $\tilde U\subset V$ of $x^0,$ and consider the averaging operators $A_t$ given by \eqref{3.29}, with associated  maximal operator $\M^\ve.$ 

Assume that  the neighborhood $\tilde U$ of the
point $x^0$ is sufficiently small. Then  there
exist numbers $M\in\NN$, $\delta_1>0$, such that for every $\tilde \phi_r\in
C^{\infty}(\tilde U,\RR)$ with $|| \phi_r||_{C^M}<\delta_1$ and any 
$p>m$  there exists a positive constant
$C_p$ such that for $\varepsilon>0 $ sufficiently small the maximal operator $\M^\ve$
satisfies  the following a priori estimate:
\begin{equation}\label{3.30}
||{\cal M}^{\varepsilon}f||_p\le
C_p\, \varepsilon^{-1/p}|| f||_p,\quad f\in{\cal S}(\RR^3)\, .
\end{equation}
\end{cor}

\bigskip

\setcounter{equation}{0}
\section{Auxiliary statements on  the multiplicity of  roots at a 
critical point of a  mixed homogeneous polynomial function}\label{multiplicities}

We refer in this section to the definitions and results in \cite{ikromov-m}.   We begin by recalling the following structural statements on mixed homogenous polynomials.

Let $P\in \bR[x_1,x_2]$ be a mixed  homogeneous
polynomial,  and
assume that $\nabla P(0,0)=0$. Following  \cite{ikm}, we denote by 
$$m(P):=\ord_{S^1} P$$ 
the maximal order of vanishing of $P$ along the unit circle $S^1$ centered at the
origin. 


If $m_1,\dots,m_n$ are positive integers, then we denote by
$(m_1,\dots,m_n)$ their greatest common divisor. 

\begin{prop}\label{s4.1} Let $P$ be a $(\ka_1,\ka_2)$-homogeneous
polynomial of degree one, and assume that $P$ is not of the form
$P(x_1,x_2)=cx_1^{\nu_1} x_2^{\nu_2}.$ Then $\ka_1$ and $\ka_2$ are
uniquely determined by $P,$ and  $\ka_1,\ka_2\in \bQ.$  

Let us assume that $\ka_1\le \ka_2,$ and write 
$$
\ka_1=\frac{q}{m},\, \ka_2=\frac{p}{m},\quad (p,q,m)=1,
$$
so that in particular $p\ge q.$
Then $(p,q)=1,$ and there exist non-negative integers $\al_1,\,\al_2$ and
a
$(1,1)$-homogeneous polynomial $Q$ such that the polynomial $P$ can be written as
\begin{equation}\label{4.1}
P(x_1,x_2)=x_1^{\al_1} x_2^{\al_2}Q(x_1^p,x_2^q).
\end{equation}
More precisely, $P$ can be written in the form
\begin{equation}\label{4.2}
P(x_1,x_2)=c x_1^{\nu_1}x_2^{\nu_2}\prod_{l=1}^M(x_2^q-\la_l
x_1^p)^{n_l},
\end{equation}
with $M\ge 1,$ distinct $\la_l\in \bC\setminus\{0\}$ and multiplicities $n_l\in \bN\setminus\{0\},$ with  $\nu_1,\nu_2\in\bN$ (possibly different from $\al_1,\,\al_2$  in
\eqref{4.1}).
\smallskip

Let us   put  $n:=\sum_{l=1}^M n_l.$ The distance $d(P)$ of $P$ can then be read off from \eqref{4.2} as follows: 

If the principal face of $\N(P)$ is compact, then it lies on the line $\ka_1 t_1+\ka_2 t_2=1,$ and the distance is given by 
\begin{equation}\label{4.3}
d(P)=\frac 1{\ka_1+\ka_2}=\frac{\nu_1 q+\nu_2 p+pqn}{q+p}.
\end{equation}
 Otherwise, we have $d(P)=\max\{\nu_1,\nu_2\}.$ In particular, in any case we have $d(P)=\max\{\nu_1,\nu_2,\frac 1{\ka_1+\ka_2}\}.$
\medskip

\end{prop}

The proposition shows that every zero (or ``root'')  $\x$ of $P$ which does not lie on a coordinate axis is of the form $x_2=\la_l^{1/q} x_1^{p/q}.$  The quantity 
$$
d_h(P)=\frac 1{\ka_1+\ka_2}
$$
will be called the {\it homogeneous distance} of the mixed homogeneous polynomial $P.$ Recall that $(d_h(P),d_h(P))$ is just the point of intersection of the bisectrix with the line $\ka_1t_1+\ka_2t_2=1$ on which the Newton diagram $\N_d(P)$ lies. Moreover,
$$d_h(P)\le d(P).
$$
Notice also that 
$$
m(P)=\max\{\nu_1,\nu_2,\max_{l=1,\dots,M}n_l\}.
$$

In view of the  homogeneity of $P,$  we shall often restrict our considerations to roots lying on the unit circle. For the next result, compare \cite{ikromov-m},  Corollary 2.3 and Corollary 3.4.

\begin{cor}\label{s4.2} 
 Let $P$ be a $(\ka_1,\ka_2)$-homogeneous
polynomial of degree one as in Proposition \ref{s4.1}, and consider the representation \eqref{4.2} of $P.$ We put again $n:=\sum_{l=1}^M n_l.$
\bee
\item[(a)]  If  $\ka_2/\ka_1\notin \bN,$ i.e., if $q\ge 2,$  then $n< d_h(P).$ In particular, every real root $x_2=\la_l^{1/q} x_1^{p/q}$ of $P$ has multiplicity $n_l< d_h(P).$

\item[(b)]  If $\ka_2/\ka_1\in \bN,$ i.e., if $q=1,$  then there exists at most one real root of $P$ on the unit circle $S^1$ of multiplicity greater than $d_h(P).$ More precisely,  if  we put $n_0:=\nu_1, n_{M+1}:=\nu_2$ and choose $l_0\in\{0,\dots,M+1\}$ so that  $n_{l_0}=\max\limits_{l=0,\dots,M+1} n_l,$  then $n_l\le d_h(P)$ for every $l\ne l_0.$

\item[(c)] The height of the Newton polyhedron of $P$ is given by 
$$h(P)=\max\{m(P),d_h(P)\}.
$$
\ee
\end{cor}

In particular, we see that the multiplicity  of every real root of $P$ not lying  on a coordinate  axis is bounded 
by the distance $d(P),$ unless $q=1,$ in which case there can at most be one real root $x_2=\la_{l_0}x_1^p$ with multiplicity exceeding  $d(P).$ If such a root exists, we shall call it the {\it principal root of $P.$} 

The next proposition will allow us  to apply Proposition \ref{s3.4} respectively Corollary  \ref{s3.5} in many situations.

\begin{prop}\label{s4.4}
 Let $P$ be a $(\ka_1,\ka_2)$-homogeneous polynomial of degree one 
 such that $\nabla P(0)=0$ and $\ka_2/\ka_1> 2,$ and assume that $\pa_2^2 P$ does not vanish
identically. If $x^0\in S^1$, then denote by $m_2(x^0)$  the order of vanishing of $\pa_2^2 P$ along $S^1$ in the point $x^0.$  By  $\R$  we shall denote the set of all roots of  $\pa_2^2 P$ on the unit circle  which do not lie on the $x_2$-axis.
\bee
 \item[(a)]   Assume that  $p:=\ka_2/\ka_1\in\bN,$ so that $q=1$ and $p\ge 3,$ and that the set $\R$ 
   is non-empty.  Let  then $x^m\in \R$ be a root of maximal multiplicity $m_2(x^m)\ge 1$ among all roots in $\R.$   Then, for   any  other root  $x^0\neq x^m$ in $\R,$ we have  $m_2(x^0)\le d_h(P)-2.$ 

 In particular, for every point $x\in S^1$ such that $x_1\ne 0$ and $x\ne x^m$ there exists some $j$ with $2\le j\le d_h(P)$ such that $\pa_2^j P(x)\ne 0.$

\medskip
 \item[(b)]   Assume that  $p:=\ka_2/\ka_1\in\bN,$ so that $q=1$ and $p\ge 3,$ and that $P$ vanishes along $S^1$ of order $\nu_2=d(P)$ in the point $e:=(1,0)$ on the $x_1$-axis. Moreover, assume that $d(P)>d_h(P)$ and $d(P)>2.$ Then $m_2(x^0)\le d_h(P)-2$ for every $x^0\in\R$  such that $x^0_2\ne 0.$ 
 
 In particular, for every point $x\in S^1$ which does not lie on a coordinate axis,  there exists some $j$ with $2\le j\le d_h(P)$ such that 
$\pa_2^j P(x)\ne 0.$

\medskip

\item[(c)]  If  $\ka_2/\ka_1\notin\bN,$ then $m_2(x_0)\le d_h(P)-2$ for every root with  $x_1^0\ne 0\ne x_2^0,$ unless the polynomial $P$ is of  the form
\begin{equation}\label{4.5}
P\x=c(x_2^2-\la_1x_1^5)(x_2^2-\la_2x_1^5),
\end{equation}
with $\la_1+\la_2\in\bR\setminus\{0\}$ and $\la_1\la_2\in\bR.$ 

\smallskip
In particular, for every point $x\in S^1$ which does not lie on a coordinate axis,  there exists some $j$ with $2\le j\le d_h(P)$ such that 
$\pa_2^j P(x)\ne 0,$
unless $P$ is of the form \eqref{4.5}
\ee
\end{prop}

\noi {\bf Remark.} {\it  In case (a), if  $m(P)>d(P),$ so that $P$ has a (unique) principal root $x^p\in S^1,$ then $x^m=x^p.$}

\bigskip
\noi {\bf Proof.} By our assumptions, $\pa_2^2
P(x)$ is a $\sigma$- homogeneous polynomial  of degree  one with respect to
the weight
$$\sigma_1:=\frac{\kappa_1}{1-2\kappa_2},\ 
\sigma_2:=\frac{\kappa_2}{1-2\kappa_2}.
$$
 According to the Proposition \ref{s4.1}, we can write the polynomial
$\pa_2^2 P(x)$ in the form
$$
\pa_2^2 P\x=x_1^{\nu_1}x_2^{\nu_2}Q_2(x_1^p,x_2^q),
$$
where $p$ and $q$ are coprime, $Q_2$ is a homogeneous polynomial of degree
$n_2,$ and 
$$\frac pq=\frac {\kappa_2}{\kappa_1}=\frac {\sigma_2}{\sigma_1}\ge 2.
$$
We shall also assume that no power of $x_2^q$ can be factored from $Q_2(x_1^p,x_2^q),$ so that 
we have 
\begin{equation}\label{4.9}
\si_1=\frac{q}{\nu_1q+\nu_2p+pqn_2},\quad \si_2=\frac{p}{\nu_1q+\nu_2p+pqn_2}.
\end{equation}

\medskip
We begin with the case $\ka_2/\ka_1\notin\bN,$ i.e., $q\ge 2.$ Recall that we then assume that $x_1^0\ne 0\ne x_2^0,$  so that in particular $m_2(x^0)\le n_2.$ 
Let us  first consider the case $\nu_1+\nu_2\ge1$.
In this case we show that the assumption $m_2(x^0)> d_h(P)-2$ cannot
hold. For, otherwise we had
$$\frac{1+2\si_2}{\si_1+\si_2}=\frac 1{\ka_1+\ka_2}= d_h(P)<n_2+2,$$
which by \eqref{4.9}, and since $q\ge 2,$ is equivalent to
$$
n_2<\frac{2q-(\nu_1q+\nu_2p)}{(p-1)q-p}.
$$
Since  $\nu_1 q+\nu_2 p\ge (\nu_1+\nu_2)q\ge q,$ we then would get
$$
n_2<\frac{q}{(p-1)q-p}.
$$
And,  straight-forward computations  show
that for any $p\ge 2q,\, q\ge2,$ we have
 $\frac{q}{(p-1)q-p}\le 1,$ so that necessarily $n_2=0,$ i.e., $\pa_2^2
P\x=c x_1^{\nu_1}x_2^{\nu_2},$ which would contradict the existence of a
root away from the coordinates axes. 

Assume next that $\nu_1=\nu_2=0, $  so that the 
assumption $m_2(x^0)> d_h(P)-2$ implies the inequality
\begin{equation}\label{4.10}
1\le n_2<\frac{2q}{(p-1)q-p}.
\end{equation}
Since  $q\ge 2,$ we get  $p\ge 4,$ and then $\frac {3q}{q-1}>p\ge
2q,$ hence $2q<5,$ so that $q=2.$ Then \eqref{4.10} implies $n_2=1$ and
$p=4,$ or $p=5.$  Since we assume that $p$ and $q$ are coprime, the only
possibility that remains is that $q=2, \,p=5, \, n_2=1,$  so that  $\pa_2^2 P$ will
be of the form 
$$
\pa_2^2 P\x= c(x_2^2-ax_1^5).
$$
Integrating the last polynomial twice with respect to $x_2,$ and observing that  $P$ must be $(\frac 1{10},\frac 14)$-homogeneous, we can apply Proposition \ref{s4.1} and obtain \eqref{4.5}.

\smallskip
The remaining  claim in case (c) is now evident. 

\medskip

Consider next the case $q=1.$ Let us put $N:=\nu_2+n_2.$  

\medskip
We first prove (a).  If $x^0$ is a root different from $x^m$ in $\R,$ then  $1\le m_2(x^0)\le m_2(x^m),$ and so we have 
$$2m_2(x^0)\le m_2(x^0)+ m_2(x^m)\le N,
$$
hence in particular $N\ge 2.$ Assume now that $m_2(x^0)>d_h(P)-2.$  Then $$\frac{1+2\si_2}{\si_1+\si_2}=d_h(P)<\frac N2 +2, $$
and in view of \eqref{4.9}, one computes that
$N<2\frac{2-\nu_1}{p-1}.$ Because of $N\ge 2,$ this implies $p<3-\nu_1,$ contradicting our assumption $p\ge 3.$

\medskip

Let us prove the statement of the Remark at this point. So, assume  that $m(P)>d(P),$ so that $P$ has a (unique) principal root $x^p\in S^1$ of multiplicity $m(x^p)=m(P).$ 

If $m(P)\ge 3,$ then $x^p\in \R,$ with multiplicity $m_2(x^p)=m(P)-2>d(P)-2\ge d_h(P)-2,$ so that, by (i), we must have $x^p=x^m,$ and the conclusion in (ii) is obvious. 

Assume finally that $m(P)\le 2.$ Then $d(P)<2,$ hence 
$\frac{1+2\si_2}{\si_1+\si_2}< 2, $ which implies $\nu_1+pN\le 1.$ Consequently, we have $N=0$ and $\nu_1\le 1,$ so that $P$ would be a polynomial of degree at most one, hence  $\pa_2^2 P$ would vanish identically. This shows that this case actually cannot arise.

\medskip What remains  to be proven  is (b). So, assume that $\nu_2=d(P)>2.$ Then $\pa_2^2 P$ vanishes of order $d(P)-2\ge 1$ in the point $e,$ i.e., $m_2(e)=d(P)-2.$  Let   $x^0$   be any  root of $\pa_2^2 P$ with $x^0_1\ne 0\ne x_2^0.$   We want to show that $m_2(x^0)\le d_h(P)-2.$

Assume to the contrary that $m_2(x^0)> d_h(P)-2.$ 

If $m_2(x^0)<m_2(e),$ then 
$$
2m_2(x^0)<m_2(e)+m_2(x^0)\le N,
$$
and we obtain $d_h(P)<\frac N2 +2.$  

If $m_2(x^0)\ge m_2(e),$ then 
$$
2m_2(e)\le m_2(e)+m_2(x^0)\le N,
$$
hence $d(P)\le \frac N2 +2.$ But $d_h(P)<d(P),$ so that  we have again $d_h(P)<\frac N2 +2.$  

As in the proof of (a), this leads to a contradiction. 

\endproof

\setcounter{equation}{0}
\section{Estimation of  the  maximal operator $\M$ when the coordinates are adapted or the height is strictly less than $2$}
\label{adapted}
We now turn to the proof of our main result, Theorem \ref{s1.2}. As observed in the Introduction, we may assume that $S$ is locally the graph $S=\{(x_1,x_2,1+\phi(x_1,x_2)): (x_1,x_2)\in \Om\}$  of a function $1+\phi.$ 
Here and in the subsequent sections, $\phi\in C^\infty(\Om)$ will be a smooth real
valued function of finite type  defined on an open neighborhood $\Om$ of the origin in $\RR^2$ and satisfying 
$$\phi(0,0)=0,\,\nabla\phi(0,0)=0.
$$
 In this section we  shall  consider the easiest cases where the coordinates $x$ are adapted to $\phi, $ or where $h(\phi)<2.$ 

\medskip

Recall that  $A_t, t>0,$  denotes the corresponding family of  averaging operators 
$$
A_tf(y):=\int_{S} f(y-tx) \rho(x) \,d\si(x),
$$
where $d\si$ denotes the surface measure on $S$ and  $\rho\in C^\infty_0(S)$ is
a non-negative cut-off function.  We shall assume that $\rho$ is supported in an
open neighborhood $U$ the point $(0,0,1)$ which will be chosen sufficiently small.  The 
associated maximal operator is given by 
\begin{equation}\label{5.1}
\M f(y):=\sup_{t>0}|A_tf(y)|, \quad ( y\in \RR^3).
\end{equation}
The averaging operator $A_t$ can be re-written in the form
$$
A_tf(y):=\int_{\bR^2} f\Big(y_1-tx_1,y_2-tx_2, y_3-t(1+\phi(x_1,x_2))\Big)\eta\x \, dx,
$$
where $\eta$ is a smooth function supported in $\Om.$ If $\chi$ is any integrable function defined on $\Om,$ we shall denote by $A^\chi_t$ the correspondingly localized averaging operator 
$$
A^\chi_tf(y):=\int_{\bR^2} f\Big(y_1-tx_1,y_2-tx_2, y_3-t(1+\phi(x_1,x_2))\Big)\chi(x) \eta(x)\, dx,
$$
and by $\M^\chi$ the associated maximal operator
$$\M^\chi f(y):=\sup_{t>0}|A^\chi_tf(y)|, \quad ( y\in \RR^3).
$$

\medskip

\begin{prop}\label{s5.1} 
Let $\phi$ be as above, and assume that $\ka=(\ka_1,\ka_2)$  is  a given weight such that $0<\ka_1\le \ka_2<1.$ As in \eqref{princ1}, we decompose 
$$\phi=\phi_\ka+\phi_{r}
$$
into its $\ka$-principal part $\phi_\ka$ and the remainder term $\phi_{r}$ consisting of terms of $\ka$-degree $>1.$ 
 Then, if   the neighborhood $\Om$ of the point $(0,0)$ is chosen sufficiently small,   the maximal operator
$\M$ is bounded on $L^p(\RR^3)$   for every  $p>\max\{2,h(\phi_\ka)\}.$
\end{prop}

\noi{\bf Proof.} Let us modify our notation slightly and write points  in $\RR^3$ in the
form  $(x,x_3),$ with $x\in\RR^2$ and $x_3\in\RR.$  
Recall from Corollary \ref{s4.2} the crucial fact  that 
$$h(\phi_\ka)=\max\{m(\phi_\ka),d_h(\phi_\ka)\}.
$$
In particular, the  multiplicity of every real root of  the $\ka$-homogeneous polynomial  $\phi_\ka$ is bounded by 
$h(\phi_\ka).$ 

Consider then the  dilations 
$\de_{r}\x:=(r^{\ka_1}x_1,r^{\ka_2}x_2), \ r>0.$  We choose  a
smooth non-negative function 
$\chi$ supported in the annulus $D:=\{1\le|x|\le R\}$ satisfying
$$
\sum_{k=k_0}^\infty \chi_k(x)=1 \quad \mbox{for}\quad 0\ne x\in \Om,
$$
where $\chi_k(x):=\chi(\de_{2^k}x).$ Notice that by choosing $\Om$ small, we can choose $k_0\in\NN$ large. 
Assuming that $\Om$ is sufficiently small, we can write $A_t$ 
as a sum of averaging operators
$$
A_tf(y,y_3)=\sum_{k=0}^\infty A_t^kf(y,y_3),
$$
where $A_t^k:=A^{\chi_k}_t.$ If we apply  the  change of variables  $x\mapsto\de_{2^{-k}}(x)$ in the
integral above, we obtain
$$
A^k_t f(y,y_3)=2^{-k|\ka|} \int_{\bR^2} f\Big(y-t\de_{2^{-k}}(x),
y_3-t(1+2^{-k}\phi^k(x)\Big)\,\eta(\de_{2^{-k}}(x))\chi(x)\,dx,
$$
where 
$$\phi^k(x):=\phi_\ka(x)+2^k\phi_r(\de_{2^{-k}}(x))$$ 
and where the perturbation  term $2^k\phi_r(\de_{2^{-k}}(\cdot))$ is of order
$O(2^{-\ve k})$  for some $\ve>0$ in any $C^M$-norm. To express this fact, we shall in the sequel again  use the short-hand notation
$$2^k\phi_r(\de_{2^{-k}}(\cdot))=O(2^{-\ve k}).
$$
By $\M^k$ we shall denote the
maximal operator $\M^{\chi_k}$ associated to the averaging operators $A^k_t.$

Assume now that  $p>\max\{2,h(\phi_\ka)\}.$  We define the scaling operator $T^k$  by
$$
T^kf(y,y_3):=2^{\frac{k|\ka|}{p}}f(\de_{2^k}(y),y_3).
$$
Then $T^k$ acts isometrically on  $L^p(\bR^3),$  and 
$$
(T^{-k}A^k_tT^k)f(y,y_3)=2^{-k|\ka|} \int_{\bR^2}
f\Big(y-tx,y_3-t(1+2^{-k}\phi^k(x))\Big)\,\eta(\de_{2^{-k}}(x))\chi(x)\,dx.
$$
Assuming that $\Om$ is a sufficiently small neighborhood of the origin, 
we need to consider only the case when $k$ is sufficiently large.
\medskip

Let $x^0\in D$ be a fixed point. 
\medskip

If $\nabla \phi_\ka(x^0)\neq0,$ then from by Euler's homogeneity relation one easily derives that $\rank(D^2\phi_\ka(x^0))\ge 1$ (see \cite{ikm}, Lemma 3.3).  Therefore, we can find a unit vector $e\in\bR^2$ such that $\pa_e^2\phi_{\ka}(x^0)\ne
0,$ where $\pa_e$ denotes the  partial derivative in direction of $e.$ 

If $\nabla \phi_\ka(x^0)=0,$ then by  Euler's homogeneity relation we have $\phi_\ka(x^0)=0$ as
well. Thus the function $\phi_\ka$ vanishes in $x^0$ at least of order two, so that 
$m(\phi_\ka)\ge2,$ hence $h(\phi_\ka)\ge2$ . On the other hand,  by what we remarked  earlier, it vanishes along the circle passing through $x^0$ and centered  at the origin at most of order $h(\phi_\ka).$ Therefore, we can find a unit vector $e\in\bR^2$ such that $\pa_e^m\phi_\ka(x^0)\ne
0,$ for some  $m$ with  $2\le m\le h(\phi_\ka).$ 

Thus, in both cases, after rotating coordinates so that $e=(0,1),$ 
we may  apply Proposition \ref{3.5} to conclude that for $p>\max\{2,h(\phi_\ka)\}$   and 
sufficiently large $k,$ 
$$
\|T^{-k}\M^kT^kf\|_p\le C 2^{k(\frac{1}{p}-|\ka|)}\|f\|_p,\quad f\in{\cal
S}(\RR^3)\, ,
$$
if we replace $\chi$ in the definition of $A^k_t$ by $\chi\eta,$ where $\eta$ is a
bump function supported in a sufficiently small neighborhood of $x^0.$
This is equivalent to 
$$
\|\M^kf\|_p\le C2^{k(\frac{1}{p}-|\ka|)}\|f\|_p.
$$
Decomposing $\chi$ and correspondingly $A^k_t$ by means of a suitable partition of
unity into a finite number of such pieces, we see that the same estimate holds
for the original operators $\M^k.$ 

 Since $\frac 1{|\ka|}=d_h(\phi_\ka)\le h(\phi_\ka)<p,$ we can sum over all $k\ge k_0$ and
obtain the desired estimate for $\M.$

\endproof

Let us apply this result first  to the  case where the coordinates $x$ are adapted to $\phi, $ possibly after a rotation of the coordinate system $\x.$ 
Observe first that a linear change of the coordinates $\x$ induces a corresponding linear change of coordinates in $\RR^3$ which fixes the coordinate $x_3.$ This linear transformation  is an automorphism of $\RR^3,$ so that it preserves the convolution product on $\RR^3$ (up to a fixed factor), hence the norm of the maximal operator $\M.$ We may thus assume that the coordinates are adapted to $\phi.$

We shall also assume that non-negative numbers $\ka_1, \ka_2$ with $|\ka|:=\ka_1+\ka_2>0$ are chosen so that the principal face $\pi(\phi)$ of the Newton polyhedron $\N(\phi)$ of $\phi$ lies on the line $\ka_1t_1+\ka_2t_2=1.$ Notice that the weight $\ka:=(\ka_1,\ka_2)$ is then determined uniquely, unless $\pi(\phi)$ is a single point. Without loss of generality, as in \cite{ikromov-m} we shall assume $\ka_2\ge \ka_1.$ 

Recall from Corollaries  4.3 and 5.2 in \cite{ikromov-m}, that the coordinates $x$ are adapted to $\phi$ if and only if   the principal face $\pi(\phi)$ of the Newton polyhedron $\N(\phi)$ satisfies one of the following conditions: 
\medskip
\bee
\item[(a)] $\pi(\phi)$ is a compact edge, and either
$ \frac {\ka_2}{\ka_1} \notin\bN,$
or  $\frac {\ka_2}{\ka_1}\in \bN$  and $m(\phi_p)\le d(\phi).$
\item[(b)] $\pi(\phi)$ consists of  a vertex. 
\item[(c)] $\pi(\phi)$ is unbounded. 
\ee
Moreover, in this case we have $h(\phi)=h(\phi_p)=d(\phi_p).$

In the sequel, we shall often refer to these cases as the {\it cases (a) to (c)} without further mentioning.

\begin{cor}\label{s5.2}
Let $\phi$ be as above, and assume that, possibly after a rotation of the coordinate system,  the coordinates $x$ are adapted to $\phi,$ i.e., that $h(\phi)=d(\phi).$  Then, if   the 
neighborhood $\Om$ of the point $(0,0)$ is chosen sufficiently small,   the maximal operator
$\M$ is bounded on $L^p(\RR^3)$   for every  $p>\max\{2,h(\phi)\}.$

\end{cor}
\proof As mentioned before, we may assume that the coordinates are adapted to $\phi.$ 
We begin with case (a), in which $\phi_p=\phi_\ka.$  In particular, we have $h(\phi)=h(\phi_\ka).$ Observe also that   
$\ka_j<1$ for $j=1,2,$  since $\nabla\phi(0)=0,$ so that $0<\ka_1\le \ka_2<1.$ The result is thus an  immediate consequence of Proposition \ref{s5.1}.

\medskip

Consider next the case (b).  If $\pi(\phi)$ consists of a  vertex  $(N,N),$ then $ h(\phi)=N\ge 1.$  Moreover, by perturbing $\ka$ slightly, we may assume  that the line $\ka_1t_1+\ka_2t_2=1$ intersects 
$\N(\phi)$ only in the point $(N,N)$ and that $0<\ka_1\le \ka_2<1,$ so that again $\phi_p=\phi_\ka.$ In particular, $h(\phi)=h(\phi_\ka),$ and  we can now argue exactly as in the case (a).

\medskip
There remains the case (c). Here, the principal face $\pi(\phi)$ is a horizontal half-line, with left endpoint $(\nu_1,N),$ where $\nu_1<N=h(\phi).$ Notice that $N\ge 2,$ since for $N=1$ we had $\nu_1=0,$ which is not possibly given our assumption $\nabla\phi(0,0)=0.$ We can then choose $\tilde\ka$ with $0<\tilde\ka_1<\tilde\ka_2$ so that  the line  $\tilde\ka_1t_1+\tilde\ka_2t_2=1$ is a supporting line to the Newton polyhedron of $\phi$ and that the point $(\nu_1,N)$ is the only point of $\N(\phi)$ on this line (we just have to choose $\tilde\ka_2/\tilde\ka_1$ sufficiently large!). Then necessarily $\tilde\ka_2<1,$ and  the $\tilde\ka$-principal part $\phi_{\tilde\ka}$ of $\phi$ is of the form $\phi_{\tilde\ka}(x)=cx_1^{\nu_1}x_2^N, $ with $c\ne 0.$ Since the coordinates are clearly also adapted to $\phi_{\tilde\ka},$ we find that  $h(\phi)=N=d(\phi_{\tilde\ka})=h(\phi_{\tilde\ka}).$ 

 The result thus follows again from Proposition \ref{s5.1}, with  $\ka$ replaced by $\tilde\ka.$ 
 
\endproof

\begin{remark}\label{s5.3} One can easily extend Corollary \ref{s5.2} as follows:

\medskip

\noi If  the neighborhood $\Om$ of the point $(0,0)$ is chosen sufficiently small,  then  the maximal operator
$\M$ is bounded on $L^p(\RR^3)$   for every  $p>\max\{2,h(\phi_p)\},$ no matter if the coordinates are adapted to $\phi$ or not.
\end{remark}

\noi {\bf Proof.} Indeed, if the coordinates are adapted, then $h(\phi_p)=h(\phi).$ 
So, assume that the coordinates $\x$ are not adapted to $\phi.$ Then the principal face of the Newton polyhedron is a compact edge, so that the principal part $\phi_p$ of $\phi$ is $\ka$-homogeneous, where $\ka$ satisfies the assumptions of Proposition \ref{s5.1}. Since $\phi_p=\phi_\ka,$ the result then follows this proposition.

\qed

\medskip
The result above  holds even when the coordinates are not adapted, but it will then in general not be sharp, since we have $h(\phi_p)\ge h(\phi)$ (see \cite{ikromov-m}, Corollary 4.3), and in general strict inequality holds. 

For example, let
$\phi\x:=(x_2-x_1^2)^2+x_1^5$. Then we have $\phi_p(x)=(x_2-x_1^2)^2$. The coordinate system is not adapted to $\phi,$ because 
$d(\phi)= 4/3<2,$ where $2$ is the multiplicity of the root of $\phi_p$. A coordinate system which is  adapted to $\phi$  is given by 
$y_2:=x_2-x_1^2$ and $y_1:=x_1$. It is then easy to
see that
$h(\phi_p)=2>\frac {10} 7=h(\phi)$.

\bigskip

\begin{cor}\label{s5.4}

If $h(\phi)<2,$ and if  the neighborhood $\Om$ of the point $(0,0)$ is chosen sufficiently small,  then  the maximal operator $\M$ is bounded on $L^p(\RR^3)$   for any $p>2,$ also when the coordinates are not adapted to $\phi.$
\end{cor}

\begin{proof}
If $D^2\phi(0,0)\neq0$ then we have at least one non-vanishing
principal curvature at the origin, so that the result follows from  C.D. Sogge's main theorem in
 \cite{sogge}.
 
\medskip
Next, we consider the case where $D^2\phi(0,0)=0.$ Then necessarily we
have $D^3\phi(0,0)\neq0$, for otherwise $h(\phi)\ge d(\phi)\ge2.$ In particular, $h(\phi)>1.$
Denote by $P_3$ the Taylor polynomial  of degree $3$ with base point $0$ of  the function $\phi.$ If $h(P_3)\le2,$ then we obtain the desired estimate from  Corollary \ref{s5.2}, which $\ka:=(\frac 13,\frac 13).$ Assume therefore  that $h(P_3)>2.$ Then, by Corollary \ref{s4.2} (c), $P_3$ must have  a root of order $3.$ Thus, possibly after rotating the coordinate system, we may assume  that
$P_3\x=cx_2^3$ with $c\neq0.$

Now, we consider the Taylor support  $\T(\phi)$ of $\phi.$ Since $\T(\phi)\subset \{\frac{t_1}{3}+\frac{t_2}{3}\ge 1\},$ one checks easily that the subset 
$$
\{\frac{t_1}{6}+\frac{t_2}{3}<1\}\cap\T(\phi)
$$
of $\T(\phi)$ 
contains at most $3$ points, namely  $(4,0),\,(5,0), \,(3,1),$ all of them lying below the bisectrix $t_1=t_2.$

Moreover, any  line passing through the point $(0,3)\in\T(\phi)$ corresponding to $P_3=cx_2^3$ contains at most one
of these points. Thus, if
$$
\{\frac{t_1}{6}+\frac{t_2}{3}<1\}\cap\T(\phi)\neq\emptyset,
$$
then the principal part $\phi_p$ of $\phi$ contains only two
monomials, one corresponding  to the point  $(0,3)$ above the bisectrix and
 the other  one corresponding  to one of the points listed  above which lie  below the bisectrix, i.e., $\phi_p$ is of the form $dx_1^4+cx_2^3, \ dx_1^5+cx_2^3$ or $dx_1^3x_2+cx_2^3,$ with $d\ne 0$ (note that these all satisfy $d(\phi_p)<2$).
Therefore on the unit circle it has no root of
multiplicity bigger than one, so that  the coordinate system is adapted
to $\phi,$ and thus  $h(\phi)<2.$ The desired estimate for $\M$ follows therefore in this case  from 
Proposition \ref{s5.1}.

\medskip
Assume finally that 
$$
\{\frac{t_1}{6}+\frac{t_2}{3}<1\}\cap\T(\phi)=\emptyset.
$$
Then $\T(\phi)\subset\{\frac{t_1}{6}+\frac{t_2}{3}\ge1\},$ hence   $h(\phi)\ge d(\phi)\ge 2,$ which contradicts to our
assumption.

\end{proof}

In view of this result, we shall from now on always assume that
$$h(\phi)\ge 2.$$

\medskip
\setcounter{equation}{0}
\section{Estimation of the maximal operator $\M$ away from the principal root jet }\label{away-root}

Let  $\phi\in C^\infty(\Om)$ be as in Section 5, and assume now that the coordinates $x$ are not adapted  to $\phi.$ Recall from   \cite{varchenko} in the analytic case (under some non-degeneracy condition),  and from  \cite{ikromov-m} in the general case  that in this situation there exists a smooth function   $\sigma$ which defines an adapted  coordinate system
\begin{equation}\label{6.1}
z_1:= x_1, \ z_2:= x_2-\si(x_1).
\end{equation} 
 for  the function $\phi.$  In these coordinates, $\phi$ is given by 
$$ \phi^a(z):=\phi(z_1,z_2+\si(z_1)).$$
Consider the Taylor expansion 
\begin{equation}\label{6.2}
\si(x_1)=\sum_{l=1}^Kb_lx_1^{m_l} \quad (1\le K\le\infty)
\end{equation}
of $\si,$ where we assume that $b_l\in\RR\setminus\{0\}$ for every $l,$ and where the $m_l\in\NN$ form a strictly increasing sequence
$1\le m_1<m_2<\cdots.$

 Such a function $\si$ can be constructed by means of Varchenko's algorithm \cite{varchenko} (see also \cite{ikromov-m}), and  if $\phi$ is real-analytic, one obtains it  in an explicit way from the Puiseux series expansion of the roots of $\phi$ as the principal root jet (see \cite{ikromov-m}). In the sequel we shall indeed assume that $\si$ is constructed by this algorithm. In particular, if this algorithm stops after finitely many steps, then $K$ coincides with this finite number of steps. This happens in particular if the principal face of the Newton polyhedron of $\phi^{a}(z)$ is compact. 

\smallskip

The goal of this section is  to prove that the main contribution to the maximal operator will be given by a small neighborhood  a modified, polynomial   curve $x_2=\psi(x_1),$ 
of the form 
$$
|x_2-\psi(x_1)|\le \ve_0 x_1^{a},
$$
where $\psi$ will be a suitable polynomial approximation to $\si$ of sufficiently high degree, and where $a\ge \deg \psi.$ 

We shall then often work in the coordinates $y$ given by 
\begin{equation}\label{6.3}
y_1:= x_1, \ y_2:= x_2-\psi(x_1)
\end{equation} 
 for  the function $\phi.$  In these coordinates, $\phi$ is given by 
$$\tilde \phi(y):=\phi(y_1,y_2+\psi(y_1)).$$

 \medskip
 
 To this end, we shall decompose $\Om$ into various regions adapted to the roots of $\phi$ and
estimate the contributions of these regions to the maximal operator $\M$ separately.

 \medskip
We first make the simple observation that in case that $m_1=1,$  the linear change of coordinates 
$$y_1:= x_1, \ y_2:= x_2-b_1x_1^{m_1}
$$
allows  to reduce  to the case $m_1\ge 2.$ Observe to this end that the corresponding  linear change of coordinates of $\RR^3,$ with $y_3:=x_3,$ is an automorphism of $\RR^3,$ so that it preserves the convolution product on $\RR^3$ (up to a fixed factor).
 
 \medskip
 In the sequel, we shall therefore always assume that 
\begin{equation}\label{6.4}
2\le m_1<m_2<\cdots.
\end{equation}
We shall also only consider the region where $x_1>0$ in order to simplify the notation. The remaining half-plane can be treated in the same way. 
\bigskip
 
In order to construct the polynomial $\psi,$  notice that one of the cases (a) - (c) described after  the proof of Proposition \ref{s5.1} applies to $\pad$ (in place of $\phi$), since the coordinates $z$ are adapted to $\pad.$ We shall construct $\psi$ and at the same time a weight $\tilde\ka$ satisfying
\begin{equation}\label{6.9}
1/|\tilde\ka|\le h(\phi).
\end{equation}

\medskip

Let us begin with case (a), in which the principal face $\pi(\pad)$ is a compact edge, and the principal part $\pad_p$ is, say, $\tilde\ka=(\tilde \ka_1,\tilde \ka_2)$-homogenous of degree one.  Observe that, by Varchenko's algorithm,  $K<\infty$ and $a_p:=\frac {\tilde \ka_2}{ \tilde \ka_1} >m_K\ge m_1\ge 2.$ In this case, we shall put $\psi(x_1):=\si(x_1)+c_px_1^{a_p},$ where the constant $c_p$ will be chosen as follows: 

\medskip
If $a_p\notin \NN,$ then we put  $c_p:=0.$ And, if $a_p\in\NN,$ then, according to Proposition \ref{s4.4}, there exists a unique real constant $c_p$ such that $c_p z_1^{a_p}$ is a real root of maximal multiplicity of the $\tilde\kappa$-homogeneous polynomial $\pa_2^2\pad_p(z).$

Observe that the $\tilde\kappa$-homogeneous change of coordinates $y_1:=z_1, y_2:=z_2-c_p z_1^{a_p}$ will again lead to adapted coordinates, and has the effect of modifying the coefficients of the roots of $\pa_2^2\pad_p(z)$ in such a way that the root of maximal multiplicity will be given by $y_2=0.$ We shall  therefore define $\psi$  in case (a) by 
\begin{equation}\label{6.5}
\psi(x_1):=\sum_{l=1}^Kb_lx_1^{m_l} +c_px_1^{a_p}.
\end{equation}
Notice that the principal face $\pi(\tilde\phi)$ is a compact edge in this case, and $h(\phi)=1/|\tilde\ka|.$
\medskip

In case (b), the principal face $\pi(\pad)$ is a vertex, say $(N,N).$ Then $N=h(\phi)\ge 2,$ and $\pad_p=cz_1^Nz_2^N.$ In this case, we choose for $\tilde\kappa$ any rational pair $0<\tilde \ka_1<\tilde \ka_2$ such that the line $\tilde \ka_1t_1+\tilde \ka_2t_2=1$ is a supporting line to the Newton polyhedron $\N(\pad)$ of $\pad$ containing only the point $(N,N)$ from $\N(\pad).$ Then clearly $\phi^a_{\tilde\ka}=\phi^a_p$ and 
$h(\phi)=\frac 1{|\tilde \kappa|}.$ Moreover, again $K<\infty,$ and we define in this case
\begin{equation}\label{6.6}
\psi(x_1):=\sum_{l=1}^Kb_lx_1^{m_l}=\si(x_1).
\end{equation}
Notice that here $\tilde \phi=\phi^a,$ and that the principal face $\pi(\tilde\phi)$ is a vertex.

\medskip
Consider finally case (c), in which the principal face $\pi(\pad)$ is unbounded and  possibly $K=\infty.$ As Varchenko's algorithm shows, then  $\pi(\pad)$ is in fact a horizontal half-line, with left endpoint
 $(\nu_1,N),$ where $\nu_1< N=h(\phi).$ In this case, we shall put
 \begin{equation}\label{6.7}
\psi(x_1):=\sum_{l=1}^Lb_lx_1^{m_l},
\end{equation}
where $L:=K,$  if $K<\infty,$  and where otherwise $L$ will be chosen sufficiently large. Indeed, if $K=\infty,$ then the algorithm shows that there is a finite number of steps $L_0$ such that for every $L\ge L_0,$ the principal part $\tp_p$ of $\phi,$ when expressed in the coordinates $y$ given by \eqref{6.3}, is of the form 
$$
\tp_p(y)=c_Ly_1^{\nu_1}(y_2-b_{L+1} y_1^{m_{L+1}})^N.
$$
The polynomial $\tp_p$ is $\tilde \kappa:=(\frac 1{\nu_1+m_{L+1}N},\frac{m_{L+1}}{\nu_1+m_{L+1}N})$-homogenous of degree one, where 
$$
1/|\tilde\ka|\le N=h(\phi).
$$
Finally, if $K<\infty,$ we shall choose $\tilde \ka$ in the same way as in case (b) (for instance, we could choose it as for the case $K=\infty,$ where we choose any sufficiently large number $m_{L+1}$). Notice that in this case $\tilde\phi_{\tilde\ka}$ is of the form $cy_1^{\nu_1} y_2^N,$ and it may not coincide with the principal part $\tilde\phi_p.$

\medskip

In all three cases (a)-(c), we shall put 
\begin{equation}\label{6.10}
a:=\frac {\tilde \ka_2}{ \tilde \ka_1}>m_1.
\end{equation}

Observe that then $a>\deg\psi,$ except for the case (a), when $a=a_p\in\NN$ and $c_p\ne 0,$ where $a=\deg\psi.$   Moreover, in case (a) we have $\tp_{\tilde\ka}=\tp_p,$  whereas in the cases (b) and (c) $\tilde\phi_{\tilde\ka}$ is of the form
\begin{equation}\label{6.8}
\tp_{\tilde\ka}(y)=cy_1^{\nu_1}(y_2-b y_1^{a})^N,
\end{equation}
with $b\in\RR$ and $N$ as before. Finally, clearly \eqref{6.9} holds true in all three cases (a)-(c).
\bigskip

We  next fix a cut-off function $\rho\in C_0^\infty (\RR)$ supported in a  neighborhood of the origin such that $\rho=1$ near the origin, and put 
$$
\rho_0\x:=\rho\Big(\frac{ x_2-\psi (x_1)}{\ve_0x_1^{a}}\Big).
$$
Define averaging operators
$$
A^{1-\rho_0}_tf(y):=\int_{\bR^2} f\Big(y_1-tx_1,y_2-tx_2, y_3-t(1+\phi(x_1,x_2))\Big)\Big(1-\rho(\frac{ x_2-\psi(x_1)}{\ve_0x_1^{a}})\Big)\, \eta(x)\, dx,
$$
and consider  the associated maximal operator $\M^{1-\rho_0}.$  We shall then prove 

 \begin{prop}\label{s6.1} 
If   the neighborhood $\Om$ of the point $(0,0)$ is chosen sufficiently small,  then the maximal operator
$\M^{1-\rho_0}$   is bounded on $L^p(\RR^3)$   for every  $p>h(\phi).$
\end{prop}

\medskip
\subsection{Preliminary reduction to a $\ka$-homogeneous neighborhood of the  principal root $x_2=b_1x_1^{m_1}$ of $\phi_p$}\label{prelim}

Recall that since the coordinates $x$ are not adapted to $\phi,$ the principal face $\pi(\phi)$ must be a compact edge of the Newton polyhedron of $\phi,$ so that it lies on a unique line $\ka_1t_1+\ka_2t_2=1.$ Again, we may assume that $\ka_2\ge \ka_1.$ Then, by the results in \cite{ikromov-m}, 
$$\frac {\ka_2}{ \ka_1}=m_1\ge2.
$$
Moreover, if $\phi_p=\phi_\ka$ denotes the principal part of $\phi,$ we must have $m(\phi_p)>d(\phi_p),$
and  $m(\phi_p)$ is just  the multiplicity of the principal root $b_1x_1^{m_1}$ of the $\ka$-homogeneous polynomial $\phi_p.$ All other roots have multiplicity less or equal to $d(\phi_p).$ 

\medskip
This already indicates that the function $\phi$ will indeed be small of ''highest order''  (in some averaged sense)  near the curve $x_2=\sigma(x_1)$ given by the principal root jet (even though $\phi$ need not vanish on this curve!), so that the region close to this curve should indeed give the main contribution to the maximal operator.  
\medskip

In order to localize to  a $\ka$-homogeneous region away from the principal root  jet,   put, in a first step, 
$$
\rho_1\x:=\rho\Big(\frac{ x_2-b_1x_1^{m_1}}{\ve_1x_1^{m_1}}\Big),
$$
where $\ve_1>0$ is a small parameter to be determined later, and set 
$$
A^{1-\rho_1}_tf(y):=\int_{\bR^2} f\Big(y_1-tx_1,y_2-tx_2, y_3-t(1+\phi(x_1,x_2))\Big)\Big(1-\rho(\frac{ x_2-b_1x_1^{m_1}}{\ve_1x_1^{m_1}})\Big)\, \eta(x)\, dx.
$$
By $\M^{1-\rho_1}$ we denote the associated maximal operator. We can now argue exactly as in the proof of Proposition \ref{s5.1}. Using the dilations $\de_r\x=\de^\ka_{r}\x:=(r^{\ka_1}x_1,r^{\ka_2}x_2), \ r>0,$ we can dyadically decompose the operators $A^{1-\rho_1}_t$ into the sum of operators $A_t^k,$ which, after re-scaling, are given by
\begin{eqnarray*}
(T^{-k}A^k_tT^k)f(y,y_3)
=2^{-k|\ka|} \int_{\bR^2}&& f\Big(y_1-tx_1,y_2-tx_2, y_3-t(1+\phi^k(x_1,x_2))\Big)\\
&&\Big(1-\rho(\frac{ x_2-b_1x_1^{m_1}}{\ve_1x_1^{m_1}})\Big)\, \eta(\de_{2^{-k}}x)\, \chi(x)\,dx.
\end{eqnarray*}
All roots of $\phi_p$ lying in the domain of integration have a positive distance to the principal root $b_1x_1^{m_1},$ hence have multiplicities bounded by the distance $d(\phi_p) $ (cf. Corollary \ref{s4.2}), so that we can again estimate the associated maximal operators $\M^k$ by means of Proposition \ref{s3.4} (applied possibly in a rotated coordinate system) and obtain

\begin{lemma}\label{s6.2} 
If   the neighborhood $\Om$ of the point $(0,0)$ is chosen sufficiently small,  then the maximal operator
$\M^{1-\rho_1}$   is bounded on $L^p(\RR^2)$   for every  $p>h(\phi).$
\end{lemma}

\medskip
We have thus reduced considerations to a narrow $\ka$-homogeneous domain near the curve $x_2=b_1x_1^{m_1},$ of the form 
$$
|x_2-b_1x_1^{m_1}|\le \ve_1 x_1^{m_1},
$$
where $\ve_1>0$ can be chosen arbitrarily small.

\medskip
\subsection{The roots of $\tp$}\label{roots}
For our further reduction, we need more information on $\tp.$ Let us assume for a while  that $\phi$ is real analytic  (in Subsection \ref{smoothan} we shall explain how the general case can be reduced to the analytic setting). According to  \cite{ikromov-m} and following \cite{phong-stein},  we may then write
$$
\tp(y_1,y_2)=U(y_1,y_2)y_1^{\nu_1} y_2^{\nu_2} \prod_{l=1}^n
\Phi\left[ \begin{matrix} 
\cdot\\
l
\end{matrix}\right ](y_1,y_2),
$$
where $U(0,0)\ne 0$ and
$$
\Phi\left[ \begin{matrix} 
\cdot\\
l
\end{matrix}\right ](y_1,y_2):=\prod_{r\in \left[ \begin{matrix} 
\cdot\\
l
\end{matrix}\right ]}(y_2-r(y_1)).
$$
The  roots  $r(y_1)$ arising in this display can be expressed in a small neighborhood of $0$ as   Puiseux series 
$$
r(y_1)=c_{l_1}^{\al_1}y_1^{a_{l_1}}+c_{l_1l_2}^{\al_1\al_2}y_1^{a_{l_1l_2}^{\al_1}}+\cdots+
c_{l_1\cdots l_p}^{\al_1\cdots \al_p}y_1^{a_{l_1\cdots
l_p}^{\al_1\cdots \al_{p-1}}}+\cdots,
$$
where
$$
c_{l_1\cdots l_p}^{\al_1\cdots \al_{p-1}\be}\neq c_{l_1\cdots
l_p}^{\al_1\cdots \al_{p-1}\ga} \quad \mbox{for}\quad \be\neq \ga,
$$

$$
a_{l_1\cdots l_p}^{\al_1\cdots \al_{p-1}}>a_{l_1\cdots
l_{p-1}}^{\al_1\cdots \al_{p-2}},
$$
with strictly positive exponents $a_{l_1\cdots l_p}^{\al_1\cdots \al_{p-1}}>0$ and non-zero complex coefficients $c_{l_1\cdots l_p}^{\al_1\cdots \al_p}\ne 0,$ and where we have kept enough terms to distinguish between all the non-identical roots of $\tp.$

The {\it cluster } $\left[ \begin{matrix} \al_1&\cdots &\al_p\\
l_1&\dots &l_{p}
\end{matrix}\right ]$
designates all the roots $r(y_1)$, counted with  their multiplicities,
which satisfy
\begin{equation}\label{6.11}
r(y_1)-c_{l_1}^{\al_1}y_1^{a_{l_1}}+c_{l_1l_2}^{\al_1\al_2}y_1^{a_{l_1l_2}^{\al_1}}+\cdots+
c_{l_1\cdots l_p}^{\al_1\cdots \al_p}y_1^{a_{l_1\cdots
l_p}^{\al_1\cdots \al_{p-1}}}=O(y_1^b) \
\end{equation}
for some exponent $b>a_{l_1\cdots l_p}^{\al_1\cdots \al_{p-1}}$. We
also introduce the clusters 
$$
 \left[ \begin{matrix} 
 \al_1&\cdots &\al_{p-1}&\cdot\\
l_1&\dots &l_{p-1}& l_p
\end{matrix}\right ]
:=
\bigcup\limits_{\al_p}
\left[ \begin{matrix} 
\al_1&\cdots &\al_p\\
l_1&\dots &l_p
\end{matrix}\right ].
$$

Each index $\al_p$ or $l_p$ varies in some finite range which we shall not specify here. We finally put 
$$
N\left[ \begin{matrix} 
\al_1&\cdots &\al_p\\
l_1&\dots &l_p
\end{matrix}\right ]:=\mbox{number of roots in} \,
\left[ \begin{matrix} 
\al_1&\cdots &\al_p\\
l_1&\dots &l_p
\end{matrix}\right ],
$$
$$
N \left[ \begin{matrix} 
 \al_1&\cdots &\al_{p-1}&\cdot\\
l_1&\dots &l_{p-1}& l_p
\end{matrix}\right ]:=\mbox{number of roots in}
\,  \left[ \begin{matrix} 
 \al_1&\cdots &\al_{p-1}&\cdot\\
l_1&\dots &l_{p-1}& l_p
\end{matrix}\right ]
$$
\medskip

Let $a_1<\dots< a_l<\dots<a_n$ be the distinct leading exponents of all the 
roots  $r.$ Each exponent $a_l$ corresponds to the cluster $\left[ \begin{matrix} 
\cdot\\
l
\end{matrix}\right ],$ so that  the set of all roots  $r$ can be divided as $\bigcup\limits_{l=1}^n
\left[ \begin{matrix} 
\cdot\\
l\end{matrix}\right ]$.

We introduce the following quantities:
\begin{equation}\label{6.12}
A_l=A\left[ \begin{matrix} 
\cdot\\
l\end{matrix}\right ]:=\nu_1+\sum_{\mu\le l}a_\mu N\left[ \begin{matrix} 
\cdot\\
\mu
\end{matrix}\right ],\quad
B_l=B\left[ \begin{matrix} 
\cdot\\
l
\end{matrix}\right ]:=\nu_2+\sum_{\mu\ge l+1}N\left[ \begin{matrix} 
\cdot\\
\mu
\end{matrix}\right ].
\end{equation}
Then the vertices of the Newton diagram $\N_d(\tp)$ of $\tp$ are  the   points
$(A_l,B_l), \  l=0,\dots,n,$ and the Newton polyhedron $\N(\tp)$ is the convex hull
of the set $\cup_{l} ((A_l,B_l)+\bR_+^2)$.

Let $L_l:=\{(t_1,t_2)\in \bN^2:\ka^l_1t_1+\ka^l_2 t_2=1\}$ denote the line passing through the  points $(A_{l-1},B_{l-1})$ and $(A_l,B_l).$ 
Then 
$$
\frac {\ka^l_2}{\ka^l_1} =a_l,
$$
which in return is the reciprocal of the slope of the line $L_l.$  The line $L_l$ intersects the bisectrix at the point  $(d_l,d_l), $ where 
$$d_l :=\frac{A_l+a_lB_l}{1+a_l}=\frac{A_{l-1}+a_{l}B_{l-1}}{1+a_l}.$$

Finally,  the $\ka^l$-principal part $\tp_{\ka^l }$ of $\tp$ corresponding to the supporting line $L_l$  is given by 
\begin{equation}\label{6.13}
\tp_{\ka^l }(y)=c_l\, y_1^{A_{l-1}} y_2^{B_l}\prod_\al \Big(y_2-c^\al_l y_1^{a_{l}}\Big)^{N\aol}.
\end{equation}

In view of this identity, we shall say that the edge $\ga_l:=[(A_{l-1},B_{l-1}) ,(A_l,B_l)]$ is {\it associated}  to the cluster of roots $\dotol.$

Now, in case (a) where the principal face of $\tp$ is a compact edge, we choose $\la$ so  that the edge $\ga_\la=[(A_{\la-1},B_{\la-1}) ,(A_\la,B_\la)]$ is the principal face $\pi(\tp)$ of the Newton polyhedron of $\tp.$ Then 
$$\tilde \ka=\ka^\la,\  a_\la=a=\frac {\tilde \ka_2}{ \tilde \ka_1}.
$$

In the case (b), where $\pi(\tp)$ is a vertex $(A_{\la-1},B_{\la-1})$ with 
$A_{\la-1}=B_{\la-1}=N=h(\phi)$ (which one may view as the limiting case of case (a) after shrinking the principal edge $\ga_\la$ to this single point), and also in case (c), where $\pi(\tp)$ is a horizontal half-line,  or a compact edge, with left endpoint $(A_{\la-1},B_{\la-1})$ such that $B_{\la-1}=N=h(\phi),$ we shall slightly abuse our previous notation and {\it define}
$$\ka^\la:=\tilde \ka, \ a_\la:=a=\frac {\tilde \ka_2}{ \tilde \ka_1}.
$$
Note that $a_\la=a$ may then possibly not be the reciprocal of the slope of an edge of $\N(\tp).$ 

\medskip
In the sequel, it will be better  to work with the $\tilde\ka$-principal part $\tp_{\tilde\kappa}$ of $\tp$ in place of the principal part $\tp_p.$ 
The following observation will become  useful.
\begin{lemma}\label{s6.3} 
If $h(\phi)\ge 2,$  then $\pa_2^2\tilde\phi_{\ka^l}$ does not vanish identically, and $\ka^l_2<1,$  for any 
$l\le \la.$ 
\end{lemma}

\proof Consider first the situation where the principal face $\pi(\tp)$ is an edge. Here  the statements will already follow from our general assumption  $\nabla\phi(0)=0.$ Indeed, write $\tilde\phi_{\ka^l}$  according to \eqref{6.13} in the form 
$$
\tilde\phi_{\ka^l}(y)=c y_1^{\nu_1}y_2^{\nu_2}\prod_{s=1}^M(y_2-\la_s
y_1^{a_l})^{n_s}.
$$
with $\la_s\ne0,$ where then $M\ge 1.$

 If we assume that $\pa_2^2\tilde\phi_{\ka^l}=0,$ then clearly 
$\nu_2+\sum_s n_s\le 1,$ so that there is only one, real root $\la_1y_1^{a_l}$ of multiplicity one. This implies that $ \tilde\phi_{\ka^l} (y)=c y_1^{\nu_1}(y_2-\la_1y_1^{a_l}).$ Thus the Newton diagram $\ga_l=\N_d(\tp_{\ka^l})$ is the interval $[(\nu_1,1),(\nu_1+a_l,0)].$ Since $l\le \la,$ its left endpoint must lie above  the bisectrix,  so that $\nu_1=0.$ But then $\nabla\tp_{\ka^l}(0)\ne 0,$ hence  $\nabla\phi(0)\ne 0,$ a contradiction.

\medskip

A similar argument applies to show that $\ka^l_2<1.$ Indeed, since the polynomial $\tilde\phi_{\ka^l}$  is $\ka^l$-homogeneous of degree one, and since $M\ge 1,$   $\ka^l_2\ge 1$ would imply that 
$\nu_1=\nu_2=0$ and $\sum_s n_s=1,$ so that we could conclude as before that $\nabla\phi(0)\ne 0.$

Finally, if $\pi(\tp)$ is a vertex or an unbounded edge,  then the previous arguments still apply for $l<\la.$ And, for $l=\la,$ by \eqref{6.8} the polynomial $\tp_{\ka^\la}$ is of the form $cy_1^{\nu_1}(y_2-b y_1^a)^N,$ with $\nu_1\le N=h(\phi)\ge 2,$ so that the statements are obvious.

\qed

\medskip
\subsection{Further domain decompositions}\label{fdc}
 We have seen that we can control the maximal operator associated to  sub-domains    
$$
|x_2-b_1x_1^{m_1}|\ge \ve_1 x_1^{m_1}
$$
of $\Om,$ where $\ve_1>0$ can be chosen arbitrarily small. Since $m_1$ is the leading exponent of $\psi,$ choosing $\Om$ sufficiently small we see that we can reduce our considerations to a domain of the form
$$
|x_2-\psi(x_1)|\le \ve_1 x_1^{m_1}.
$$
This domain,  except for a small $\ka^\la$-homogeneous neighborhood of  the principal root jet of the form $ |x_2-\psi(x_1)|\le \ve_\la  x_1^{a_\la},$  will be decomposed into  domains  $D_l$ of the form
$$
D_l:=\{\ve_l  x_1^{a_l}< |x_2-\psi(x_1)|\le N_l x_1^{a_l}\},\quad l=l_0,\dots,\la,
$$
which, when expressed in terms of the coordinates $y,$ are $\ka^l$-homogeneous, and the intermediate domains
$$E_l:=\{N_{l+1} x_1^{a_{l+1}}<|x_2-\psi(x_1)| \le \ve_l x_1^{a_l}\}, \quad l=l_0,\dots,\la-1,
$$
and 
$$E_{l_0-1}:=\{N_{l_0} x_1^{a_{l_0}}<|x_2-\psi(x_1)| \le \ve_1 x_1^{m_1}\}.
$$
Here,  the $\ve_l>0 $ are arbitrarily small and the $N_l>0$ are arbitrarily large parameters, and $l_0\ge 1$ is chosen such that 
\begin{equation}\label{6.14}
a_l\le m_1 \ \mbox{for} \ l<l_0\ \mbox{and}\ a_l>m_1 \ \mbox{for} \ l\ge l_0.
\end{equation}

\medskip
To localize to domains of type $D_l,$ we put 
$$
\rho_l\x:=\rho\Big(\frac{ x_2-\psi(x_1)}{N_lx_1^{a_l}}\Big)
-\rho\Big(\frac{ x_2-\psi(x_1)}{\ve_lx_1^{a_l}}\Big), \quad l=l_0,\dots,\la,
$$
and set 
$$
A^{\rho_l}_tf(z):=\int_{\bR^2} f\Big(z_1-tx_1,z_2-t x_2, z_3-t(1+\phi(x_1,x_2))\Big)\rho_l(x) \, \eta(x)\, dx,
$$
 with associated maximal operator $\M^{\rho_l}.$
  
  \medskip
Similarly, in order to localize to domains of type $E_l,$ we put 

$$
\tau_l\x:=\rho\Big(\frac{ x_2-\psi(x_1)}{\ve_lx_1^{a_l}}\Big)\, (1-\rho)
\Big(\frac{ x_2-\psi(x_1)}{N_{l+1}x_1^{a_{l+1}}}\Big), \quad l=l_0,\dots,\la-1,
$$
and 
$$
\tau_{l_0-1}\x:=\rho\Big(\frac{ x_2-\psi(x_1)}{\ve_1 x_1^{m_1}}\Big)\, (1-\rho)
\Big(\frac{ x_2-\psi(x_1)}{N_{l_0}x_1^{a_{l_0}}}\Big), 
$$
and set 
$$
A^{\tau_l}_tf(z):=\int_{\bR^2} f\Big(z_1-tx_1,z_2-t x_2, z_3-t(1+\phi(x_1,x_2))\Big)\tau_l(x) \, \eta(x)\, dx,
$$
 with associated maximal operator $\M^{\tau_l}.$

\medskip

Notice that it suffices to control all the maximal operators defined in this way in order to prove Proposition \ref{s6.1}.

\medskip
\subsection{The maximal operators  $\M^{\rho_l}$}\label{subs6.4}

 \begin{lemma}\label{s6.4} 
If   the neighborhood $\Om$ of the point $(0,0)$ is chosen sufficiently small,  then the maximal operator
$\M^{\rho_l}$   is bounded on $L^p(\RR^2)$   for every  $p>h(\phi).$
\end{lemma}
 
 \proof  I) We begin with the special case $l=\la,$ where $\ka_\la=\tilde\ka.$
 \medskip
 
 The change of variables \eqref{6.3} transforms the integral for $A^{\rho_\la}_tf(z)$ into 
 $$
A^{\rho_\la}_tf(z)=\int_{\bR^2} f\Big(z_1-ty_1,z_2-t(y_2+\psi(y_1)), z_3-t(1+\tilde \phi(y_1,y_2))\Big)\tilde \rho_\la(y) \, \tilde \eta(y)\, dy,
$$
with 
$$
\tilde\rho_\la (y):=\rho\Big(\frac{ y_2}{N_\la y_1^{a_\la}}\Big)
-\rho\Big(\frac{ y_2-c^\beta_\la y_1^{a_\la}}{\ve_\la y_1^{a_\la}}\Big), 
$$
and $\tilde\eta(y):=\eta(y_1,y_2+\psi(y_1)).$ Since $\tilde\phi_{\tilde\ka}$ is $\tilde\ka$-homogeneous of degree one and  $\tilde\rho_\la$ is $\tilde\ka$-homogeneous of degree zero  with respect to the new dilations $\tilde\de_r(y_1,y_2):=\de^{\tilde\ka}_{r}(y_1,y_2):=(r^{\tilde \ka_1}y_1,r^{\tilde\ka_2}y_2), \ r>0,$ using these dilations we now  dyadically decompose the operators $A^{\rho_\la}_t$ into the sum of operators $A_t^k,$ with associated maximal operators $\M^k,$ given by 
\begin{eqnarray*}
A_t^kf(z)
=2^{-k|\tilde \ka|} \int_{\bR^2}&& f\Big(z_1-t2^{-\tilde\ka_1k}y_1,z_2-t(2^{-\tilde\ka_2k}y_2+\psi(2^{-\tilde\ka_1k}y_1)),\\
&& z_3-t(1+2^{-k}\tilde \phi^k(y_1,y_2))\Big)
\tilde \rho_\la(y) \, 
\tilde \eta(\tilde\delta_{2^{-k}}y)\, \chi(y)\, dy,
\end{eqnarray*}
with 
$$
\tilde\phi^k(y):=\tilde\phi_{\tilde\ka}(y)+2^k\tilde\phi_r(\tilde \de_{2^{-k}} y).
$$
Notice that $2^k\tilde\phi_r(\tilde \de_{2^{-k}} y)=O(2^{-\ve k})$ in $C^\infty,$ for some $\ve>0,$ so that this term can be considered as a perturbation term. 
Re-scaling by means of the operators
$$
\tilde T^kf(z,z_3):=2^{\frac{k|\tilde\ka|}{p}}f(\tilde\de_{2^k}(z),z_3),
$$
we obtain
\begin{eqnarray*}
(\tilde T^{-k}A^k_t\tilde T^k)f(z)
=2^{-k|\tilde \ka|} \int_{\bR^2}&& f\Big(z_1-ty_1,z_2-t(y_2+\psi^k(y_1)),\\
&& z_3-t(1+2^{-k}\tilde \phi^k(y_1,y_2))\Big)
\tilde \rho_\la(y) \, 
\tilde \eta(\tilde \delta_{2^{-k}}y)\, \chi(y)\, dy,
\end{eqnarray*}
where by our construction of $\psi(x_1)=b_1x_1^{m_1}+\cdots$
\begin{equation} \nonumber
\psi^k(y_1):=2^{\tilde\ka_2k}\psi(2^{-\tilde\ka_1k}y_1)
= O(2^{(\tilde\ka_2-\tilde\ka_1 m_1)k})\ \mbox{in} \ C^\infty.
\end{equation}

Applying the change of variables $x_1:=y_1,\ , x_2:=y_2+\psi^k(y_1)$ in this integral, we eventually arrive at 
\begin{eqnarray}\nonumber
&&(\tilde T^{-k}A^k_t\tilde T^k)f(z)
=2^{-k|\tilde \ka|} \int_{\bR^2} f\Big(z_1-tx_1,z_2-tx_2,\nonumber \\
&& z_3-t(1+2^{-k}\tilde \phi^k(x_1,x_2-\psi^k(x_1)))\Big)
 \rho^k_\la(x) \, \chi^k(x)\,  \eta(\tilde\delta_{2^{-k}}x)\,dx,\label{6.15}
\end{eqnarray}
with $\chi^k(x):=\chi(x_1, x_2-\psi^k(x_1))$ and $\rho^k_\la(x):=\tilde\rho_\la(x_1, x_2-\psi^k(x_1)).$
Since $\tilde\ka_2-\tilde\ka_1m_1=\tilde\ka_1(a-m_1)>0$ (compare \eqref{6.10}), we can no longer argue with Proposition \ref{s3.4} as in the previous cases in order to estimate the corresponding maximal operators. However, we shall see that we can make use of Corollary \ref{s3.5} in combination with Proposition \ref{s4.4}.

\medskip

Notice that, according to Lemma \ref{s6.3},  $\pa_2^2\tilde\phi_p$ does not vanish identically.
We shall prove  that for every point $y^0$ in the support of $\tilde\rho_\la \chi $ the following holds true:

\begin{equation}\label{6.16}
\mbox{There exists some $j$ with $2\le j\le  h(\phi)$ such that $\pa_2^j\tilde\phi_{\tilde\ka}(y^0)\ne 0.$}
\end{equation}

\medskip

To prove this, consider first the  case (a), where $\tp_{\tilde\ka}=\tp_p.$   If  $a_p\in\NN,$ then by \eqref{6.10}, we have  $\frac {\tilde \ka_2}{ \tilde \ka_1}\ge 3.$ Recall also that we have changed coordinates in  such a way that the root of maximal multiplicity of $\pa_2^2\tp_p(y)$ away from the $y_2$ - axis is given by $y_2=0.$ Our claim therefore follows from Proposition \ref{s4.4} (a) applied to $\tp_p,$ since $d(\tilde\phi_p)\le h(\phi)$  (notice that the support of $\tilde\rho_\la \chi $  has positive distance to the root $y_2=0$ of $\tilde\phi_p$).

On the other hand, if  $a_p\notin \NN,$  then we can argue in the same way as before, by applying  Proposition \ref{s4.4} (b) in place of  Proposition \ref{s4.4} (a), unless $\tilde\phi_p$ is one of the exceptional polynomials $P$  given by \eqref{4.5}. 

\medskip
In fact, these exceptional polynomials are $\tilde\ka$-homogeneous of degree one, with $\tilde\ka_1:=\frac 1{10}$ and $\tilde\ka_2:=\frac 1{4},$ so that here $h(\phi)=1/|\tilde\ka|=20/7.$ Then, necessarily $K=1$ and $m_1=2,$ so that $\psi(x_1)=b_1x_1^2,$ and clearly $a_p=5/2.$ Moreover, 
$$\pa_2^2 P(y)=12c(y_2^2-\tfrac{\la_1+\la_2}6 y_1^5).
$$ 
Thus, if  $\la_1+\la_2<0,$ then we can argue as before and see that even the maximal operator associated to the domain $|x_2-\psi(x_1)|\le N_\la x_1^{5/2},$ for any $N_\la>0,$ is bounded on $L^p$ for $p>h(\phi).$ 

On the other hand, if $\la_1+\la_2>0,$ then $\pa_2^2 \tp_p$ will have real roots given by $y_2=\pm\sqrt{\tfrac{\la_1+\la_2}6}\, y_1^{5/2}.$ In this case, with  the same technics we can still reduce to small neighborhoods of  these roots, which, in our original coordinates $x,$ are of the form 
\begin{equation}\label{6exc}
\Big|x_2-(b_1 x_1^2\pm \sqrt{\tfrac{\la_1+\la_2}6}\, x_1^{5/2})\Big|\le \ve_0x_1^{5/2},
\end{equation}
where $\ve_0$ can be chosen as small as we wish.

\medskip These remaining domains  will be treated in Subsection \ref{exception}.

\bigskip

In the cases (b) and (c), according to \eqref{6.8} we can write $\tp_{\tilde\ka}$ in the form
$$\tp_{\tilde\ka}(y)=cy_1^{\nu_1}(y_2-by_1^{m})^N,$$
with $\nu_1\le N:=h(\phi),$  where  $b=0,\nu_1=N$ in the  case (b). 
It follows that   \eqref{6.16} holds true with $j=h(\phi).$ 

\medskip

Notice that in these cases, this remains true even 
for $y^0$ lying on the $y_2$-axis!

\medskip
Our claim is thus proved. Observe also that, if we put $\ve:=2^{-k}$ and $\psi_\ve(x_1):=\psi^k(x_1)$ in Corollary \ref{s3.5}, then, by \eqref{6.15}, $\psi_\ve=O(\ve^{-\de})$ in $C^\infty,$ with
$0<\de:= \tilde\ka_2-\tilde\ka_1m_1<1.$ 

We can therefore apply Corollary \ref{3.6} to the maximal operators $\M^k$ and obtain the estimate 
$$
||\M^k f||_p\le C 2^{-|\tilde \ka| k+\frac k p} ||f||_p.
$$ 
Since $p>h(\phi)\ge \frac 1{|\tilde\kappa|}$ (compare  \eqref{6.9}), these estimates sum in $k$ and we obtain the desired estimate  for the maximal operator $\M^{\rho_\la}.$ Notice here that the term  $\eta(\tilde\delta_{2^{-k}}x)$ in the integral above causes no problem in the application of Corollary \ref{3.6}, since  $2^{-\tilde \ka_2 k}\psi^k(x_1) $ is small. 

\bigskip

II) We now turn  to  the case $l_0\le l\le \la-1.$ We can here follow the arguments in case i) up to formula \eqref{6.15} almost verbatim, if we replace $\tilde \ka$ by $\ka^l,$  the function $\tilde\rho_\la$ by  the $\ka^l$-homogeneous function 
$$
\tilde\rho_l(y):=\rho\Big(\frac{ y_2}{N_l x_1^{a_l}}\Big)
-\rho\Big(\frac{ y_2}{\ve_l y_1^{a_l}}\Big), 
$$
the dilations $\tilde\delta_r$ by the dilations $\delta_r^l\x:=(r^{\ka_1^l}x_1,r^{\ka_2^l}x_2)$ and the function $\tp_{\tilde\ka}$ by the $\ka^l$-homogeneous part 
$$
\tp_l:=\tp_{\ka^l}
$$
of $\tp.$ Notice that,  again by Lemma \ref{s6.3},  $\pa_2^2\tilde\phi_l$ does not vanish identically and  $\ka^l_2<1.$ 
Moreover,  because of \eqref{6.14} we then have 
$$
0<\ka_1^l(a_l-m_1)=\ka_2^l-\ka_1^l m_1<1\ \mbox{and} \ \frac{\ka_2^l}{\ka_1^l}=a_l>2.
$$
What remains to be shown in order to conclude as in the previous case $l=\la$ is that given $y^0$ in the support of $\tilde\rho_l \chi ,$  then there exists some $2\le j\le h(\phi)$ such that $\pa_2^j\tilde\phi_l(y^0)\ne 0.$ Notice that for roots   in the support of $\tilde\rho_l \chi ,$we have $y^0_2\ne 0.$ 

Now, from the geometry of the Newton polyhedron of $\tp,$ it is evident that $d_h(\tilde\phi_l)\le d(\tp_p)\le h(\phi).$ It will therefore be sufficient to prove that 

\begin{equation}\label{6.17}
\pa_2^j\tilde\phi_l(y^0)\ne 0\ \mbox{for some}\ 2\le j\le d_h(\tilde\phi_l).
\end{equation}

\medskip
i) Consider first the case where $\frac {\ka^l_2}{\ka^l_1} =a_l\notin \NN.$ Then, by Proposition \ref{s4.4} (c), \eqref{6.17} is true, unless $\tp_l$ is the exceptional polynomial \eqref{4.5}. But, in the latter case the Newton diagram $\N_d(\tp_l)$ of $\tp_l$ would be the interval $[(0,4),(10,0)],$ which intersects the bisectrix, so that $\tp_l$ would have to be the principal part of $\tp,$ contradiction our assumption $l<\la.$

\medskip
ii) Assume finally that $\frac {\ka^l_2}{\ka^l_1} =a_l\in \NN,$ so that $a_l\ge 3.$ We first show that the root $y_2=0$ has maximal multiplicity $B_l>d_h(\tp_l)$ among all real roots of  $\tp_l$ away from the $y_2$-axis.

Indeed, since $a_{l_0}>m_1,$ it is clear from Varchenko's algorithm (see \cite{ikromov-m}) that the edge $\tilde\ga_l$ of the Newton polyhedron of $\tp$ associated to  the $\ka_l$-homogeneous polynomial  $\tp_l$ is an interval which is contained  in an  edge of the Newton polyhedrons arising in the course of this algorithm. More precisely, we must have $a_l=m_k$ for some $k<K,$ if $K<\infty,$ or $k< L,$ if $K=\infty.$ 

 If then  $\phi^{(k-1)}(z_1,z_2):=\phi(z_1,z_2+\sum_{j=1}^{k-1}b_j z_1^{m_j})$ is the  function appearing in the $(k-1)$st step of the algorithm, then $b_kz_1^{m_k}=b_kz_1^{a_l}$ is the principal root of $(\phi^{(k-1)})_p(z),$ which has multiplicity $B_k>d_h(\phi^{(k-1)}_p),$ since the coordinates $z$ are not yet adapted. The next step in the algorithm, which changes the coordinates to $y_1=z_1, y_2=z_2-b_kz_1^{a_l},$ turns the root $z_2=b_kz_1^{a_l}$ into the root $y_2=0,$ still of  multiplicity $B_k,$ of the $\ka_l$-homogeneous polynomial $(\phi^{(k-1)})_p(y_1,y_2+b_k y_1^{m_k}).$ But, in the subsequent steps of the algorithm, only terms of higher order than $O( y_1^{a_l})$ are added, so that clearly $\tp_l(y)=(\phi^{(k-1)})_p(y_1,y_2+b_k y_1^{m_k}),$ and $y_2=0$ is the real root of highest multiplicity $B_k$ of $\tp_l.$ 

Since the edge $\tilde\ga_l$ lies in the closed subspace above the bisectrix, we then conclude by means of Proposition \ref{s4.1} that in fact $B_k=d(\tp_l)>d_h(\tp_l).$ Moreover, the left end point of this edge is of the form $(A_k,B_k),$ and since it belongs to the Newton diagram, but not to the principal face,  of $\tp,$ it is clear from the geometry of the Newton polyhedron that  $d(\tp_l)=B_k> h(\tp)=h(\phi)\ge 2.$

 This shows that we can apply Proposition \ref{s4.4} (c) to $\tp_l$ and obtain \eqref{6.17}.

\medskip

\qed

\medskip
\subsection{The maximal operators  $\M^{\tau_l}$}\label{tau}

 \begin{lemma}\label{s6.5} 
If   the neighborhood $\Om$ of the point $(0,0)$ is chosen sufficiently small,  then the maximal operator
$\M^{\tau_l}$   is bounded on $L^p(\RR^2)$   for every  $p>h(\phi).$
\end{lemma}
 
 \proof  I) We begin with the case $l_0\le l\le \la-1.$    Since the domain $E_l,$ when viewed in $y$-coordinates, is a domain of transition between two different homogeneities, namely the ones given by the weights $\ka^l$ and $\ka^{l+1}$ (at least if $l\ge 1$), we shall apply an idea from \cite{phong-stein} and decompose it dyadically in each coordinate separately,  and then re-scale each of the bi-dyadic pieces obtained in this way.

 By the change of variables \eqref{6.3},  we can write 
 $$
A^{\tau_l}_tf(z)=\int_{\bR^2} f\Big(z_1-ty_1,z_2-t(y_2+\psi(y_1)), z_3-t(1+\tilde \phi(y_1,y_2))\Big)\tilde \tau_l (y) \, \tilde \eta(y)\, dy,
$$
with 
$$
\tilde\tau_l(y):=\rho\Big(\frac{ y_2}{\ve_ly_1^{a_l}}\Big)\, (1-\rho)
\Big(\frac{ y_2}{N_{l+1}y_1^{a_{l+1}}}\Big), 
$$
and $\tilde\eta(y):=\eta(y_1,y_2+\psi(y_1)).$  

Consider a dyadic partition of unity 
$\sum_{k=0}^\infty \chi_k(s)=1,\  (0<s<1)$ on $\RR,$ with $\chi \in C_0^\infty(\RR)$ supported in the interval $ [1/2,4],$ 
where $\chi_k(s):=\chi(2^ks),$ and put 
$$
\chi_{j,k}(x):=\chi_j(x_1)\chi_k(x_2),\ j,k\in\NN.
$$
We then decompose $A^{\tau_l}_t$ into the operators 
$$
A^{j,k}_tf(z):=\int_{\bR^2} f\Big(z_1-ty_1,z_2-t(y_2+\psi(y_1)), z_3-t(1+\tilde \phi(y_1,y_2))\Big)\tilde \tau_l (y) \, \tilde \eta(y)\,\chi_{j,k}(y)\, dy,
$$
with associated maximal operators $\M^{j,k}.$

Notice that by choosing the neighborhood $\Om$ of the origin sufficiently small, we need only consider sufficiently large $j,k.$ Moreover, because of the localization imposed by $\tilde \tau_l ,$ it suffices to consider only pairs $(j,k)$ satisfying
\begin{equation}\label{6.18}
a_lj+M\le k\le a_{l+1}j-M,
\end{equation}
where $M$ can still be choosen sufficiently large, because we had the freedom to choose $\ve_l$ sufficiently small and $N_{l+1}$ sufficiently large. In particular, we have $j\sim k.$

By re-scaling in the integral, we have 
\begin{eqnarray*}
A_t^{j,k}f(z)
=2^{-j-k} \int_{\bR^2}&& f\Big(z_1-t2^{-j}y_1,z_2-t(2^{-k}y_2+\psi(2^{-j}y_1),\\
&& z_3-t(1+\tilde \phi(2^{-j}y_1,2^{-k}y_2))\Big)
\tilde \tau^{j,k}(y) \, 
\tilde \eta^{j,k}(y)\, \chi(y_1)\chi(y_2)\, dy,
\end{eqnarray*}
with 
$$
\tilde\tau^{j,k}(y):=\rho\Big(\frac{ y_2}{\ve_l2^{k-a_l j}y_1^{a_l}}\Big)\, (1-\rho)
\Big(\frac{ y_2}{N_{l+1}2^{k-a_{l+1} j}y_1^{a_{l+1}}}\Big), 
  \ \tilde \eta^{j,k}(y):=\tilde \eta(2^{-j}y_1,2^{-k}y_2).
$$
Notice that, by \eqref{6.18}, all derivatives of $\tilde \tau^{j,k}$ are uniformly bounded in $j,k.$ 

The scaling operators 
$$
T^{j,k}f(z):=2^{\frac{j+k}p}f(2^j z_1, 2^k z_2,z_3)
$$ then transform these operators into 
\begin{eqnarray*}
(T^{-j,-k}A^{j,k}T^{j,k})f(z)
=2^{-j-k} &&\int_{\bR^2} f\Big(z_1-ty_1,z_2-t(y_2+\psi^{j,k}(y_1),\\
&& z_3-t(1+\tp^{j,k}(y))\Big)
\tilde \tau^{j,k}(y) \, 
\tilde \eta^{j,k}(y)\, \chi(y_1)\chi(y_2)\, dy,
\end{eqnarray*}
where

$$\tp^{j,k}(y):=\tilde \phi(2^{-j}y_1,2^{-k}y_2),\quad\psi^{j,k}(y_1):=2^k\psi(2^{-j}y_1).
$$
\medskip
Notice that 
$$
\psi^{j,k}=O(2^{k-m_1j})\ \mbox{in}\ C^\infty.
$$

Applying the change of variables $x_1:=y_1,\, x_2:=y_2+\psi^{j,k}(y_1)$ in this integral, we eventually arrive at
\begin{eqnarray} \nonumber
(T^{-j,-k}A^{j,k}T^{j,k})f(z)
&=&2^{-j-k} \int_{\bR^2} f\Big(z_1-tx_1,z_2-tx_2, \nonumber\\
 &&z_3-t(1+\tp^{j,k}(x_1,x_2-\psi^{j,k}(x_1)))\Big)
 \tau^{j,k}(x) \, 
\eta^{j,k}(x)\, \chi^{j,k}(x)\, dx,\label{6.19}
\end{eqnarray}
where
\begin{eqnarray*}
&&\tau^{j,k}(x):=\tilde\tau^{j,k}(x_1,x_2-\psi^{j,k}(x_1)), \quad \eta^{j,k}(x):=\eta(2^{-j}x_1,2^{-k}x_2)  \\
&& \mbox{\qquad and}\quad \chi^{j,k}(x):=\chi(x_1)\chi(x_2-\psi^{j,k}(x_1)).
\end{eqnarray*}

\medskip
We next determine $\tp^{j,k},$ up to an error term. To this end, notice that if $y_1\sim 1$ and $y_2\sim 1,$ and if $r\in\dotomu,$  then $r(2^{-j}y_1)=c_\mu^\al 2^{-a_\mu j}y_1^{a_\mu}+O(2^{-\ve (j+k)})$ in $C^\infty,$ for some $\ve>0.$ In view of \eqref{6.18}, 
we thus get 
$$
2^{-k}y_2-r(2^{-j}y_1)=\begin{cases}
-c_\mu^\al 2^{-a_\mu j}\Big(y_1^{a_\mu}+O(2^{-\ve (j+k)})\Big),\  \mbox{if}\  \mu< l,\\
-c_\mu^\al 2^{-a_\mu j}\Big(y_1^{a_\mu}+O(2^{-M})\Big),\  \mbox{if}\  \mu= l,\\
2^{-k}\Big(y_2+O(2^{-\ve (j+k)})\Big),\quad   \mbox{if}\  \mu\ge l+1,
\end{cases}
$$
with $M$ as in \eqref{6.18}.
Multiplying all these terms, we then see that 
\begin{equation}\label{6.20}
\tp^{j,k}(y)=2^{-(A_l j+B_l k)}\Big(c_ly_1^{A_l} y_2^{B_l}+O(2^{-C M})\Big),
\end{equation}
for some constant $C>0,$ where $A_l$ and $B_l$ are given by \eqref{6.12} and $M$ can still be chosen as large as we wish.

Observe that since $l\le \la-1,$ we have  $B_l\ge B_{\la-1}\ge \nu_2+N\dotola,$ and similarly as in the proof of Lemma \ref{s6.3}, it is easy to see that we must have $\nu_2+N\dotola \ge 2,$ hence $B_l\ge 2.$ This implies that 
$$
\pa_2^2(y_1^{A_l} y_2^{B_l})\sim 1, 
$$
and that  $A_l j+B_l k\ge 2k,$ so that 
$$
2^{k-m_1j}\le C 2^{\frac 12(A_l j+B_l k)}.
$$
We can therefore argue in a similar way as in the previous subsection and apply Corollary \ref{3.6} to obtain
$$
||\M^{j,k} f||_{p}\le C2^{\frac{A_l j+B_l k}p-j-k}\,||f||_{p},
$$
whenever $p>2,$ provided  $j+k$ is sufficiently large. 

Summing all these estimates, we thus have 
$$
||\M^{\tau_l} f||_{p}\le CJ\,||f||_{p},
$$
where 
$$
J:=\sum_{(j,k): a_lj+M\le k\le a_{l+1}j-M}2^{\frac{A_l j+B_l k}p-j-k}.
$$

\medskip
Assume now that $p>h(\phi).$ Since $h(\phi)\ge d(\tp_{\tilde\ka})\ge d_h(\tp_{\ka^{l+1}}),$ and since $d_h(\tp_{\ka^{l+1}})=\frac{A_l+a_{l+1}B_l}{1+a_{l+1}},$ we have 
$$p>\frac{A_l+a_{l+1}B_l}{1+a_{l+1}}.
$$
This condition is equivalent to 
\begin{equation}\label{6.21}
(1-\frac {A_l}p)+a_{l+1}(1-\frac {B_l}p)>0.
\end{equation}
Similarly, since the mapping $a\mapsto \frac{A_l+aB_l}{1+a}$ is increasing, we may replace $a_{l+1}$ by $a_l$ in this estimate and also get 
\begin{equation}\label{6.22}
(1-\frac {A_l}p)+a_{l}(1-\frac {B_l}p)>0.
\end{equation}
In order to estimate $J,$ let us write $k$ in the form $k=\theta a_l j +(1-\theta)a_{l+1} j+\om, $ with $0\le \theta\le 1$ and $|\om|\le M. $ Then 
\begin{eqnarray*}
&&j+k-\frac{A_l j+B_l k}p=(1-\frac {A_l}p)j+(1-\frac {B_l}p)k\\
&=&\Big(\theta[(1-\frac {A_l}p)+a_l(1-\frac {B_l}p)]+(1-\theta)[(1-\frac {A_l}p)+a_{l+1}(1-\frac {B_l}p)]\Big)j+(1-\frac {B_l}p)\om.
\end{eqnarray*}
In view of \eqref{6.21}  and \eqref{6.22}, this shows that there exists a positive constant $\ve>0$ such that 
$$j+k-\frac{A_l j+B_l k}p>\ve j,
$$
provided $j$ is sufficiently big. It is now clear that $J<\infty,$ so that the maximal operator $\M^{\tau_l}$ is bounded on $L^p$ whenever $p>h(\phi).$ 

\medskip
\noi II) There remains the case $l=l_0-1.$ This case can be treated in a very similar way (formally, it is like the previous case, only with $a_{l_0-1}$ replaced by $m_1\ge a_{l_0-1}$).
Indeed, in this case \eqref{6.18} must be replaced by the inequalities
$$
m_1 j+M\le k\le a_{l_0}j-M,
$$
from which one derives that \eqref{6.20} remains valid, with $l=l_0-1.$ From here, we can proceed exactly as before.

\qed
\medskip

\subsection{Reduction of the smooth case to the analytic setting}\label{smoothan}

The estimates for the maximal operators in the preceding subsections hold true also for  smooth functions $\phi.$ Indeed, denote by $\phi_{(n)}$ the Taylor polynomial of order $n$ of  $\phi$ centered at the origin. For $n$ sufficiently large,  the Newton polyhedra of  $\phi$ and $\phi_{(n)}$  coincide, as do their faces and the corresponding principal parts (see \cite{ikromov-m}). It is then clear that the estimations of the operators $\M^{\rho_l}$ work in the same way for smooth functions as in the analytic setting. Moreover, there exists a constant $c>0$ such that if
$R_{j,k}$ is a dyadic rectangle on which  $x_1\sim 2^{-j}$ and $x_2\sim 2^{-k},$ then the remainder term $\phi-\phi_{(n)}$ is of order $O(2^{-c(j+k)n})$ in $C^\infty(R_{j,k}).$ We may thus apply our previous approach to the polynomial $\phi_{(n)}$ in place of $\phi$ and choose $n$ so large that the contributions of the remainder term $\phi-\phi_{(n)}$ can be considered as negligible errors for the estimations of the operators $\M^{\tau_l}$ (compare the order $O(2^{-c(j+k)n})$ with the order of $\tp^{j,k}$ in formula \eqref{6.20}).

\medskip
\setcounter{equation}{0}
\section{Estimation of the maximal operator $\M$  near the principal root jet}\label{near-root}

We have reduced ourselves  to the   domain
\begin{equation}\label{7.1}
|x_2-\psi(x_1)|\le \ve_0 x_1^{a},
\end{equation} 
where $\psi$ is given as before, with leading term $b_1x_1^{m_1}$ and where $\ve_0>0$ can still be chosen as small as we like. Moreover,  we always assume that $x_1>0.$ 
More precisely, in view of Proposition \ref{s6.1}, there only remains to prove that the maximal operator $\M^{\rho_0}$ associated to this domain is bounded on $L^p(\RR^3)$   for every  $p>h(\phi).$

\medskip

Now, combining what we have proved so far, the following result is easy:

\begin{cor}\label{s7.1}
Let $\phi$ and its associated functions $\pad$ and $\tp$ be as in the previous section, and let $\pi(\pad)$ be  the principal face of the Newton polyhedron of $\pad$ (i.e., of  $\phi$ when expressed in adapted coordinates). If any of the following conditions is satisfied, and   if   the neighborhood $\Om$ of the point $(0,0)$ is chosen sufficiently small,   then the maximal operator
$\M$ is bounded on $L^p(\RR^3)$   for every  $p>h(\phi):$
\bee
\item[(a)]  $\pi(\pad)$ is a compact edge, and there is some $j$ with $2\le j\le h(\phi)$ such that 
$\pa_2^j\tp_p(1,0))\ne 0.$
\item[(b)]  $\pi(\pad)$ is a vertex.
\item[(c)]  $\pi(\pad)$ is a unbounded.
\ee
\end{cor}

\proof Since $\ve_0>0$ can still be chosen as small as we wish, if there exists some $j$ with $2\le j\le h(\phi)$ such that $\pa_2^j\tp_p(1,0))\ne 0,$ it is clear that we can argue near  $y_2=0$  exactly as in the discussion of the maximal operator $\M^{\rho_\la}$ and obtain  for the dyadic constituents $\M^k$ of the operator $\M^{\rho_0}$ the estimate
$$
||\M^k f||_p\le C 2^{-|\tilde \ka| k+\frac k p} ||f||_p,
$$ 
provided $p>h(\phi).$ Moreover, by \eqref{6.9}, these estimates sum in $k$ as before. This proves in particular (a). 

And, we have seen before in Subsection \ref{subs6.4}, that in the cases (b) and (c) such an  integer $j$ exists automatically  - we can indeed choose $j=h(\phi).$ This completes the proof of the corollary.
\qed

In view of this corollary,  we are are left with the proof of 
 \begin{prop}\label{s7.2} 
 Assume that $\pi(\pad)$ is a compact edge, and that 
 \begin{equation}\label{7.2}
\partial_2^j\tp_p(1,0)=0 \ \mbox{for every } \ 2\le j\le h(\phi),
\end{equation}
If   the neighborhood $\Om$ of the point $(0,0)$ is chosen sufficiently small,  then the maximal operator
$\M^{\rho_0}$   is bounded on $L^p(\RR^3)$   for every  $p>h(\phi).$
\end{prop}
Under the assumption \eqref{7.2}, it will no longer be possible to estimate the maximal operator $\M^{\rho_0}$ by means of oscillatory integral estimates in the  variable $x_2$ alone, but we will have to take into account the oscillations in $x_1$ too. 

We shall therefore consider the Fourier transforms of the convolution kernels of the averaging operators $A^{\rho_0}_t,$ i.e., 
$$
\widehat{A^{\rho_0}_t f}(\xi)=e^{it\xi_3}J^{\rho_0}(t\xi)\hat f(\xi),
$$
where
$$
J^{\rho_0}(t\xi):=\int e^{it(\xi_1 x_1+\xi_2 x_2+\xi_3 \phi\x)}\rho\Big(\frac{ x_2-\psi (x_1)}{\ve_0x_1^{a}}\Big)\eta(x)\,dx, \quad\xi\in\RR^3.
$$
Our goal will be  to derive suitable estimates of the oscillatory integrals $J^{\rho_0}(\xi) $ (compare the method in Ê\cite{ikm}).
If we change to the  coordinates  $y_1:= x_1, \ y_2:= x_2-\psi(x_1)$ in the integral (notice that these are adapted, by our construction of $\psi$ in Section \ref{away-root}, since we assume that we are in case (a)),  and assume again that $y_1>0,$ we obtain
$$
J^{\rho_0}(\xi):=\int_{\RR_+^2} e^{i(\xi_1 y_1+\xi_2 \psi(y_1)+\xi_2 y_2+\xi_3 \tp(y))}\rho\Big(\frac{ y_2}{\ve_0y_1^{a}}\Big)\tilde\eta(y)\,dy,
$$
where $\tilde \eta$ is again a smooth cut-off function supported in a sufficiently small neighborhood of the origin and where $\RR_+^2$ denotes the half-plane $\{\x\in\RR^2:x_1>0\}.$ 

\medskip 
At this point, it will be convenient to defray our notation by writing $\phi$ in place of $\tp$ and $\eta$ in place of $\tilde\eta.$ This means that from now  on we shall consider Fourier multipliers  of the form $e^{i\xi_3}J(\xi),$ with
\begin{equation}\label{7.3}
J(\xi):=\int_{\RR_+^2} e^{i\Big(\xi_1 x_1+\xi_2 \psi(x_1)+\xi_2 x_2+\xi_3 \phi(x)\Big)}\rho\Big(\frac{ x_2}{\ve_0x_1^{a}}\Big)\eta(x)\,dx,
\end{equation}

such that  the following general assumptions are fulfilled:
\begin{assumptions}\label{s7.3}
{\rm The functions $\phi, \psi$ and $\eta$ are smooth functions such that
\bee
\item [(i)] $\psi$ is given by $\psi(x_1)=\sum_{l=1}^Kb_lx_1^{m_l} +c_px_1^{a},$ where $b_l\ne 0$ for $l=1,\dots,K;$
\item[(ii)] $\phi$ is of finite type, and  $\phi(0)=0,\nabla \phi(0)=0;$ 
\item [(iii)] the coordinates $x$ are adapted to $\phi,$ i.e., $h(\phi)=d(\phi)\ge 2$;
\item [(iv)] the principal face $\pi(\phi)$ is a compact edge, and the associated principal part $\phi_p$ of $\phi$ is $\ka$-homogeneous of degree one, with $a=\frac {\ka_2}{\ka_1}>m_K\ge m_1\ge 2$ (in particular, $h(\phi)=\frac 1{|\ka|}$);
\item [(v)]  $\eta$ is a smooth bump function supported in a sufficiently small neighborhood $\Om$ of the origin.
\ee}
\end{assumptions}

Moreover, we may and shall assume  that 
\begin{equation}\label{7.4}
\partial_2^j\phi_p(1,0)=0 \ \mbox{for every } \ 2\le j\le h(\phi).
\end{equation}

In order to estimate the maximal operator $\M^{\rho_0}$ associated to the Fourier multiplier $e^{i\xi_3}J(\xi),$
we shall further decompose it and  estimate the corresponding constituents. If $\chi$ is a bounded measurable function, we shall use the notation 
$$J^{\chi}(\xi):=\int_{\RR^2_+} e^{i\Big(\xi_1 x_1+\xi_2 \psi(x_1)+\xi_2 x_2+\xi_3 \phi(x)\Big)}\rho\Big(\frac{ x_2}{\ve_0x_1^{a}}\Big)\eta(x)\chi(x)\, dx.
$$
The corresponding re-scaled Fourier multiplier operators are the averaging operators $A^{\chi}_t$ given by 
$$
\widehat{A^{\chi}_ tf}(\xi)=e^{it\xi_3}J^{\chi}(t\xi)\hat f(\xi),\quad t>0,
$$
with associated maximal operator $\M^{\chi}.$ Then we shall make use the following  essentially well-known result in order to estimate $\M^\chi.$ 
\begin{lemma}\label{s7.4}
Assume that, for some $n\in\NN$ and  $\ve >0,$ the following estimate 
\begin{equation}\label{7.5}
|J^{\chi}(\xi)|\le A_{\chi}||\eta ||_{C^n(\RR^2)}(1+|\xi|)^{-(1/2+\ve)}, \  \xi\in\RR^3,
\end{equation}
holds, where the constant $A_{\chi}$ is independent of $\eta.$ Moreover, put 
$$B_\chi:=\int |\rho\Big(\frac{ x_2}{\ve_0x_1^{a}}\Big)\eta(x)\chi(x)|\, dx.
$$Then, for $2\le p\le\infty,$ 
$$||\M^{\chi}f||_p\le C(A_{\chi})^{\frac 2p} (B_\chi)^{1-\frac 2p} ||f||_p,
$$
where the constant $C$ depends only on the $C^n$- norms of $\phi$ and $\psi$ and the $C^n$-norm of 
$\eta,$ but not on $\chi.$
\end{lemma}

\proof Observe that
 $$|\frac{\pa}{\pa t}[e^{it\xi_3}J^{\chi}(t\xi)]|\le |\xi|\,(|J^{\chi}(t\xi)|+Ê|(\nabla J^{\chi})(t\xi)|),$$
 where, because of \eqref{7.5},
$$|J^{\chi}(\xi)|+|(\nabla J^{\chi})(\xi)|\le  C A_{\chi}(1+|\xi|)^{-(1/2+\ve)}.$$
 
 The desired estimate of the maximal operator for $p=2$  follows then essentially from Littlewood-Paley theory and Sobolev's embedding theorem (for details, compare, e.g.,  \cite{stein-book}, ch.XI.1, or our discussion in Subsection 3.1). Moreover, since $B_\chi$ is just the $L^1$-norm of the convolution kernel of $A_t^{\chi},$ the estimate for $p=\infty$ is trivial. The general case $2\le p\le \infty$ then follows by interpolation. 
 
 \qed

\medskip

\subsection{The case where $\pa_2\phi_p(1,0)\ne 0$}

Let us write
$$\Phi(x,\xi):=\xi_1 x_1+\xi_2\psi(x_1)+\xi_2 x_2+\xi_3 \phi(x)
$$
for the complete phase function of $J,$ and decompose 
$$
\phi=\phi_p+\phi_r.
$$
As in Subsection \ref{subs6.4}, we perform a dyadic decomposition 
$$J=\sum_{k=k_0}^\infty J_k,$$
where 
$$
J_k(\xi):=J^{\chi_k}(\xi)=\int e^{i\Phi(x,\xi)}\rho\Big(\frac{ x_2}{\ve_0x_1^{a}}\Big)\eta(x)\chi_k(x)\, dx,
$$
with $\chi_k(x):=\chi(\delta_{2^k}x).$ Here, the $\delta_r$ denote the dilations with respect to $\ka,$ and $\chi$ is supported in an annulus $1\le |x|\le R.$ Moreover, by choosing the neighborhood $\Om$ of the origin sufficiently small, we may choose $k_0$ as large as we need. Notice that then 
$$
|\M^{\rho_0}f|\le \sum_{k=k_0}^\infty |\M^{\chi_k}f|.
$$
 By a change of coordinates, we obtain
$$
J_k(\xi)=2^{-k |\ka|} \int_{\RR^2} e^{i 2^{-k} \la\Phi_k(x,s)}\rho\Big(\frac{ x_2}{\ve_0x_1^{a}}\Big)\eta(\delta_{2^{-k}}x)\chi(x)\, dx,
$$
where we have put $\la:=\xi_3,\ s=(s_1,s_2)$ and 
$$
\Phi_k(x,s):= s_1x_1+S_2\psi_k(x_1)+ s_2 x_2 +\phi_p\x+\phi_{r,k}(x),
$$
with
\begin{equation}\label{7.6}
\psi_k(x_1):=2^{\ka_1m_1k}\psi(2^{-\ka_1 k}x_1)=b_1 x_1^{m_1}+O(2^{-\delta_1k})\ \mbox{in}\ C^\infty,
\end{equation}
\begin{equation}\label{7.7}
\phi_{r,k}(x):=2^k\phi_r(\delta_{2^{-k}}x)=O(2^{-\delta_2k})\ \mbox{in}\ C^\infty,
\end{equation}
\begin{equation}\label{7.8}
s_1:=2^{(1-\ka_1)k}\tfrac{\xi_1}{\la}, \  s_2:=2^{(1-\ka_2) k}\tfrac{\xi_2}{\la}, \ S_2:=2^{(1-\ka_1m_1)k}\tfrac{\xi_2}{\la}=2^{(\ka_2 -\ka_1m_1)k}s_2,
\end{equation}
(assuming without loss of generality that $\xi_3\ne 0$),
where $\delta_1,\de_2>0.$

We  remark that indeed $\psi_k(x)$ and $\phi_{r,k}(x)$ can be viewed as  smooth functions $\tilde\psi(x_1,\de)$  respectively $\tilde\phi_r(x,\de)$ depending on the  small parameter $\de=2^{-k/r}$ for some suitable positive integer $r\ge 1$ such that 
$$\tilde\psi(x_1,0)=b_1x_1^{m_1},\quad \tilde\phi_r(x,0)\equiv0.
$$

Observe also that 
  $1 -\ka_1m_1>\ka_2 -\ka_1m_1>0,$ $1-\ka_j>0,$ so that in particular
  \begin{equation}\label{7.9}
|S_2|>>|s_2|\ \mbox{and}\ |\la s_j|>>|\xi_j|.
\end{equation}

Recall   that in our  domain of integration, we have 
$$
x_1\sim 1, \ |x_2|\lesssim \ve_0.
$$

\medskip
The following proposition will be useful not only in in the present situation. Its proof will make use of estimates for oscillatory integrals given in the later  Section \ref{oscint}.

\begin{prop}\label{s7.5}
Assume that $\phi$ and $\psi$ satifsy the Assumptions \ref{s7.3} (but not necessarily \eqref{7.4}), and that $\pa_2\phi_p(1,0)\ne 0.$ If $\ve_0$ above is chosen  sufficiently small, then the following estimate 
\begin{equation}\label{7.10}
|J_k(\xi)|\le C||\eta ||_{C^3(\RR^2)}\frac{2^{-k |\ka|}}{(1+|2^{-k}\xi|)^{\tfrac 12+\ve}}
\end{equation}
holds true for some $\ve>0,$ where the constant $C$ does not depend on $k$ and $\xi.$

Consequently, the maximal operator $\M^{\rho_0}$ associated to the averaging  operators $A^{\rho_0}_t,$
$ t>0,$ defined by 
$\widehat{A^{\rho_0}_t f}(\xi)=e^{it\xi_3}J^{\rho_0}(t\xi)\hat f(\xi),
$
is bounded on $L^p(\RR^3)$ for every $p>1/|\ka|.$
\end{prop}

\proof
We shall distinguish several cases, assuming for simplicity that $\la>0.$
\medskip

{\bf 1. Case.} $|s_1|+|S_2|\le C$ for some large constant $C>>1.$
\medskip
In this case, if $k$ is sufficiently large, then we have $|s_2|<<1,$  and since $\pa_2\phi_p(1,0)\ne 0,$ we can integrate by parts in $x_2$ and obtain
$$
|J_k(\xi)|\le C \,2^{-k |\ka|}(1+2^{-k}\la)^{-1},
$$
hence \eqref{7.10}, since, by \eqref{7.9},  $|\xi|\sim \la$ in this case.
\medskip

{\bf 2. Case.} $|s_1|+|S_2|\ge C,$  with $C$ as above, and either 
 $|s_1|<<|S_2|$ or $|s_1|>>|S_2|.$
\smallskip

In this case we can  integrate by parts in $x_1$ and obtain
$$
|J_k(\xi)|\le C\, 2^{-k |\ka|}(1+2^{-k}\la(|s_1|+|S_2|))^{-1}),
$$
which again implies \eqref{7.10}, since here, by \eqref{7.9}, $|\xi|\lesssim \la(|s_1|+|S_2|).$ 
\medskip

{\bf 3. Case.} $|s_1|+|S_2|\ge C,$  with $C$ as above, and
 $|s_1|\sim |S_2|.$ 
\medskip
Observe first that $|s_1|\sim |S_2|$ implies  $|\xi_2|\sim 2^{\ka_1(m_1-1)k}|\xi_1|,$ so that
$$|\xi_2|>>|\xi_1|.
$$

We then write
$$
2^{-k}\la\Phi_k(x,s)=2^{-k}\la S_2 \,F(x,\si,\delta),
$$
where 
$$
F(x,\si,\delta):= \frac{s_1}{S_2} x_1+\tilde\psi(x_1,\de)+\si \Big(\phi_p\x+\tilde\phi_r(x_1,x_2,\de)+s_2x_2\Big) 
$$
and $ \delta:=2^{-k/r}<<1,\  \si:= \frac 1{S_2}$, so that 
$$| \frac{s_1}{S_2}|\sim 1, \ |\si|<<1.
$$
Observe that 
$$\Big|\pa_{x_1}^2\Big(\frac{s_1}{S_2} x_1+\tilde\psi(x_1,0)\Big)\Big|\sim 1
$$ for $x_1\sim 1.$  We also claim that the polynomial $P(x_2):=\phi_p(x_1^0,x_2)$has degree
\begin{equation}\label{7.13}
m:=\deg P\ge 2.
\end{equation}

For, otherwise, by the homogeneity of $\phi_p,$ the polynomial $\phi_p$ was of the form $\phi_p(x)=c_1x_1^n+c_2x_1^lx_2,$ where the point $(l,1)$ had to lie in the closed half-space above the bisectrix, since $\phi_p$ is the principal part of $\phi.$ Thus $l\le 1,$ so that $d(\phi)\le 1,$ in contradiction to our assumption $d(\phi)=h(\phi)\ge 2.$ 

From \eqref{7.13} we conclude that there is some integer $m\ge 2$ so that 
$$\Big|\pa_{x_2}^m \Big(\phi_p\x+s_2x_2\Big)\Big|\sim 1.
$$

If we now fix $x_1^0\sim 1$ and translate the $x_1$-coordinate by $x_1^0,$ we see that we can apply Proposition \ref{s8.1} if we localize our oscillatory integral $J_k$ to a small neighborhood of $(x_1^0,0)$ by introducing a suitable cut-off function into the amplitude, and  obtain an estimate of order 
$$
O(2^{-k |\ka|}(1+2^{-k}\la(|S_2|))^{-1/2}(1+2^{-k}\la)^{-1/m})
$$
for the corresponding localized integral, uniformly in $s_1$ and $s_2,$ since Proposition \ref{s8.1} also gives uniform estimates for small  perturbations of such parameters. Since we can decompose $J_k(\xi)$ by means of a suitable partition of unity into such localized oscillatory integrals, we see that 
$$
|J_k(\xi)|\le C\,2^{-k |\ka|}(1+2^{-k}\la(|S_2|))^{-1/2}(1+2^{-k}\la)^{-1/m},
$$
where $m\ge 2.$

a) If we assume that $|s_2|\le C$ for some fixed, large constant $C,$ then we have $|\xi_1|<<|\xi_2|<<|s_2\la|\le C|\la|, $ hence $|\xi|\sim \la, $ so that this estimate implies \eqref{7.10}.

\medskip
b) If $|s_2|>>1, $ then we proceed in a slightly different way. We first perform one integration by parts in $x_2,$ and then apply the method of stationary phase in $x_1.$ This leads to the estimate
$$
|J_k(\xi)|\le C\,2^{-k |\ka|}(1+2^{-k}\la(|S_2|))^{-1/2}(1+2^{-k}|s_2|\la)^{-1},
$$
If now $|\xi_2|\le \la, $ then $|\xi|\sim \la,$ and if $|\xi_2|\ge \la, $ then $|\xi|\sim |\xi_2|<<|s_2|\la, $ so that again  \eqref{7.10} follows.

\bigskip
In order to estimate the maximal operator $\M^{\rho_0},$ we observe that \eqref{7.10}  implies  that 
$$
|J_k(\xi)|\le C_\ve 2^{-k |\ka|}2^{k(\tfrac 12+\ve)}(1+|\xi|)^{-\tfrac 12-\ve}
$$
for every sufficiently small $\ve>0.$ We may therefore choose $A_{\chi_k}:= C_\ve 2^{-k|\ka|}2^{k(\tfrac 12+\ve)}$   for $\chi=\chi_k$ in Lemma \ref{s7.4}. Moreover, clearly we can choose $B_{\chi_k}:=C 2^{-k|\ka|},$ so that we have  
$$
||\M^{\chi_k}f||_p\le C_\ve 2^{-k(|\ka|-\frac 1p-\ve)},
$$
with a constant $C_\ve $ which is independent of $k.$ If $p>1/|\ka|,$ and if $\ve$ is chosen small enough, these estimates sum in $k,$ so that the maximal operator $\M^{\rho_0}$ is bounded on $L^p.$

\qed

\medskip

We shall indeed need a slight  extension of this result to the following situation. As before, we shall always assume that $x_1>0.$

\medskip
{\bf Definitions.} 
Let $q\in\NN^{\times}$ be a fixed positive integer. Assume that $\phi$ is a smooth function of the variables $x_1^{1/q}$ and $x_2$ near the origin, i.e., that there exists a smooth function $\phi^{[q]}$ near the origin such that $\phi(x)=\phi^{[q]}(x_1^{1/q},x_2).$ If the Taylor series of $\phi^{[q]}$ is given by 
$$
\phi^{[q]}\x\sim \sum_{j,k=0}^\infty c_{j,k}x_1^jx_2^k,
$$
then $\phi$ has the formal Puiseux series expansion
$$
\phi\x\sim \sum_{j,k=0}^\infty c_{j,k}x_1^{\frac jq}x_2^k.
$$
We therefore define the {\it Taylor-Puiseux support }of $\phi$ by
$$
\T(\phi):=\{(\tfrac jq,k)\in \NN_q^2: c_{jk}\ne 0\},
$$
where 
$$\NN_q^2:=(\tfrac 1q \NN)\times\NN.
$$
The {\it Newton-Puiseux  polyhedron} $\N(\phi)$ of $\phi$ at the origin is then defined to be the convex hull of the union of all the quadrants $(\tfrac jq,k)+\bR^2_+$ in $\bR^2,$ with $(\tfrac jq,k)\in\T(\phi).$  The associated {\it Newton-Puiseux diagram}  $\N_d(\phi)$  is the union of all compact faces  of the Newton-Puiseux polyhedron, and the notions of principal face, distance and homogenous distance are defined as in the case of Newton diagrams. The principal part $\pi(\phi)$ is analogously defined by 
$$
\phi_p(x):=\sum_{(\tfrac jq,k)\in\pi(\phi)}c_{j,k}x_1^{\frac jq}x_2^k.
$$

We shall then again decompose $\phi=\phi_p+\phi_r.$

\begin{cor}\label{s7.6}
Proposition \ref{s7.5} remains true even under the following weaker assumptions on $\psi$ and $\phi$ in place of Assumptions \ref{s7.3}, provided again that $\pa_2\phi_p(1,0)\ne 0:$ 

\bee
\item [(i)] $\psi$ is given by $\psi(x_1)=\sum_{l=1}^Lb_lx_1^{m_l} ,$ where $b_l\ne 0$ for $l=1,\dots,K,$ and where $2\le m_1<\dots<m_L$ are positive real numbers.
\item[(ii)] $\phi$  is a smooth function of the variables $x_1^{1/q}$ and $x_2$ as above, 
the principal face $\pi(\phi)$ is a compact edge, and the associated principal part $\phi_p$ of $\phi$ is $\ka$-homogeneous of degree one, where  $0<\ka_1<\ka_2<1$ and $a:=\frac {\ka_2}{\ka_1}> m_1;$ 
\item[(iii)] for the distance $d(\phi)=\frac1{|\ka|}$ we have  $d(\phi)\ge 2.$
\item [(iv)]  $\eta$ is a smooth bump function supported in a sufficiently small neighborhood $\Om$ of the origin.
\ee
\end{cor}

\proof All of our arguments extend in a straight-forward manner to this setting,  except perhaps for the proof of \eqref{7.13} and the straight-forward application of Lemma \ref{s7.4}.  However, if \eqref{7.13} was false in the present situation, then we could write $\phi_p(x)=c_1x_1^{n/q}+c_2x_1^{l/q}x_2.$   The point $(l/q,1)$ had to lie  above the bisectrix, since $\phi_p$ is the principal part of $\phi.$ Thus $l< q. $ Moreover, we would have
$\ka_1=\frac qn,\ \ka_2=1-\frac ln,$ so that 
$$|\ka|=1+\frac{q-l}n>1,$$
hence $d(\phi)<1,$ in contradiction to our assumption in (iii).
\medskip 

As for  Lemma \ref{s7.4}, notice that when applying the gradient to $J_k(\xi),$ the function $\eta$ will be multiplied with terms like $\phi$ or $\psi,$ which may not be smooth at $x_1=0,$ so that the argument in the proof of the lemma fails to hold. However, if we look at the formula for $J_k(\xi)$ after scaling the coordinates $x,$ we find that the factor $\eta(\delta_{2^{-k}}x)$ will have to be replaced, for instance, by   $\phi(\delta_{2^{-k}}x)\eta(\delta_{2^{-k}}x),$ where we now are in the domain where $x_1\sim 1, |x_2|\lesssim \ve_0.$
But, in this domain, the $C^n$-norms of such expressions are still uniformly bounded in $k,$ so that we obtain the same type of estimate as for $J_k(\xi).$ 

\qed

As a consequence of Proposition \ref{s7.5}, we see in particular that  Proposition \ref{s7.2} holds true in the case where $\pa_2\phi_p(1,0)\ne 0,$  since here $h(\phi)=1/|\ka|.$

\bigskip
\subsection{The case where $\phi_p$ is one of the exceptional polynomials \eqref{4.5} in Proposition \ref{s4.4}}\label{exception}

Corollary  \ref{s7.6} will be useful also in order to deal with the situation where $\tp_p$ is one of the exceptional polynomials $P$ in \eqref{4.5}, i.e., 
$$
P(y)=c(y_2^2-\la_1y_1^5)(y_2^2-\la_2y_1^5),
$$
 in Subsection \ref{subs6.4}. We were left with the domain \eqref{6exc} (in the original coordinates). Now, if we here put 
 $$\psi(x_1):=b_1 x_1^2\pm \sqrt{\tfrac{\la_1+\la_2}6}\, x_1^{5/2},$$ change coordinates as in \eqref{6.3} and call the new coordinates again $x,$ then the domain \eqref{6exc} will correspond to the domain 
$$
|x_2|\le \ve_0 x_1^{5/2}
$$
in the present context, and we have 
$$\phi_p(x)=P(x_1, x_2\mp  \sqrt{\tfrac{\la_1+\la_2}6}\, x_1^{5/2}).$$
Recall also that  $\la_1+\la_2>0,$  and that, in our present notation, the function $\phi_p$ is  $\ka$-homogeneous of degree one, with $\ka_1:=\frac 1{10}$ and $\ka_2:=\frac 1{4}.$  Moreover, $h(\phi)=d(\phi_p)=1/|\ka|=20/7,$ $K= 1, m_1=2$ and  $a=m_2=5/2.$ 

So, we are again just left with the estimation of the maximal operator $\M^{\rho_0}$ of the previous subsection. 
But, notice that 
$$\pa_2P(y)=4 y_2(y_2^2-\tfrac {\la_1+\la_2}2 y_1^5).$$
This shows that $\pa_2\phi_p(1,0)\ne 0,$ and clearly the assumptions in Corollary \ref{s7.6} are satisfied,  so that  $\M^{\rho_0}$ is indeed bounded on $L^p$ for $p>h(\phi)=20/7.$

\bigskip
\subsection{Further domain decompositions  in the case where  $\pa_2\phi_p(1,0)=0$}\label{subsec7.3}

We first observe that the Assumptions \ref{s7.3} imply in this case  that $\phi_p(1,0)\ne 0.$ 

\medskip
For otherwise the root $x_2=0$ had multiplicity $N$ at least 
$2.$ On the other hand, since the coordinates $x$ are adapted to $\phi,$ we must have $N\le h(\phi),$  so that \eqref{7.4} would fail to be true for $j=N.$

We  can thus write 
$$
\phi_p\x=x_2^BQ\x+cx_1^n,\ \mbox{with} \ c\ne 0 ,
$$
where $B\ge 1,$ and where $Q$ is a $\ka$-homogeneous polynomial such that 
$Q(x_1,0)=bx_1^q,\ b\ne 0,$ so that $Q(x_1,0)\ne 0$ for $x_1>0.$ Without loss of generality, we shall assume that $c=1.$ Recall also that we are in the domain
\begin{equation}\label{7.14}
|x_2|\le \ve_0x_1^a.
\end{equation}
Notice  $B\ge 2,$ since  $\pa_2\phi_p(1,0)= 0,$ and then our assumption \eqref{7.4}Ê implies that in fact  
\begin{equation}\label{7.13n}
B>h(\phi)\ge 2.
\end{equation}

 In order to understand the behavior of $\phi$ as a function of $x_2,$ for $x_1$ fixed, 
we shall decompose 
$$\phi\x=\phi(x_1,0)+\th\x,$$
 and 
write the complete phase $\Phi$ in the form
\begin{equation}\label{7.15}
\Phi(x,\xi)=(\xi_3\phi(x_1,0)+\xi_1 x_1+\xi_2\psi(x_1))\  +\ (\xi_3 \th\x +\xi_2x_2),
\end{equation}
Notice that 
\begin{equation}\label{7.16}
\phi(x_1,0)=x_1^n(1+O(x_1)), \ \psi(x_1)=b_1x_1^{m_1}(1+O(x_1)),\ \th_\ka\x=x_2^B Q\x,
\end{equation}
where $\th_\ka$ denotes the $\ka$-principal part  of $\th.$ 

Now, by means of the $\ka$-dilations we would like to  reduce our considerations as before  to the domain where $x_1\sim 1.$ In this domain, $|x_2|<<1,$ so that $\th_\ka(x)\sim x_2^BQ(x_1,0).$ What leads to problems is that the ''error term'' $\th_{\ka,r}:=\th-\th_\ka,$ which consists  of terms of higher $\ka$-degree than $\th_\ka,$ may nevertheless contain terms $c_j x_2^{l_j} x_1^{n_j}$ of lower $x_2$-degree $l_j<B.$ After scaling by $\de_{2^{-k}},$  so that then $x_1\sim 1$ and $|x_2|\lesssim \ve_0,$ these terms will have small coefficients compared to $x_2^B Q\x,$ but for $|x_2|$ very small they may nevertheless become dominant and have to be taken into account.

\bigskip

Consider now the Newton polyhedron $\N(\th).$ Since the Taylor support $\T(\th)$ arises from $\T(\phi)$ by removing all points on the $t_1$-axis, we have 
\begin{equation}\label{7.17}
\N(\pa_2\th)=(0,-1)+\N(\th).
\end{equation}
Moreover, if we put 
$$\ka^1:=\ka,\ a_1:=a=\ka^1_2/\ka^1_1,
$$
then the line $\ka^1_1t_1+\ka^1_2t_2=1$ contains the point $(q,B)$ of $\N(\th).$ This point is contained in the face 
$$\ga_1=[(A_0,B_0),(A_1,B_1)], \ \mbox{with}\ (A_1,B_1):=(q,B),
$$
of the Newton diagram $\N_d(\th)$ lying on this line. Note that possibly  
$(A_0,B_0)=(A_1,B_1).$

It is also clear from the construction of $\th$ from $\phi$ that 
\begin{equation}\label{7.18}
\N(\th)\cap\{t_2\ge B_1\}=\N(\phi)\cap\{t_2\ge B_1\}.
\end{equation}
\begin{figure}
\scalebox{0.7}{\input{figure1.pstex_t}}
\caption{\label{fig1}}
\end{figure}

We next describe a stopping time  algorithm oriented at the level sets of $\pa_2\th$ which will decompose our domain \eqref{7.14} in  a finite number of steps into subdomains, whose contributions to our maximal operator will be treated in different ways in the subsequent subsections. This algorithm will follow a similar line of thought as Varchenko's algorithm (compare \cite{ikromov-m}), and it will stop at latest when we have reached a domain containing only one root of $\pa_2\th$ (with multiplicity).

\bigskip

{\bf Case A.} $\N(\th)\subset\{t_2\ge B_1\}$

\medskip
\noi Then no term in $\th$ has higher $x_2$-exponent than $B_1=B,$ and we stop at this point.  
\medskip

{\bf Case B.} $\N(\th)$ contains points below the line $t_2= B_1.$

\medskip
\noi Then the Newton diagram $\N_d(\th)$ will contain a further edge 
$$\ga'_2=[(A_1,B_1),(A'_2,B'_2)]
$$
below the line  $t_2= B_1,$ lying, say, on the line  $\ka^2_1t_1+\ka^2_2t_2=1$  (compare figure \ref{fig1}). We then put
$$\ka^2:=(\ka^2_1,\ka^2_2), \ \ a_2:=a=\ka^2_2/\ka^2_1,\ \mbox{where clearly}\  a_2>a_1.
$$
Notice that $a_2\in\QQ.$ 
We then decompose the domain \eqref{7.14} Êinto the domains
$$E_1:=\{N_2 x_1^{a_2}< |x_2|\le \ve_1x_1^{a_1}\}
$$
and 
$$H_2:=\{ |x_2|\le N_2 x_1^{a_2}\},
$$
where $N_2$ will be any  sufficiently large constant and $\ve_1:=\ve_0.$  

In the domain $E_1,$ which is again domain of transition between two different homogeneities, we stop our algorithm. It will later  be  treated be means of bi-dyadic decompositions. 

\medskip
The $\ka^2$-homogeneous domain  $H_2$ will be further decomposed as follows: 
\medskip

\noi We first notice that the $\ka^2$-homogeneous part $(\pa_2\th)_{\ka^2}$ will be associated  to the edge $(0,-1)+\ga'_2=[(A_1,B_1-1),(A'_2,B'_2-1)]$ of the Newton diagram of $\thd$ and is $\ka^2$-homogeneous of degree $1-\ka_2^2.$  Observe also that in view of \eqref{7.17} we have $(\pa_2\th)_{\ka^2}=\pa_2(\th_{\ka^2}),$ which is a polynomial in the fractional power $x_1^{a_2}$ of $x_1$  and $x_2.$

Decomposing the polynomial $t\mapsto (\pa_2\th)_{\ka^2}(1,t)$ into linear factors and making use of the $\ka^2$-homogeneity of $(\pa_2\th)_{\ka^2},$ we see that we can  write it in the form
$$(\pa_2\th)_{\ka^2}(x)=c_2x_1^{A_1} x_2^{B'_2-1}\prod_{\al}(x_2-c_2^\al x_1^{a_2})^{n_2^\al},
$$
where
$$B_1=B'_2+\sum_\al n_2^\al,\quad A'_2=A_1+a_2\sum_\al n_2^\al,
$$
with roots $c_2^\al\in\CC\setminus\{0\}$ and multiplicities $n_2^\al\ge 1.$  
Let us assume in the sequel  that $N_2>>\max_{\al} |c_2^\al|.$ 

\medskip
By  $R_2$ we shall denote the set of all real roots $c_2^\al\in\RR,$ where we include also the root $d=0$ in case that $B'_2-1>0.$ 

\medskip
We shall  need to understand the behavior of the complete  phase function $\Phi(x,\xi)$ in display \eqref{7.15} on the domain $H_2.$  Now, after dyadic decomposition with respect to the $\ka^2$-dilation and re-scaling, we have to look at $\Phi(2^{-\ka^2_1k}x_1,2^{-\ka^2_2 k}x_2,\xi)$ in the domain where $x_1\sim 1$ and $|x_2|\le N_2.$ We write 
$$\Phi(2^{-\ka^2_1k}x_1,2^{-\ka^2_2 k}x_2,\xi)=2^{-\ka^2_1nk}\la \Phi_k(x,s),$$
where
\begin{eqnarray*}
\Phi_k(x,s)&:=&x_1^n(1+v_k(x_1))+s_1x_1+S_2b_1x_1^{m_1} (1+w_k(x_1))\\
&+& 2^{(\ka^2_1n-1)k}\Big(\th_{\ka^2}\x+\th_{r,k}\x+s_2x_2\Big)
\end{eqnarray*}
and again $\la:=\xi_3$ (assumed to be positive) and 
$$
s_1:=2^{\ka^2_1(n-1)k}\tfrac{\xi_1}{\la}, \  s_2:=2^{(1-\ka^2_2) k}\tfrac{\xi_2}{\la}, \ S_2:=2^{\ka^2_1(n-m_1)k}\tfrac{\xi_2}{\la}=2^{(\ka^2_1(n-m_1) +\ka^2_2 -1)k}s_2.
$$
The functions $v_k,w_k$ and $\th_{r,k}$ are of order $O(2^{-\delta k})$ in $C^\infty$ for some $\delta>0.$
 
 In the estimation of the corresponding oscillatory integral, the worst possible case arises when 
 $|s_1|\simÊ|S_2|\sim 1, $ so that 
\begin{equation}\label{7.19}
|s_2|\sim 2^{-(\ka^2_1(n-m_1) +\ka^2_2 -1)k}.
\end{equation}

Fix now an arbitrary $\ve_2>0.$ For any  point $d$ in the interval $[-N_2,N_2]$ denote by $D_2(d)$ the $\ka^2$-homogeneous domain (inside the half-plane $x_1>0$)
$$
D_2(d):=\{|x_2-dx_1^{a_2}|\le\ve_2x_1^{a_2}\}.
$$

\smallskip
Since we can cover the domain $H_2$ be a finite number of domains $D_2(d),$ it will be sufficient to examine the contribution of the domains $D_2(d).$ 
\medskip

{\bf Case B (a).} If $\ka^2_1(n-m_1) +\ka^2_2 \le 1,$ then we have $|s_2|\ge c>0$   in \eqref{7.19}. In this case, it will be possible to control the corresponding oscillatory integrals if $\ve_2$ is chosen sufficiently small, and we shall stop our algorithm with the domains $D_2(d).$ 

Indeed, if $\ka^2_1(n-m_1) +\ka^2_2 < 1,$ then $|s_2|>>1,$ which will allow for an  integration  by parts with respect to  $x_2$ as in the first case of the proof of Proposition \ref{s7.5}. The worst possible case will actually arise when 
$\ka^2_1(n-m_1) +\ka^2_2 = 1$ and when in addition  $d\notin R_2,$ i.e.,  $(\pa_2\th)_{\ka^2}(1,d)\ne 0,$ which will indeed lead to kind of {\it ''degenerate Airy-type''}   integrals.

\medskip
{\bf Case B (b).}  If $\ka^2_1(n-m_1) +\ka^2_2 > 1,$ then  $|s_2|<<1$ in \eqref{7.19}. 

\medskip
{\bf (i)} If $d\notin R_2,$ then  $(\pa_2\th)_{\ka^2}(1,d)\ne 0$ and $|s_2|<<1,$ so that again one  can integrate by parts with respect to  $x_2,$ and again the algorithm will stop.

\medskip
{\bf (ii)} Assume finally that $d\in R_2,$ so that   $(\pa_2\th)_{\ka^2}(1,d)=0$ and $|s_2|<<1.$ In this case, we introduce new coordinates 
$$
y_1:=x_1,\ y_2:=x_2-dx_1^{a_2},
$$
and denote our original functions, when expressed in the new coordinates $y,$ by a subscript $_{(2)},$ e.g.,
$$\phi_{(2)}(y):= \phi(y_1,y_2+dy_1^{a_2}).$$
$\th_{(2)}$ is defined by 
$$ \phi_{(2)}(y_1,y_2)=\phi_{(2)}(y_1,0)+\th_{(2)}(y),
$$
etc.. Notice that in general we don't have  $\th_{(2)}(y)= \th(y_1,y_2+dy_1^{a_2}),$ but 
$$
\pa_2\th_{(2)}(y)= \pa_2\th(y_1,y_2+dy_1^{a_2}).
$$

Notice that this $\ka^2$-homogeneous change of coordinates will have the effect on the Newton-polyhedron that the edge $\ga'_2=[(A_1,B_1),(A'_2,B'_2)]$ of $\N(\th)$  on the line $\ka^2_1t_1+\ka^2_2 t_2=1$ will be turned into a face  
$$\ga_2=[(A_1,B_1),(A_2,B_2)]
$$ of $\N(\th_{(2)})$  on the same line, with same left end-point $(A_1,B_1)$ but possibly different right endpoint $(A_2,B_2)$ (which may even agree with the left endpoint), where still $B_2\ge 1.$

Notice that $B_1\ge B_2,$ and that the domain $D_2(d)$ corresponds to the domain where $|y_2|\le \ve_2 y_1^{a_2}$ in the new coordinates $y.$ 

\bigskip

In the Case B(b)(ii), which is the only one  where our algorithm did not stop, we see that  by passing from $\phi=:\phi_{(1)}$  to $\phi_{(2)}$ and denoting the new coordinates $y$ again by $x,$ we have thus reduced ourselves to the smaller, $\ka^2$-homogeneous domain 
$$|x_2|\le \ve_2 x_1^{a_2}$$
in place of \eqref{7.14}.

We observe also that since the $\ka=\ka^1$-homogenous part of  our change of coordinates $y_1=x_1,\ y_2=x_2-dx_1^{a_2}$ is given by $x_1,x_2,$ i.e., by  the identity mapping, the Newton polyhedra of $\th_{(1)}$ and $\th_{(2)}$ will have the same $\ka^1$- principal faces and corresponding principal parts. This implies in particular that still
$$\phi_{(2)}(x_1,0)=x_1^n(1+O(x_1^{1/r}))$$ for some rational exponent $r>0.$
Moreover, since $a_2>a_1>m_1,$ also the new function $\psi_{(2)}(x_1):=\psi(x_1)+dx_1^{a_2},$ which corresponds to $\psi$ in the new coordinates,  will still satisfy
$$\psi_{(2)}(x_1)=b_1x_1^{m_1}(1+O(x_1^{1/r})).$$

Replacing $\phi$ by  $\phi_{(2)},$ we can now iterate this procedure. Notice that already the function $\phi_{(2)}$ will in general be only  a smooth function of $x_2$ and some fractional power of $x_1,$ so that from here on we shall have to work with Newton-Puiseux polyhedra  in place of Newton-polyhedra, etc.. 

\bigskip
\noi {\bf Example.} Let $\phi\x:=x_1^n+x_2^l+x_2x_1^{n-m_1}$ and $\psi(x_1):= x_1^{m_1},$ where $n/l>m_1\ge 2.$  The coordinates $\x$  are adapted to $\phi.$  
Notice also that in the original coordinates, say $(y_1,y_2)$, $\phi$  was  given by $(y_2-y_1^{m_1})^l+y_2y_1^{n-m_1}.$ Here
$$\phi(x_1,0)=x_1^n,\quad \th(x)=x_2^l+x_2x_1^{n-m_1},\quad \th_\ka(x)=x_2^l,$$
whereas 
$$\th_{\ka^2}(x)=x_2^l+x_2x_1^{n-m_1}.
$$
Thus, if $d:=0,$ we arrive at the ''degenerate  Airy type'' situation describes in Case B (a).

\bigskip

\noi{\bf Details on  and modification of the algorithm.} Suppose we have constructed in this way recursively a sequence 
$$\phi=\phi_{(1)},\phi_{(2)},\dots, \phi_{(L)}
$$
of functions, where $\phi_{(l)}$ is obtained from $\phi_{(l-1)}$ for $l\ge 2$ by means of a change of coordinates $y_1:=x_1,\ y_2:=x_2-d_lx_1^{a_l}$ (figure \ref{fig2}).
\begin{figure}
\scalebox{0.7}{\input{figure2.pstex_t}}
\caption{\label{fig2}}
\end{figure}

Then  $\phi_{(l)}$ arises from $\phi$ by the total change of coordinates $x=\vp_{(l)}(y),$ where 
$$
y_1:=x_1,\ y_2:=x_2-\sum_{j=2}^ld_jx_1^{a_j},
$$
i.e., $\phi_{(l)}=\phi\circ\vp_{(l)},$ and correspondingly $\th_{(l)}$ is defined by 
$$ \phi_{(l)}(y_1,y_2)=\phi_{(l)}(y_1,0)+\th_{(l)}(y),
$$
etc.. Notice that in general we don't have  $\th_{(l)}= \th\circ\vp_{(l)},$ but 
$$
\pa_2\th_{(l)}= \pa_2\th\circ\vp_{(l)}=\pa_2\phi\circ\vp_{(l)}.
$$
For the functions $\phi_{(l)}(x_1,0)$ and $\psi_{(l)}(x_1)=\psi(x_1)+\sum_{j=2}^ld_jx_1^{a_j}$ we then still have
\begin{equation}\label{}
\phi_{(l)}(x_1,0)=x_1^n(1+O(x_1^{1/r})), \ \psi_{(l)}(x_1)=b_1x_1^{m_1}(1+O(x_1^{1/r})),
\end{equation}
for some rational exponent $r>0.$

In each step, we produce a new face $\ga_l=[(A_{l-1},B_{l-1}),(A_l,B_l)]$ (possibly a single point) 
of $\N(\th_{(l)}),$ so that the Newton diagram  $\N_d(\th_{(l)})$ will in particular posses the faces 
$$\ga_1=[(A_{0},B_{0}),(A_1,B_1)], \dots, \ga_l=[(A_{l-1},B_{l-1}),(A_l,B_l)],
$$
where $B_l\ge 1.$  The Newton diagram  $\N_d( \th_{(l-1)})$ will have in addition a compact edge  $\ga'_l=[(A_{l-1},B_{l-1}),(A'_l,B'_l)],$ lying on a unique line 
$$\ka^l_1t_1+\ka^l_2 t_2=1,
$$
which contains also $\ga_l,$ 
such that $a_l=\frac{\ka^l_2}{\ka^l_1}.$ Moreover, $x_2=d_lx_1^{a_l}$ will be a real root of the 
$\ka^l$-homogeneous principal part  $\pa_2(\th_{(l-1)})_{\ka^{l}}$ of $\pa_2(\th_{(l-1)})$ corresponding to the edge $\ga'_l,$  i.e., $\pa_2(\th_{(l-1)})_{\ka^{l}}(1,d_l)=0,$ where  $ \pa_2(\th_{(l-1)})_{\ka^{l}}$ is a polynomial in a fractional power of $x_1$ and in $x_2.$
Moreover, 
$$B_1\ge B_2\ge \dots \ge B_l\ge 1 \ \mbox{and}\  m_1<a=a_1<a_2<\dots  <a_l.
$$

In particular, the descending sequence $\{B_l\}_l$ must eventually become constant (unless our algorithm stops already earlier).  

\medskip
 Our algorithm will always  stop after a finite number of steps, since eventually we will have $\ka^l_1(n-m_1) +\ka^l_2 \le 1,$ because $a_l\to \infty.$ This is evident from the geometry of the Newton-Puiseux polyhedra $\N(\th_{(l)})$.

\medskip
More precisely,  in  case that our algorithm did not terminate, then we could  find some minimal $L\ge 1$ such that $B_l=B_L$ for every $l\ge L.$ Then $B_L\ge 2,$ since for $B_L=1$  we had $\N_d( \th_{(L)})\subset \{t_2\ge B_L\},$ and we would stop. Moreover, from 
$1=\ka^l_1A_l+\ka^l_2 B_l\ge\ 2 \ka^l_2$ we conclude that $\ka^l_2\le 1/2.$

Next, we must have that $a_l\to \infty.$ For analytic $\phi,$ this follows easily from the Puiseux-series expansions of roots of $\pa_2\th.$ However, for sufficiently large $N,$ the points $(t_1,t_2)\in\N(\th)$ with $\ka_1t_1+\ka_2t_2>N$ will have no influence on the Newton-Puiseux diagrams of the functions $\th_{(l)}$ (compare the discussion in \cite{ikromov-m}), so that we can reduce the statement to the case of polynomials. This shows that 
$$
\ka^l_1(n-m_1) +\ka^l_2 =\ka^l_2(1+\frac{n-m_1}{a_l})\le\frac 12(1+\frac{n-m_1}{a_l})\le 1
$$
for $l$ sufficiently large, and so our algorithm would stop at this step.

\medskip
Let us therefore assume from now on that our algorithm terminates at step $l=L.$ 

\medskip
Next, in case that $B_l=B_{l+1}=\cdots=B_{l+j}$ for some $j\ge1,$ then we will modify our stopping time argument as follows:

We shall skip the intermediate steps and pass from $\phi_{(l)}$ to $\phi_{(l+j)}$ directly,  decomposing  in  the passage from $\phi_{(l)}$ to $\phi_{(l+j)}$ the domain $\{ |x_2|\le \ve_lx_1^{a_l}\}$ into the bigger transition domain
 $$E'_{l}:=\{N_{l+j} x_1^{a_{l+j}}< |x_2|\le \ve_lx_1^{a_l}\}
$$
and the $\ka^{l+j}$-homogeneous domain
$$H'_{l+j}:=\{ |x_2|\le N_{l+j} x_1^{a_{l+j}}\},
$$
where $N_{l+j}$ will be any  sufficiently large constant. 

\medskip
We may and shall therefore assume that the sequence  $\{B_l\}_l$ is strictly decreasing, until  our algorithm stops at step $L.$ In particular, we have $L<B_1,$ so that the number of all domains on which our algorithm will stop is finite (notice, however, that the domains arising in the course of the algorithm will depend on the choices of roots $d_j$  at every  step).
 The corresponding domains will  cover $\Om,$  so that it will suffice  to study the contributions to our maximal operator of these domains.  
\medskip

Now, when expressed in our original coordinates $x,$  then a domain on which  we stop our algorithm is  either a  transition domain 
$$E_l:=\{N_{l+1} x_1^{a_{l+1}}< |x_2-\sum_{j=2}^{l}d_jx_1^{a_j}|\le \ve_{l}x_1^{a_{l}}\},\quad 1\le l\le L,
$$
where the case $l=L$ arises only if $\N(\th_{(L)})$ is not contained in $ \{t_2\ge B_L\}$  - otherwise, when $\N(\th_{(L)})\subset \{t_2\ge B_L\},$
then we have to replace $E_{L}$ by the ''generalized'' transition  domain (which is at the same time $\ka^L$-homogeneous) 
$$E'_{L}:=\{ |x_2-\sum_{j=2}^{L}d_jx_1^{a_j}|\le \ve_{L}\, x_1^{a_{L}}\},$$
where formally $a_{L+1}=\infty$ (compare Case A);
or it is a domain
$$D_{l+1}(d):=\{ |x_2-\sum_{j=2}^{l}d_jx_1^{a_j}-d x_1^{a_{l+1}}|\le \ve_{l+1}x_1^{a_{l+1}}\},\quad 1\le l\le L,
$$
which is  $\ka^{l+1}$-  homogeneous after applying the change of coordinates $x=\varphi_{(l)}(y),$ 
where $|d|\le N_{l+1},$ and where  $\ka^{l+1}_1(n-m_1) +\ka^{l+1}_2 \le 1,$ in case that $d=d_{l+1}$ is a real root of  $\pa_2(\th_{(l)})_{\ka^{l+1}}(1,\cdot).$ The case $l=L$ can here only arise if  $\N_d(\phi_{(L)})$ is not contained in $ \{t_2\ge B_L\},$ and   if 
$\ka^{L+1}_1(n-m_1) +\ka^{L+1}_2 >1,$ then there is no real root of $\pa_2(\th_{(L)})_{\ka^{L+1}}(1,\cdot).$

\bigskip

The contribution to the oscillatory integral $J^{\rho_0}$ of a domain $E_l,$ after changing  to the coordinates $y$ given by $\vp_{(l)}$ in the integral,  can be put into the form 
$$
J^{\tau_l}(\xi):= \int_{\RR^2_+} e^{i\Phi_{(l)}(y,\xi)}\tilde\eta(y) \tau_l(y)\, dy,
$$
where we put
$$
\tau_l(y):=\rho\Big(\frac{ y_2}{\ve_{l}y_1^{a_{l}}}\Big)\, (1-\rho)
\Big(\frac{ y_2}{N_{l+1}y_1^{a_{l+1}}}\Big), 
$$
if  $\N(\th_{(l)})$ is not contained in $ \{t_2\ge B_{l}\},$  respectively 
$$\tau_{l}(y):=\rho\Big(\frac{ y_2}{\ve_{l}y_1^{a_{l}}}\Big),
$$
 if  $\N(\th_{(l)})\subset \{t_2\ge B_{l}\};$ of course, this will here only be possible for $l=L$ and will then correspond to  the domain $E'_L.$
Here, 
$$
\Phi_{(l)}(y,\xi):=(\xi_3\phi_{(l)}(y_1,0)+\xi_1 y_1+\xi_2\psi_{(l)}(y_1))\  +\ (\xi_3 \th_{(l)}(y_1,y_2) +\xi_2y_2).
$$
Similarly, the contribution of a domain $D_{l+1}(d)$ is of the form
$$
J^{\rho_{l+1}}(\xi):= \int_{\RR^2_+} e^{i\Phi_{(l)}(y,\xi)}\tilde\eta(y) \rho_{l+1}(y_1,y_2-dy_1^{a_{l+1}})\, dy,
$$
where 
$$
\rho_{l+1}(y):=\rho\Big(\frac{ y_2}{\ve_{l+1}y_1^{a_{l+1}}}\Big).
$$

\bigskip
At this point, it will again be helpful to defray the notation by writing $\phi$ in place of  $\phi_{(l)},$  $\psi$ in place of  $\psi_{(l)}$ etc.,  and assuming that $\phi,\psi$ and $\th$ satisfy the following assumptions on $\RR^2_+:$ 

\begin{assumptions}\label{s7.7}
{\rm The functions $\phi$  and $\eta$ are smooth functions of $x_1^{1/r}$ and $x_2,$ and 
$\psi$ is a smooth function of $x_1^{1/r},$ where $r$ is a positive integer. If we write  $\phi\x=\phi(x_1,0)+\th\x,$ then the following hold true: 

\bee
\item [(i)] The Newton diagram $\N_d(\th)$ contains at least the faces 
$$\ga_1=[(A_{0},B_{0}),(A_1,B_1)], \dots, \ga_l=[(A_{l-1},B_{l-1}),(A_l,B_l)],
$$
where $B_1>B_2>\cdots >B_l,$ so that $\ga_j$ is an edge, if $j>1,$ and $B_1>h(\phi)\ge 2,$ 
and in case that $\N(\th)$ is not contained in $\{t_2\ge B_l\},$ it contains the additional edge 
$\ga'_{l+1}=[(A_{l},B_{l}),(A'_{l+1},B'_{l+1})].$ The face $\ga_j$ lies on the line $\ka^j_1t_1+\ka^j_2t_2=1,$ where $\ka^1=\ka.$ Putting $a_j:=\frac{\ka^j_2}{\ka^j_1},$ we have 
$$a=a_1<\cdots<a_j<a_{j+1}<\cdots.$$
\item [(ii)] We have 
$$
\phi(x_1,0)=x_1^n(1+O(x_1^{1/r})), \ \psi(x_1)=b_1x_1^{m_1}(1+O(x_1^{1/r})),
$$
where $n=1/\ka_1>\ka_2/\ka_1=a>m_1\ge2.$ 
\ee}
\end{assumptions}

\bigskip
With these data, we define the phase function
$$
\Phi(x,\xi):=(\xi_3\phi(x_1,0)+\xi_1 x_1+\xi_2\psi(x_1))\  +\ (\xi_3 \th\x +\xi_2x_2),
$$
and the oscillatory integrals
$$
J^{\tau_l}(\xi):= \int_{\RR^2_+} e^{i\Phi(x,\xi)}\eta(x) \tau_l\x\, dx,
$$
and 
$$
J^{\rho_{l+1}}(\xi):= \int_{\RR^2_+} e^{i\Phi(x,\xi)}\eta(x) \rho_{l+1}(x_1,x_2-dx_1^{a_{l+1}})\, dx,
$$
where again $\eta$ denotes a smooth bump function supported in a sufficiently small neighborhood $\Om$  of the origin and $\tau_l$ and $\rho_{l+1}$ are defined as before, only with $\th_{(l)}$ replaced by $\th.$

The maximal operators corresponding to the Fourier multipliers $e^{i\xi_3}J^{\tau_l}$ and 
$e^{i\xi_3}J^{\rho_{l+1}}$ will again be denoted by $\M^{\tau_l}$ and  $\M^{\rho_{l+1}},$ respectively. 

\medskip
In view of our previous  discussion, and since we had $h(\phi_{(l)})=\frac 1{|\ka|},$  what remains to be proven is the following 

\begin{prop}\label{s7.8} 
 Assume that  the neighborhood $\Om$ of the point $(0,0)$ is chosen sufficiently small. Then the following hold true:
 
 (a) The maximal operator
$\M^{\tau_l}$   is bounded on $L^p(\RR^3)$   for every  $p>\frac 1{|\ka|}.$

(b) The maximal operator $\M^{\rho_{l+1}}$   is bounded on $L^p(\RR^3)$   for every  $p>\frac 1{|\ka|},$ provided that $\ka^{l+1}_1(n-m_1) +\ka^{l+1}_2 \le 1$ in case that $d=d_{l+1}$ is a real root of  $\pa_2\th_{\ka^{l+1}}(1,\cdot).$

\end{prop}

The proof will make use of estimates for oscillatory integrals with small parameters which will be given in  the next section.

\bigskip

\setcounter{equation}{0}
\section{Proof of Proposition \ref{s7.8}}\label{proof}

\bigskip
\subsection{Estimation of $J^{\tau_l}$}\label{subsec8.1}

Let us first assume that  $\N(\th)$ is not contained in $ \{t_2\ge B_{l}\},$ so that
 $\tau_l(x):=\rho\Big(\frac{ x_2}{\ve_{l}x_1^{a_{l}}}\Big)\, (1-\rho)
\Big(\frac{ x_2}{N_{l+1}x_1^{a_{l+1}}}\Big).$  Arguing in a similar way as in Subsection \ref{tau}, we then consider a dyadic partition of unity 
$\sum_{k=0}^\infty \chi_k(s)=1,\  (0<s<1)$ on $\RR,$ with $\chi \in C_0^\infty(\RR)$ supported in the interval $ [1/2,4],$ 
where $\chi_k(s):=\chi(2^ks),$ and put again
$$
\chi_{j,k}(x):=\chi_j(x_1)\chi_k(x_2),\ j,k\in\NN.
$$
Then 
 
 \begin{equation}\label{9.1}
J^{\tau_l}=\sum_{j,k}J_{j,k},
\end{equation}
 where 
 \begin{eqnarray*}
J_{j,k}(\xi)&:=&\int_{\RR^2_+} e^{i\Phi(x,\xi)}\eta(x) \tau_l(x)\,\chi_{j,k}(x)\, dx\\
&=& 2^{-j-k}\int_{\RR^2_+} e^{i\Phi_{j,k}(x,\xi)}\eta_{j,k}(x)\,\chi \otimes\chi(x)\, dx,
\end{eqnarray*}
with $\Phi_{j,k}(x,\xi):=\Phi(2^{-j}x_1,2^{-k}x_2,\xi),$ and where the functions $\eta_{j,k}$ are uniformly bounded in $C^\infty.$ The summation in \eqref{9.1} takes place over pairs $(j,k)$ satisfying
\begin{equation}\label{9.2}
a_lj+M\le k\le a_{l+1}j-M,
\end{equation}
where $M$ can still be choosen sufficiently large, because we had the freedom to choose $\ve_l$ sufficiently small and $N_{l+1}$ sufficiently large. In particular, we have $j\sim k.$

Moreover, our Assumptions \ref{s7.7} on the Newton diagram of $\th$ imply exactly as in Subsection \ref{tau} that 
$$\th_{j,k}(x)=2^{-(A_l j+B_l k)}\Big(c_lx_1^{A_l} x_2^{B_l}+O(2^{-C M})\Big),
$$
for some constants $c_l\ne 0$ and $C>0.$ Notice also that $B_l>B_{l+1}\ge 1$ here, so that $B_l\ge 2,$ and that we are here only interested in the domain where
$$x_1\sim 1\sim  x_2.$$ 
In combination with our further assumptions in \ref{7.7}, we thus obtain

\begin{eqnarray*}
\Phi_{j,k}(x,\xi)&=&2^{-jn}\xi_3 x_1^n(1+v_{j,k}(x_1))+2^{-jm_1}\xi_2 b_1 x_1^{m_1}(1+w_{j,k}(x_1))+2^{-j}\xi_1 x_1\\
&+& 2^{-(A_l j+B_l k)}\xi_3\Big(c_lx_1^{A_l} x_2^{B_l}+u_{j,k}\x)\Big) +2^{-k}\xi_2x_2,
\end{eqnarray*}
where the functions $v_{j,k}, w_{j,k}$ and $u_{j,k}$ are of order $O(2^{-\delta (j+k)})$  respectively  $O(2^{-\delta M})$ in $C^\infty$ for some $\delta>0.$ 

\begin{remark}\label{s9.1}
More precisely, the functions $v_{j,k}, w_{j,k}$ and $u_{j,k}$ depend smoothly on the small parameters $\delta_1:=2^{-j/r}$ and $\delta_2:=2^{-k}$ respectively $\delta_3:=2^{-M}$ and vanish identically for $\delta_1=\delta_2=0$ respectively $\delta_3=0.$
\end{remark}

Assuming again without loss of generality that $\la:=\xi_3>0,$ we may thus write 
$$\Phi_{j,k}(x,\xi)=2^{-jn}\la F_{j,k}(x,s,\si),
$$
with
\begin{eqnarray*}
F_{j,k}(x,s,\si)&:=& x_1^n(1+v_{j,k}(x_1))+S_2  x_1^{m_1}(1+w_{j,k}(x_1))+s_1 x_1\\
&+& \si\Big(c_lx_1^{A_l} x_2^{B_l}+u_{j,k}\x +s_2x_2\Big),
\end{eqnarray*}
and 
\begin{equation}\label{9.3}
s_1:=2^{(n-1)j}\tfrac{\xi_1}{\la}, \  s_2:=2^{A_lj+(B_l-1) k}\tfrac{\xi_2}{\la}, \ S_2:=2^{(n-m_1)j}b_1\tfrac{\xi_2}{\la},\ \si=\si_{j,k}:=2^{nj-A_lj-B_lk}.
\end{equation}
\medskip

\begin{lemma}\label{s9.2}
Under Assumptions \eqref{s7.7}, the following hold true:
\bee
\item[(a)] The sequence $\{\frac 1{\ka_1^m}\}_m$ is increasing  and the sequence $\{\frac 1{\ka_2^m}\}_m$ is decreasing .
\item[(b)] For $j,k$ satisfying \eqref{9.2} we have 
$$
\frac j{\ka^l_1} << A_lj+B_l k << \frac k{\ka^l_2}.
$$
In particular,
$$
\frac j{\ka_1}=nj << A_lj+B_l k << \frac k{\ka _2}.
$$
\ee
\end{lemma}

\proof
(a) is evident from the geometry of the Newton diagram of $\th.$ It follows also from the identity (4.4) in \cite{ikromov-m}, according to which 
\begin{eqnarray*}
\frac 1{\ka^m_2}&=& \frac{A_m}{a_m}+B_m=\frac{A_{m-1}}{a_m}+B_{m-1},\\
\frac 1{\ka^m_1}&=& A_m+a_m B_m=A_{m-1}+a_m B_{m-1},
\end{eqnarray*}
since the sequence $\{a_m\}_m$ is increasing.

(b) is a consequence of (a) and the identities above.
\qed

Since $B_l\ge 1$ and $n>m_1\ge2,$ in combination with Lemma \ref{s9.2} we see that 

\begin{eqnarray}\label{9.4}
\si<<1,\ |\xi_1|<<\la|s_1|,\ |\xi_2|<<\la|s_2|,\ \mbox{and also }\ |\xi_2|<<\la|S_2|.
\end{eqnarray}

\medskip
\begin{prop}\label{s9.3}
If $M$ in \eqref{9.2} is chosen sufficiently large,  then the following estimate
\begin{equation}\label{9.5}
|J_{j,k}(\xi)|\le C ||\eta ||_{C^3(\RR^2)} 2^{-j-k}(1+2^{-n j}|\xi|)^{-1/3} (1+2^{-nj}\si_{j,k}|\xi|)^{-1/2}
\end{equation}
holds true, where the constant $C$ does not depend on $j,k$ and $\xi.$ 

\smallskip
Consequently, the maximal operator $\M^{\tau_l}$ is bounded on $L^p(\RR^3)$ for every $p>1/|\ka|.$
\end{prop}

\proof
We first notice that $B_l\ge 2,$  so that 
$\pa_2^2(x_1^{A_l} x_2^{B_l})\sim 1.$ 
\medskip

As in the proof of Proposition \ref{s7.5} we shall distinguish several cases.

\medskip

{\bf 1. Case.} $|s_1|+|S_2|<< 1,$ or $|s_1|+|S_2|>> 1$ and $|s_1|<<|S_2|$ or 
$|s_1|>>|S_2|.$
\medskip

Here, an integration by parts in $x_1$ yields 
$$
J_{j,k}(\xi)=O(2^{-j-k} (1+2^{-nj}\la(1+|s_1|+|S_2|))^{-1}),
$$
which implies \eqref{9.5} because of \eqref{9.4}. 

\medskip

{\bf 2. Case.} $|s_1|+|S_2|>>1$  and  $|s_1|\sim |S_2|.$
\medskip

Since $m_1\ge 2,$ we have  $\pa_1^2(x_1^{m_1})\sim 1.$  Therefore, if $s_2$ with $|s_2|\lesssim 1$ is fixed, in view of Remark \ref{s9.1} we can apply Proposition \ref{s8.1} in a similar way as in the proof of Proposition \ref{s7.5}, with $\la$ replaced by $2^{-nj}\la (|s_1|+|S_2|),$  and obtain 
\begin{equation}\label{9.6}
|J_{j,k}(\xi)|\le C 2^{-j-k} (1+2^{-nj}\la(1+|s_1|+|S_2|))^{-1/2}(1+2^{-nj}\si\la (1+|s_2|))^{-1/2}.
\end{equation}

In fact, the proposition even shows that this estimate remains valid under small perturbations of $s_2,$ so that we can choose the constant $C$ uniformly for  $s_2$ in a fixed, compact interval. 

On the other hand, if $|s_2|>>1,$ we can obtain the even stronger estimate where the second exponent $-1/2$ is replaced by $-1$ by first integrating by parts in $x_2$ and then applying the method of stationary phase in $x_1.$ 

\bigskip

Observe at this point that if $|\xi_1|+|\xi_2|\ge \la,$ so that $|\xi|\sim |\xi_1|+|\xi_2|,$ then by \eqref{9.4} 
$$
|s_1|+|S_2|>>1.
$$
Notice also that $|s_1|\sim |S_2|$ implies, by \eqref{9.3}, that $1\sim 2^{-(m_1-1)j}|\xi_2|/|\xi_1|,$ hence 
$$
|\xi_1|<<|\xi_2|.
$$

Thus, if $|\xi_1|+|\xi_2|\ge \la$ and $|s_1|\sim |S_2|,$ then $|\xi|\sim |\xi_2|,$ and since
 $|s_2|\la>>|\xi_2|,$ we see that \eqref{9.6}  implies \eqref{9.5} in this case, as well as of course in  the case where $|\xi_1|+|\xi_2|\le \la.$ 
 We are thus left with the case

\medskip

{\bf 3. Case.} $|s_1|+|S_2|\sim 1$  and  $|\xi_1|+|\xi_2|\le \la,$ hence $|\xi|\sim \la.$
\medskip

Since  $n>m_1,$ it is easy to see that in this case the polynomial $p(x_1):=x_1^n+S_2b_1x_1^{m_1}+s_1x_1$ satisfies $|p''(x_1)|+|p'''(x_1)|\ne 0$ for every $x_1\sim 1.$ Therefore, if we fix some point $x_1^0\sim 1,$ then  we can either apply 
Proposition  \ref{s8.1}Ê or Proposition  \ref{s8.2} if we localize the oscillatory integral $J_{j,k}(\xi)$ by means of a suitable cut-off function to a small neighborhood of $x_1^0$ and translate coordinates, and finally obtain by means of a suitable partition of unity  in a similar way as in the previous case that 
$$
|J_{j,k}(\xi)|\le C 2^{-j-k} (1+2^{-nj}\la)^{-1/3}(1+2^{-nj}\si\la (1+|s_2|))^{-1/2},
$$
hence \eqref{9.5}. Note again that this argument first applies for fixed $s_1,s_2,S_2,$ but since Propositions  \ref{s8.1} and  \ref{s8.2} allow for small perturbations of parameters, the estimate above will hold uniformly in $s_1,s_2,S_2.$

\medskip
Next, observe that we may replace the factor  $(1+2^{-nj}\si_{j,k}|\xi|)^{-1/2}$ in \eqref{9.5}   by $(1+2^{-nj}\si_{j,k}|\xi|)^{-1/6-\ve},$ for any  sufficiently small  $\ve>0,$ which leads to 
\begin{eqnarray*}
|J_{j,k}(\xi)|&\le& C\ ||\eta ||_{C^3(\RR^2)} 2^{-j-k}2^{\frac{nj}3}2^{(A_l j+B_l k)(\frac 16 +\ve)} (1+|\xi|)^{-\frac 12 -\ve}\\
&\le& C\ ||\eta ||_{C^3(\RR^2)} 2^{-j-k}2^{\frac{j}{3\ka_1}}2^{\frac k{\ka_2}(\frac 16 +\ve)} (1+|\xi|)^{-\frac 12 -\ve},
\end{eqnarray*}
 since  Lemma \ref{s9.1} shows that  $A_l j+B_l k<\frac k{\ka_2^l}\le\frac k{\ka_2}.$ 
 
Lemma \ref{s7.4}  then implies that   the maximal operators $\M^{j,k}$ associated to the multipliers $J_{j,k}$ can be estimated by 
$$
||\M^{j,k}f||_p\le C2^{-j-k}2^{\frac{2j}{3\ka_1 p}}2^{\frac k{\ka_2p}(\frac 13 +\ve)}||f||_p
$$
for every sufficiently small $\ve>0$ and $p\ge 2.$

Observe that for $p=\frac 1{|\ka|},$ we have 
$$
\frac{2}{3\ka_1 p}=\tfrac 23\frac{\ka_1+\ka_2}{\ka_1}=\tfrac 23 (1+a)>1,
$$
so that for $p>\frac 1{|\ka|}$ sufficiently close to $\frac 1{|\ka|},$ we have 
\begin{eqnarray}
&&\sum_{a_l j+M\le  k}2^{-j-k}2^{\frac{2j}{3\ka_1 p}}2^{\frac k{\ka_2p}(\frac 13 +\ve)} 
\le  \sum_{ j\le  \frac ka,\,  k\ge M}2^{-j-k}2^{\frac{2j}{3\ka_1 p}}2^{\frac k{3\ka_2p}+\ve}\nonumber\\
&\le& \sum_{k\ge M}2^{(\tfrac 23 (1+a)-\delta-1)\frac ka -k+\frac{\ka_1+\ka_2}{3\ka_2}k+\ve k}=\sum_{k\ge M}2^{(\ve-\frac {\delta}a)k},\label{9.7}
\end{eqnarray}
where $\delta>0$ depends on $p.$ Choosing $\ve$ sufficiently small, this series converges, so that $\M^{\tau_l}$ is bounded on $L^p.$  For $p=\infty,$ the series converges as well. By  real interpolation, we thus find that $\M^{\tau_l}$ is $L^p$-bounded for every $p>\frac 1{|\ka|}.$

\qed

\bigskip

The case where $\N(\th)\subset \{t_2\ge B_{l}\}$ can be treated in a very similar way, if we formally replace $a_{l+1}$ by $+\infty.$ Indeed, in this case we have $\tau_l(x):=\rho\Big(\frac{ x_2}{\ve_{l}x_1^{a_{l}}}\Big)$, so that condition \eqref{9.2} has to be replaced by 
\begin{equation}\label{9.8}
a_lj+M\le k.
\end{equation}
Moreover, in this case we obviously have 
$$\th_{j,k}(x)=2^{-(A_l j+B_l k)}x_2^{B_l}\Big(c_lx_1^{A_l} +O(2^{-\delta (j+k)})\Big)
$$
for some $\delta>0.$ Therefore, if $B_l\ge 2,$ we can argue exactly as before and see that Proposition \ref{s9.3} remains valid (notice that in \eqref{9.7} we only made use of  \eqref{9.8}).

\bigskip  

What remains open at this stage is the case where $B_l=1.$ It turns out that  here the oscillatory integrals  $J_{j,k}(\xi)$ may possibly   be of degenerate Airy type. We shall then need more detailed information, which we shall obtain be regarding $\M^{\tau_l}$ rather as a maximal operator of type $\M^{\rho_l},$ which will be treated in the next subsection.

\bigskip
\subsection{Estimation of $J^{\rho_{l+1}}$}\label{subsec8.2}

We now consider the maximal operators $\M^{\rho_{l+1}}$ in Proposition \ref{s7.8} (b). It will here be convenient to change to the $\ka^{l+1}$-homogeneous  coordinates  
$$
y_1:=x_1, y_2:= x_2-dx_1^{a_{l+1}}.
$$
This change of coordinates  has the effect that we  can assume that $d=0.$ The Newton diagram of $\th$ in the new coordinates will still contain the faces $\ga_1,\dots,\ga_l,$ but the edge  $\ga'_{l+1}=[(A_{l},B_{l}),(A'_{l+1},B'_{l+1})]$  may change to an interval $[(A_{l},B_{l}),(A_{l+1},B_{l+1})] $ on the same line  $\ka^{l+1}_1t_1+\ka^{l+1}_2t_2=1,$ but possibly with a different right endpoint  $(A_{l+1},B_{l+1})$, which may even coincide with the left endpoint
 $(A_{l},B_{l}).$ 
 
 \medskip
 Simplifying the notation by writing $\ka':=\ka^{l+1}$ and $a':=\frac {\ka'_2}{\ka'_1}=a_{l+1},$  we shall then have to estimate the oscillatory integral $J(\xi)= J^{\rho_{l+1}}(\xi),$ with
\begin{equation}\label{9.9}
J(\xi):= \int_{\RR^2_+} e^{i\Phi(x,\xi)}\eta(x) \rho\Big(\frac {x_2}{\ve' x_1^{a'}}\Big)\, dx,
\end{equation}
corresponding to the domain 
 $$
 |x_2|\le \ve' x_1^{a'},
 $$
 where $\ve'=\ve_{l+1}>0$ can still be chosen as small as we like,
under one of the following assumptions:

\bee
\item[(i)]  $\pa_2\th_{\ka'}(1,0)=0,$ i.e., $B_{l+1}\ge 2,$ and $\ka'_1(n-m_1)+\ka'_2\le 1.$

\item[(ii)]   $\pa_2\th_{\ka'}(1,0)\ne 0,$ i.e, $B_{l+1}=1,$ and $\ka'(n-m_1)+\ka'_2\ne1.$

\item[(iii)]   $\pa_2\th_{\ka'}(1,0)\ne 0,$ i.e, $B_{l+1}=1,$ and  $\ka'(n-m_1)+\ka'_2=1.$
\ee

The most delicate case is case (iii), which will lead to degenerate Airy-type integrals. Notice that the second condition in (iii) just means that the point $(n-m_1,1)=(A_{l+1},B_{l+1})$ belongs to $\N_d(\th).$

We shall denote the maximal operator associated to the Fourier multiplier $e^{i\xi_3}J(\xi)$ by $\M'.$ 

\medskip 
Observe at this point that the oscillatory integral $J^{\tau_l}$ for the still open  case where $B_l=1$ can be written in the form \eqref{9.9} too, with $\ka':=\ka^l,$ hence  $a'=a_l$ and $(A_{l},B_{l})=(A_{l+1},B_{l+1}) ,$  and since $B_l=1,$ it will satisfy the assumption (ii) or (iii).
Notice that here necessarily $l>1.$ 

\medskip 
We shall therefore in the sequel relax the condition $a'>a_l$ and assume only that $a'\ge a_l$ in case that $(A_{l},B_{l})=(A_{l+1},B_{l+1})$ and $B_l=1.$ Then, as in the proof of Proposition \ref{s7.5}, we can decompose 
\begin{equation}\label{9.10}
J=\sum_{k=k_0}^\infty J_k
\end{equation}
 by means of a dyadic decomposition based on the $\ka'$-dilations $\delta'_r\x:=(r^{\ka'_1}x_1,r^{\ka'_2}x_2),$ where the dyadic constituent $J_k$ of $J$ is given, after re-scaling, by

$$
J_k(\xi)=2^{-k |\ka'|} \int_{\RR^2} e^{i  2^{-\ka'_1nk}\la \Phi_k(x,s)}\rho\Big(\frac{ x_2}{\ve'x_1^{a}}\Big)\eta(\delta'_{2^{-k}}x)\chi(x)\, dx,
$$
where  again $\la:=\xi_3$ is assumed to be positive, and where 
\begin{eqnarray*}
\Phi_k(x,s,\si)&:=&x_1^n(1+v_k(x_1))+s_1x_1+S_2b_1x_1^{m_1} (1+w_k(x_1))\\
&+& \si\Big(\th_{\ka'}\x+\th_{r,k}\x+s_2x_2\Big),
\end{eqnarray*}
with 
\begin{equation}\label{9.11}
s_1:=2^{\ka'_1(n-1)k}\tfrac{\xi_1}{\la}, \  s_2:=2^{(1-\ka'_2) k}\tfrac{\xi_2}{\la}, \ S_2:=2^{\ka'_1(n-m_1)k}\tfrac{\xi_2}{\la}, \ \si=\si_k:=2^{(\ka'_1n-1)k}.
\end{equation}
In particular, we have 
\begin{equation}\label{9.12}
S_2=2^{(\ka'_1(n-m_1) +\ka'_2 -1)k}s_2.
\end{equation}
Moreover, since $\ka'_1< \ka_1=1/n,$ we have $\ka'_1(n-1)>0$ and $\ka'_1n-1< 0,$  and since $1=\ka'_1A_{l+1}+\ka'_2B_{l+1}\ge \ka'_2,$ we have $1-\ka'_2>0,$  we see that if $\Om$ is chosen sufficiently small so that  $k_0>>1$ in  \eqref{9.10}, then 
\begin{eqnarray}\label{9.13}
|\si|<< 1,\ |\xi_1|<<\la|s_1|,\ |\xi_2|<<\la|s_2|,\ \mbox{and also }\ |\xi_2|<<\la|S_2|.
\end{eqnarray}

Recall that  $\th_{\ka'}$ denotes the $\ka'$-homogeneous part of $\th.$ 
The functions $v_k,w_k$ and $\th_{r,k}$ are of order $O(2^{-\ve k})$ in $C^\infty$ for some $\ve>0,$ and can in fact be viewed as smooth functions  $v(x_1,\delta),w(x_1,\delta)$ respectively  $\th_{r}(x,\delta)$ depending also on the small parameter $\delta=2^{-k/r}$ for some positive integer $r>0,$ which vanish identically when $\delta=0.$

Notice again that in our  domain of integration for $J_k(\xi),$  we have 
$$
x_1\sim 1, \ |x_2|\lesssim \ve',
$$
and clearly $|\M'f|\le \sum_{k=k_0}^\infty |\M^k f|,$
if $\M^k$ denotes the maximal operator associated to the Fourier multiplier $e^{i\xi_3}J_k(\xi).$
\medskip

The following proposition will then cover Proposition \ref{s7.8}(b) as well as the remaining case of Proposition \ref{s7.8}(a). The constants $l_m$ and $c_m$ will be as in Theorem  \ref{8.3}. We remark at this point that clearly 
\begin{equation}\label{9.14}
1/6\le l_m <1/4.
\end{equation}

\begin{prop}\label{s9.4}
If $k_0$ in \eqref{9.10} is chosen sufficiently large and  $\ve'$  sufficiently small,  then
\begin{equation}\label{9.15}
|J_{k}(\xi)|\le C ||\eta ||_{C^3(\RR^2)} 2^{-|\ka'|k}\sigma_k^{-(l_m+c\ve)}(2^{-\ka'_1nk}|\xi|)^{-1/2-\ve} \end{equation}
for some $m\in\NN$ with $2\le m\le B_l,$ some constant $c>0$ and every sufficiently small  $\ve>0,$ where the constant $C$ does not depend on $k$ and $\xi.$ 

\smallskip
Consequently, the maximal operator $\M'$ is bounded on $L^p(\RR^3)$ for every $p>1/|\ka|.$
\end{prop}

\proof
We proceed in a similar way as in the proof of Proposition \ref{s9.3}.
\medskip

{\bf 1. Case.} $|s_1|+|S_2|<< 1,$ or $|s_1|+|S_2|>> 1$ and $|s_1|<<|S_2|$ or 
$|s_1|>>|S_2|.$
\medskip

Here, an integration by parts in $x_1$ yields 
$$
|J_{k}(\xi)|\le C\,2^{-|\ka'|k} (1+2^{-\ka'_1nk}\la(1+|s_1|+|S_2|))^{-1},
$$
which implies \eqref{9.15} because of \eqref{9.13}. 

\medskip

{\bf 2. Case.} $|s_1|+|S_2|>>1$  and  $|s_1|\sim |S_2|.$
\medskip

Observe first that for any $x_1^0\sim 1,$ the polynomial  $P(x_2):=\th_{\ka'}(x_1^0,x_2)$  has  degree $\deg P\ge 2.$ Indeed, this is clear under assumption (i), since $B_{l+1}\ge 2,$ and under the assumptions  (ii) and (iii) it follows from $B_l\ge 2,$ respectively $B_{l-1}\ge 2$ in case that $B_l=B_{l+1}=1.$ Clearly also $\deg P\le B_l.$ 

 Therefore, if $|s_2|\lesssim 1,$ we can argue in a similar way as in Case 3 of the proof of Proposition \ref{s7.5}, and obtain by means of Proposition \ref{s8.1} that 
\begin{equation}\label{9.16}
|J_{k}(\xi)|\le C\,2^{-|\ka'|k} (1+2^{-\ka'_1nk}\la(1+|s_1|+|S_2|))^{-1/2}(1+2^{-\ka'_1nk}\si \la(1+|s_2|))^{-1/m}
\end{equation}
for some $m$ with $2\le m\le B_l,$ 
provided $\ve'$ is chosen sufficiently small.

On the other hand, if $|s_2|>>1,$ we can obtain the even stronger estimate where the second exponent $-1/m$ is replaced by $-1$ by first integrating by parts in $x_2$ and then integrating in $x_1.$ 

\bigskip

Now from \eqref{9.13} we deduce  as in the proof  of Proposition \ref{s9.3} that if 
 $|\xi_1|+|\xi_2|\ge \la,$ so that $|\xi|\sim |\xi_1|+|\xi_2|,$ then we have 
$|s_1|+|S_2|>>1,$
and  $|s_1|\sim |S_2|$ implies that 
$|\xi_1|<<|\xi_2|.$

Thus, if $|\xi_1|+|\xi_2|\ge \la$ and $|s_1|\sim |S_2|,$ then $|\xi|\sim |\xi_2|,$ and since
 $|s_2|\la>>|\xi_2|,$ we see that \eqref{9.16}  implies \eqref{9.15} in this case, as well as of course in  the case where $|\xi_1|+|\xi_2|\le \la,$  provided $\ve$ is chosen small enough. We are thus left with the case

\medskip

{\bf 3. Case.} $|s_1|+|S_2|\sim 1$  and  $|\xi_1|+|\xi_2|\le \la,$ hence $|\xi|\sim \la.$
\medskip

Since  $n>m_1,$  the polynomial $p(x_1):=x_1^n+S_2b_1x_1^{m_1}+s_1x_1$ satisfies $|p''(x_1)|+|p'''(x_1)|\ne 0$ for every $x_1\sim 1.$ 
But, if either $|s_1|<<|S_2|$  or $|s_1|>>|S_2|,$  then all critical points of    the polynomial $x_1^n+S_2b_1x_1^{m_1}+s_1x_1$ will be non-degenerate, so that we can argue exactly as in Case 2. We shall therefore assume that 
$$|s_1|\sim |S_2|\sim 1.
$$

Now, under assumption (i), we have $\pa_2\th_{\ka'}(x_1^0,0)= 0$ whenever  $x_1^0\sim 1,$ whereas $|s_2|\gtrsim 1,$ by \eqref{9.12}, so that 
\begin{equation}\label{9.17}
\pa_2 (\th_{\ka'}+s_2x_2)(x_1^0,0)=\pa_2 \th_{\ka'}(x_1^0,0)+s_2\ne 0.
\end{equation}
The same is true also under assumption (ii), for then either $|s_2|>>1$ or $|s_2|<<1$ (by \eqref{9.12}), whereas $\pa_2\th_{\ka'}(x_1^0,0)\ne 0,$ and it also applies in case (iii), provided  $|s_2|>>1$ or $|s_2|<<1.$ 

 In these  cases, we shall first integrate by parts in 
$x_2$ and then apply  a Bj\"ork type version of van der Corput's lemma in $x_1,$ which results in the estimate
$$
|J_{k}(\xi)|\le C\,2^{-|\ka'|k} (1+2^{-\ka'_1nk}\la)^{-1/3}(1+2^{-\ka'_1nk}\si \la(1+|s_2|))^{-1}.
$$
By replacing the second exponent $-1$ by $-1/6-\ve,$ we see in view of \eqref{9.14} that this implies  \eqref{9.15}.

\medskip 

We are thus left with the case where assumption (iii)  holds true, and where $|s_2|\sim 1.$  Fix $x_1^0\sim 1.$ Then $\pa_1\pa_2 (\th_{\ka'}+s_2x_2)(x_1^0,0)\ne 0,$ since $\th_{\ka'}\x=c_0x_1^{n-m_1}x_2+O(x_2^2),$ where $c_0\ne 0.$ 

Assume first that \eqref{9.17} holds true. Then we can again argue as before, provided we introduce in our formula for $J_k(\xi)$ an additional smooth cut-off function $a(x_1)$ supported in a sufficiently small neighborhood of $x_1^0.$ 

\medskip
So, assume next that $\pa_2 (\th_{\ka'}+s_2x_2)(x_1^0,0)=0.$ Since the degree of the polynomial  $P(x_2):=\th_{\ka'}(x_1^0,x_2)$ satisfies $B_l\ge \deg P\ge 2,$ after shifting the $x_1$ coordinates by $x_1^0,$ we can apply  Theorem \ref{s8.3} and obtain estimate \eqref{9.15} for some $m$ with $2\le m\le B_l$, if we again introduce a cut-off  function $a(x_1)$ supported in a sufficiently small neighborhood of $x_1^0$ into $J_k(\xi).$ Recall here that the functions $v_k,w_k$ and $\th_{r,k}$ are smooth functions  $v(x_1,\delta),w(x_1,\delta)$ respectively  $\th_{r}(x,\delta)$ depending also on the small parameter $\delta=2^{-k/r}$ for some positive integer $r>0,$ which vanish identically when $\delta=0.$

The estimate \eqref{9.15} then follows by decomposing $J_k(\xi)$ into a finite number of such ''localized'' integrals by means of a partition of unity.

\bigskip
Next, in order to estimate the maximal operator $\M',$ observe that \eqref{9.15} implies that for any  sufficiently small  $\ve>0$ we have 
\begin{eqnarray*}
|J_{k}(\xi)|\le C ||\eta ||_{C^3(\RR^2)}\, 2^{-|\ka'|k}\,2^{\ka'_1nk(1/2+\ve)}\,2^{(1-\ka'_1n)k(l_m+c\ve)}\,(1+|\xi|)^{-1/2-\ve}.
\end{eqnarray*}
Recalling  that  $1-\ka'_1n>0$  and $l_m<1/4$ by \eqref{9.14}, we thus see that there is  some $\delta>0$ such that 
$$
|J_{k}(\xi)|\le C ||\eta ||_{C^3(\RR^2)}\,2^{-\delta k}\, 2^{-|\ka'|k}\, 2^{\frac {(1+\ka'_1n)k}4}\,(1+|\xi|)^{-1/2-\ve},
$$
provided $\ve $ is sufficiently small.  Lemma \ref{s7.4}  then implies 
\begin{eqnarray*}
||\M^kf||_p&\le& C 2^{-\delta k}\, 2^{-|\ka'|k}\, 2^{\frac {(1+\ka'_1n)k}{2p}}||f||_p
\end{eqnarray*}
for every  $p\ge 2.$
Notice that 
\begin{equation}\label{9.18}
\frac {1+\ka'_1n}{2|\ka'|}\le \frac 1{|\ka|}.
\end{equation}
Indeed, we have $t:=\ka'_1n=\frac{\ka'_1}{\ka_1}\le 1$ and $\ka'_2\ge \ka_2$  by Lemma \ref{9.2}, so that 
$$\frac {1+\ka'_1n}{2|\ka'|}=\frac{1+t}{2(\ka_1t+\ka'_2)}\le \frac{1+t}{2(\ka_1t+\ka_2)}.
$$
The latter function is increasing in $t,$ so that we may replace $t$ by $1$ and obtain \eqref{9.18}.
\medskip

The estimate \eqref{9.18} shows that  the norms of the maximal operators $\M^k$ sum in $k$ when $p\ge \frac 1{|\ka|},$ which concludes the proof of Proposition \ref{s9.4}, hence also the proof of our main result, Theorem \ref{s1.2}.

\qed

\medskip
\setcounter{equation}{0}
\section{Estimates for  oscillatory integrals with small parameters}\label{oscint}

In this section, we shall provide the estimates for oscillatory integrals that were needed in the previous sections. More precisely,  we shall study oscillatory integrals 
$$
J(\lambda,\sigma,\delta):=\int_{\bR^2}e^{i\lambda
F(x,\sigma,\delta)}\psi(x,\delta)\, dx,\qquad (\la>0),
$$
with a  phase function  $F$ of the form 

$$
F(x_1,x_2,\sigma,\delta):=f_1(x_1,\delta)+\si f_2(x_1,x_2,\delta),
$$
and an amplitude $\psi$ defined for $x$ in some open neighborhood of the origin in $\RR^2$ with compact support in $x.$
The functions $f_1,f_2$  are assumed to be real-valued and will depend, like  the function $\psi,$ smoothly on $x$ and on small real parameters $\de_1,\dots,\de_\nu,$ which form the vector $\de:=(\de_1,\dots,\de_\nu)\in \bR^\nu.$ $\si$ denotes a small real parameter.

With a slight abuse of language we  shall say that $\psi$ is {\it compactly supported in some open set  $U\subset \RR^2$ } if there is a compact subset $K\subset U$ such that $\supp \psi(\cdot,\de)\subset K$ for every $\de.$ 

\bigskip
\subsection{Oscillatory integrals with non-degenerate critical points in $x_1$ }

\begin{prop}\label{s8.1}

Assume that 
$$|\pa_1f_1(0,0)|+|\pa_1^2f_1(0,0)|\ne 0,
$$
and that there is some $m\ge 2$ such that
$$\pa_2^mf_2(0,0,0)\ne 0.
$$
 Then there exists a
neighborhood $U\subset \bR^2$  of the
origin and some $\ve>0$ such that for any $\psi$ which is compactly supported in $U$   the
following estimate
\begin{equation}\label{8.1}
|J(\lambda,\sigma,\delta)|\le
\frac{C\|\psi(\cdot,\de)\|_{C^3}}{(1+\lambda)^{1/2}(1+|\lambda\si|)^{1/m}}
\end{equation}
holds true uniformly for $|\si|+|\delta|<\ve.$ 
\end{prop}

\proof
If $\pa_1f_1(0,0)\ne 0,$ then we can integrate by parts in $x_1$ if $\la>1$ and obtain the stronger estimate 
$$|J(\lambda,\sigma,\delta)|\le
\frac{C\|\psi(\cdot,\de)\|_{C^1}}{1+\lambda}.
$$
Assume therefore that $\pa_1f_1(0,0)= 0,$ so that the mapping $x_1\mapsto f_1(x_1,0)$ has a non-degenerate critical point at $x_1=0.$   Then,  by the implicit function theorem, for $|\de|$ sufficiently small  there exists a unique critical point $x_1=x_1^0(\de)$ depending smoothly on $\de$ of the mapping $\xi\mapsto f_1(x_1,\de)=0,$ i.e.,  $\pa_1f_1(x_1^0(\de),\de)\equiv 0,$ where 
$x_1^0(0)=0.$

In a similar way, we see that there is  a unique, smooth function $x_1^c(x_2,\si,\de)$ for $|x_2|+|\si|+|\de|$ sufficiently small such that 
$$\pa_1F(x_1^c(x_2,\si,\de),x_2,\si,\de)\equiv 0,
$$
where $x_1^c(0,0,0)=0.$ By comparison, we see that $x_1^c(x_2,0,\de)=x_1^0(\de),$ so that 
$$x_1^c(x_2,\si,\de)=x_1^0(\de)+\si\ga(x_2,\si,\de)
$$
for some smooth function $\ga.$ Applying the stationary phase formula with parameters to the integration in $x_1,$ we thus obtain 
  \begin{equation}\label{8.2}
J(\lambda,\si,\delta)= \int_\RR e^{i\la\phi(x_2,\si,\de)} a(\la,x_2,\si,\de)\, dx_2,
\end{equation}
where
$$
\phi(x_2,\si,\de):=F(x_1^0(\de)+\si \ga(x_2,\de,\si),x_2,\si,\delta),
$$
and where $a(\la,x_2,\si,\de)$ is a symbol of order $-1/2$ in $\la,$  so that in particular
\begin{equation}\label{8.3}
|\pa_{x_2}^l a(\la,x_2,\si,\de)|\le C_l (1+|\la|)^{-1/2},
\end{equation}
with constants $C_l$ which are independent of $x_2,\si$ and $\de$  (see, e.g., Sogge \cite{sogge-book} or H\"ormander \cite{hoermander1}). 

Moreover, a Taylor series expansion of 
$\phi$ with respect to $\si$ near $\si=0$ shows that 
$$
\phi(x_2,\si,\de)=f_1(x_1^0(\de),\de)+\si\Big(f_2(x_1^0(\de),x_2,0,\de)+O(\si)\Big)
$$
in $C^\infty.$ Since $\pa_2^mf_2(0,0,0)\ne 0,$ for $|\si|$ sufficiently small we can thus apply van der Corput's lemma (cf.\cite{stein-book}) to the integral \eqref{8.2} in $x_2$ and obtain the estimate \eqref{8.1}.

\qed

\bigskip
\subsection{Oscillatory integrals of non-degenerate Airy type }

\begin{prop}\label{s8.2}
Assume that 
$$\pa_1^3f_1(0,0)\ne 0\ \mbox{ and }\ \pa_2^2f_2(0,0,0)\ne 0.
$$

 Then there exists a
neighborhood $U\subset \bR^2$  of the
origin and some $\ve>0$ such that for any $\psi$ which is compactly supported in $U$   the
following estimate
\begin{equation}\label{8.4}
|J(\lambda,\sigma,\delta)|\le
\frac{C\|\psi(\cdot,\de)\|_{C^3}}{(1+\lambda)^{1/3}(1+|\lambda\si|)^{1/2}}
\end{equation}
holds true uniformly for $|\si|+|\delta|<\ve.$ 
\end{prop}

\proof Consider  first the case where $\pa_2f_2(0,0,0)\neq0.$ Then, if  $|\lambda\si|>>1,$ we first perform an integration by parts in $x_2.$ Subsequently, we can apply van der Corput's lemma to the integration in $x_1,$  provided $U$ and $\ve$ are chosen sufficiently small, and obtain the stronger estimate
$$|J(\lambda,\sigma,\delta)|\le
\frac{C\|\psi(\cdot,\de)\|_{C^2}}{(1+\lambda)^{1/3}(1+|\lambda\si|)}.
$$

Now, assume that $\pa_2f_2(0,0,0)=0$ but $\pa_2^2f_2(0,0,0)\neq0.$ Then for $U$ and $\ve$  chosen sufficiently small, by the implicit function theorem there exists a unique critical point $x_2^c(x_1,\de)$ of the function $x_2\mapsto f_2(x_1,x_2,\de).$ 
Then, by applying the  stationary phase method  with small
parameters  to the  $x_2$-integration,  we see that
  \begin{equation}\label{8.5}
J(\lambda,\si,\delta)= \int_\RR e^{i\la\phi(x_1,\si,\de)} a(\la\si,x_1,\de)\, dx_1,
\end{equation}
where
$$
\phi(x_1,\si,\de):=f_1(x_1,\de)+\si f_2(x_1,x_2^c(x_1,\de),\delta),
$$
and where $a(\la,x_1,\de)$ is a symbol of order $-1/2$ in $\la,$  so that in particular
\begin{equation}\label{8.6}
|\pa_{x_1}^l a(\la\si,x_1,\de)|\le C_l (1+|\la\si|)^{-1/2},
\end{equation}
with constants $C_l$ which are independent of $x_1$ and $\de.$  

We can now apply van der Corput's lemma to the integral \eqref{8.5} and obtain in view of \eqref{8.6} the desired estimate \eqref{8.4}.

\qed

\bigskip
\subsection{Oscillatory integrals of degenerate Airy type }

\begin{thm}\label{s8.3}
Assume that 
\begin{equation}\label{8.7}
|\pa_1f_1(0,0)|+|\pa_1^2f_1(0,0)|+|\pa_1^3f_1(0,0)|\ne 0\ \mbox{ and }\ \pa_1\pa_2f_2(0,0,0)\neq 0,
\end{equation}
and that there is some $m\ge 2$ such that 
\begin{equation}\label{8.8}
\pa_2^lf_2(0,0,0)=0\mbox{  for } l=1,\dots, m-1\mbox{ and  }
\pa_2^{m}f_2(0,0,0)\neq0.
\end{equation}

Then there exists a
neighborhood $U\subset \bR^2$  of the
origin and  constants $\ve,\ve'>0$ such that for any $\psi$ which is compactly supported in $U$   the
following estimate
\begin{equation}\label{8.9}
|J(\lambda,\sigma,\delta)|\le
\frac{C\|\psi(\cdot,\de)\|_{C^3}}{\lambda^{\frac
{1}{2}+\ve}|\si|^{(l_m+c_m\ve)}}
\end{equation}
holds true  uniformly for $|\si|+|\delta|<\ve',$ where $l_m:=\frac{1}{6} $ and $c_m:=1$ for $m<6$, and $l_m:=\frac{m-3}{2 (2m-3)}$ and $c_m:=2$ for $m\geq 6.$ \end{thm}

\begin{remark}\label{s8.4}
If $|\pa_1f_1(0,0)|+|\pa_1^2f_1(0,0)|\ne 0,$ then a stronger estimate than \eqref{8.9} follows from Proposition \ref{s8.1}, since $1/6\le l_m<1/4.$ The full thrust of Theorem \ref{s8.3} therefore lies in the case where $\pa_1f_1(0,0)= \pa_1^2f_1(0,0)= 0$ and $\pa_1^3f_1(0,0)\ne 0,$ on which we shall concentrate in the sequel. 
\end{remark}

The proof of Theorem \ref{s8.3} will be an immediate consequence of the  following two lemmas. Our first lemma allows to reduce the phase function $F$ to some normal form and is based on Martinet's theorem.

\begin{lemma}\label{s8.5}
 Assume that the function  $F$  satisfies  the
conditions of Theorem \ref{s8.3}, and in addition that  $\pa_1f_1(0,0)= \pa_1^2f_1(0,0)= 0.$   Then there exist smooth functions
$X_1=X_1(x_1,\si,\delta)$ and $X_2= X_2(x_1,x_2,\delta)$ defined in a
sufficiently small neighborhood $U\times V\subset \bR^2\times
\bR^{\nu+1}$ of the origin such that the following hold true:
\bee
\item[(i)]  $X_1(0,0,0)=X_2(0,0,0)=0,\ 
\pa_1X_1(0,0,0)\neq0,\,\pa_2X_2(0,0,0)\neq0,$ so that we can change coordinates from
 $(x_1,x_2,\si,\de)$ to  $(X_1,X_2,\si,\de)$ near the origin.

\item[(ii)] In the new coordinates $X_1,X_2$ for  $\RR^2$ near the origin, we can write 
 $F(x_1,x_2,\si,\de)=g_1(X_1,\si,\de)+\si g_2(X_1,X_2,\si,\de),$
with 
$$
g_1(X_1,\si,\de)=X_1^3+a_{1}(\si,\de)X_1+a_m(\si,\de)
$$
and 
\begin{eqnarray*}
g_2(X_1,X_2,\si,\de)&=&X_2^{m}+\sum_{j=2}^{m-2}a_j(\de)X_2^{m-j}\\
&+&\Big(X_1-a_{m-1}(\si,\de)\Big)X_2\,b(X_1,X_2,\si,\de) ,
\end{eqnarray*}
if $m\ge 3,$ and 
$g_2(X_1,X_2,\si,\de)=X_2^2,$ if $m=2,$ 
where $a_1,\dots, a_{m}$ are smooth functions of the variables $\si,\de$ such that 
$a_l(0,0)=0,$ and where $b$ is a smooth function such that  $ b(0,0,0,0)\neq0.$ 
\ee
\end{lemma}

\proof  In a first step, we apply Martinet's theorem (more precisely, the special case proved in \cite{hoermander1}, Theorem 7.5.13)  to  the
function $f_2(x_1,x_2,\delta).$ Due to our assumption \eqref{8.8}, 
there exists a smooth function $X_2=X_2(x_1,x_2,\delta)$
defined in a sufficiently small neighborhood of the origin with
$$ X_2(0,0,0)=0,\ 
\pa_2X_2(0,0,0)\neq 0,
$$ so that  in the new coordinate $X_2$ for $\RR$ near the origin $f_2$ assumes the form 
$$
f_2(x_1,x_2,\delta)=
X_2^{m}+\tilde a_2(x_1,\de)X_2^{m-2}+\dots
+\tilde a_{m-1}(x_1,\de)X_2+\tilde a_m(x_1,\de),
$$
where $\tilde a_2,\dots,\tilde a_m$ are smooth functions satisfying 
$\tilde a_l(0,0)=0,  \ l=2,\dots,m-1.$ 

Notice that the case $m=2$ is special, since in this case 
$$
f_2(x_1,x_2,\delta)=X_2^2+\tilde a_2(x_1,\de)
$$
contains no linear term in $X_2.$ 

If $m\ge 3,$ then  by assumption \eqref{8.7}, we have 
\begin{equation}\label{8.10}
\frac{\pa\tilde a_{m-1}}{\pa x_1 }(0,0)\neq 0,
\end{equation}
 since 
$$
0\ne \frac{\pa^2 f_2}{\pa x_1\pa
x_2}(0,0,0)=\frac{\pa\tilde a_{m-1}}{\pa x_1}(0,0)\frac{\pa
X_2}{\pa x_2}(0,0,0).
$$
Consequently, any smooth function $\vp=\vp(x_1,\de)$ defined in a sufficiently small 
 neighborhood of the origin can be written in the form
$$
\vp(x_1, \de)=\eta(x_1,\de)\,\tilde a{}_{m-1}(x_1,\de)+\tilde\vp(\de),
$$
with smooth functions $\eta=\eta(x_1,\de)$  and $\tilde\vp(\de).$ Applying this observation to the functions $\tilde a_l,$ we can write 
$$
\tilde a{}_l(x_1,\de)=\tilde a{}_{m-1}(x_1,\de)\, b_l(x_1,\de)+a_l(\de), \quad l=2,\dots, m-2,
$$
with smooth functions $b_l(x_1,\de)$ and $a_l(\de),$ where $a_l(0)=0.$

We can accordingly re-write the function $F=f_1+\si f_2$  in  the form
$F=\tilde f_1+\si\tilde f_2,$ where 
\begin{eqnarray}
\tilde f_1(x_1,\si,\de)&=&f_1(x_1,\de)+\si \tilde a{}_m(x_1,\de),\nonumber\\
\tilde f_2(x_1,x_2,\si,\de)&=&X_2^m+
\tilde a{}_{m-1}(x_1,\de)\, X_2\,\tilde b(x_1,X_2,\de)\nonumber\\
&+& a_2(\de)X_2^{m-2}+\dots \label{8.11}
+a_{m-2}(\de)X_2^2,
\end{eqnarray}
with
$$\tilde b(x_1,X_2,\de):=1+b_{m-2}(x_1,\de)X_2+\dots
+b_2(x_1,\de)X_2^{m-2}.
$$
In particular, $\tilde b(0,0,0)\ne 0.$

\medskip
In a second step, we apply  Martinet's  theorem to the function
$\tilde f_1(x_1,\si,\de).$  Since  $\pa_1f_1(0,0,0)=\pa_1^2f_1(0,0,0)= 0$ and  $\pa_1^3f_1(0,0,0) \ne 0,$ we then see that there exists a smooth function $X_1=X_1(x_1,\si,\delta)$
defined in a sufficiently small neighborhood of the origin with
$$ X_1(0,0,0)=0,\ 
\pa_1X_1(0,0,0)\neq 0,
$$ so that  in the new coordinate $X_1$ for $\RR$ near the origin $\tilde f_1$ assumes the form 
$$
\tilde f_1(x_1,\si,\de)=X_1^3+a_{1}(\si,\de)X_1+a_m(\si,\de),
$$
where $a_{1},\,a_m$ are smooth functions such that $a_{1}(0,0)=a_m(0,0)=0.$ 

 Let us write $\tilde a_{m-1}(x_1,\de)=\al(X_1(x_1,\si,\de),\si,\de),$ so that $\al$ expresses $\tilde a_{m-1}$ in the new coordinates $X_1.$ By \eqref{8.10} and the chain rule, we have  
$$
\al(0,0,0)=0,\quad \mbox{and}\quad \frac{\pa\al}{\pa X_1}(0,0,0)\neq 0.
$$

This implies that  there exists a unique, smooth function $a_{m-1}(\si,\de)$ with 
$a_{m-1}(0,0)=0,$ such that $\al(a_{m-1}(\si,\de),\si,\de)\equiv 0.$ Taylor's formula then implies that $\al(X_1,\si,\de)$ can be written in the form
$$
\al(X_1,\si,\de)=(X_1-a_{m-1}(\si,\de))\,\tilde g(X_1,\si,\de),
$$
where $\tilde g(X_1,\si,\de)$ is a smooth function with $\tilde
g(0,0,0)\neq0.$ This shows that 
$$\tilde a{}_{m-1}(x_1,\de)\, X_2\,\tilde b(x_1,X_2,\de)=(X_1-a_{m-1}(\si,\de))\, X_2\,\tilde g(X_1,\si,\de)\tilde b(x_1,X_2,\de).
$$
When  expressed   in the new variables $(X_1,X_2),$ we see that in combination with \eqref{8.11} we obtain the form of $F$ as described in  (ii).

\qed

After changing coordinates, the previous lemma allows to reduce Theorem \ref{s8.3} to the estimation of  two-dimensional oscillatory integrals with phase functions of the form 
$F(x_1,x_2,\de,\delta)=f_1(x_1,\delta)+\si f_2(x_1,x_2,\delta),$ where 
\begin{eqnarray}
f_1(x_1,\de)&=&x_1^3+\de_{1}x_1, \nonumber\\
&&\label{8.12}\\
f_2(x_1,x_2,\de)&=& x_2^{m}+\sum_{j=2}^{m-2}\de_jx_2^{m-j}+(x_1-\de_{m-1})\,x_2\,b(x_1,x_2,\si,\delta),\nonumber
\end{eqnarray}
if $m\ge 3,$ and $f_2(x_1,x_2,\de)=x_2^2,$ if $m=2.$
Here, $\si$ and $\de_1,\dots,\de_\nu$ are small real parameters (where $\nu\ge m-1$ ), the latter forming the vector  $\de:=(\de_1,\dots,\de_\nu)\in\bR^\nu,$  and 
 $b=b(x_1,x_2,\si,\de)$ is a smooth function defined on a neighborhood of the origin with 
 $b(0,0,0,0)\neq0.$

\begin{lemma}\label{s8.6}
Assume that the phase function $F$ is given by \eqref{8.12}. 
Then there exists a
neighborhood $U\subset \bR^2$  of the
origin and  constants $\ve,\ve'>0$ such that for any $\psi$ which is compactly supported in $U$   the
following estimate
\begin{equation}\label{8.13}
|J(\lambda,\sigma,\delta)|\le
\frac{C\|\psi(\cdot,\de)\|_{C^3}}{\lambda^{\frac
{1}{2}+\ve}|\si|^{(l_m+c_m\ve)}}
\end{equation}
holds true  uniformly for $|\si|+|\si|<\ve',$ where $l_m$ and $c_m$ are defined as in Theorem \ref{s8.3}.
\end{lemma}

\proof  We shall  prove Lemma \ref{s8.6} and Theorem \ref{s8.3} at the same time by induction over $m.$

If $m=2$ then the phase function \eqref{8.12}  is reduced to the form
$$
F\x=x_1^3+\de_1x_1+\si x_2^2,
$$
and by applying the method of stationary phase in $x_2$ and van der Corput's lemma in $x_2$ we easily obtain estimate \eqref{8.13}, with $l_2=1/6.$  This proves also Theorem \ref{8.3} for $m=2.$

\medskip
Assume that $m\ge 3,$ and that the statement of Theorem  \ref{s8.3} holds for every strictly smaller value of $m.$  We shall apply again  a Duistermaat type argument, in a similar way as in Section \ref{uniform estimates}, in order to prove the statement of Lemma \ref{s8.6},  hence also that of Theorem  \ref{s8.3}, for $m.$ 
To this end, we introduce  the mixed-homogeneous scalings $\Delta_\rho\x:=(\rho^{\frac12}x_1,\rho^{\frac{1}{2(m-1)}}x_2),\ \rho>0.$ Notice that these are such that the principal part of $f_2$ with respect to these dilations is given by $x_2^m+x_1x_2\,b(0,0,\si,\de).$ Then
$$
F(\Delta_\rho(x),\si,\de)=\rho^{\frac 32}F(x,\tilde \si,\tilde \de,\rho,\si,\de),
$$
where $F(x,\tilde \si,\tilde \de,\rho,\si,\de)=f_1(x_1,\tilde\de)+\tilde\si\,f_2(x,\tilde \si,\tilde \de,\rho,\si,\de)$ is given by 
\begin{eqnarray}
f_1(x_1,\tilde\de)&:=&x_1^3+\tilde\de_1x_1\label{8.14}\\
f_2(x,\tilde \si,\tilde \de,\rho,\si,\de)&:=& x_2^{m}+\sum_{j=2}^{m-2}\tilde\de_jx_2^{m-j}+(x_1-\tilde \de_{m-1})x_2\,b(\Delta_\rho(x),\si,\de),\nonumber
\end{eqnarray}
with $\tilde\si,\tilde \de$ defined by 
$$\tilde \si:=\frac{\si}{\rho^{\frac{2m-3}{2(m-1)}}},\quad \tilde \de_1:=\frac{\de_1}{\rho},\quad \tilde\de_j:=\frac{\de_j}{\rho^
{\frac{j}{2(m-1)}}}\quad ( j=2,\dots, m-1),
$$
so that in particular $\tilde\delta_{m-1}=\frac{\de_{m-1}} {\rho^{\frac12}}.$ Thus, if we define ''dual scalings'' by
$$
\Delta^*_\rho(\si,\de):=(\tilde\si,\tilde \de),
$$
we see that if $b$ is constant,  then $F(\Delta_\rho(x),\si,\de)=\rho^{\frac 32}F(x,\Delta^*_\rho(\si,\de)).$

It is then natural to introduce the quasi-norm
$$
N(\si,\de):=|\si|^{\frac{2(m-1)}{2m-3}}+|\de_1|+|\de_2|^
{m-1}+\dots+|\de_{m-2}|^{\frac{2(m-1)}{m-2}}+|\de_{m-1}|^2,
$$
which is $\Delta^*_\rho$-homogeneous of degree $-1,$ i.e., 
$N(\Delta^*_\rho(\si,\de))=\rho^{-1}N(\si,\de).$

\medskip
Given $\si,\de,$ we  now choose $\rho$ so that $N(\tilde \si,\tilde\de)=1,$ i.e., 
$$\rho:=N(\si,\de).
$$
Notice that $\rho<<1,$ and that $(\tilde\si,\tilde\de)$ lies in the ''unit sphere"
$$\Sigma:=\{(\si',\de')\in\bR^m: N(\si',\de')=1\}.$$ 
Then, after scaling, we may re-write 
$$
J(\la,\si,\de)=J(\la,\tilde\si,\tilde\de,\rho,\si,\de):=\rho^{\frac{m}{2(m-1)}}\int_{\bR^2}
e^{i\lambda\rho^{\frac32}
F(x,\tilde\si,\tilde\de,\rho,\si,\de)}\psi(\Delta_\rho(x),\de)\, dx,
$$
where here $\rho,\si$ and the $\de_j$ are small parameters. For a while, it will be convenient  to consider $\tilde\si$ and the $\tilde \de_j$ as additional, independent real parameters, which may not be small, but bounded.

We shall apply a dyadic decomposition to this integral. To this end, we choose $\chi_0,\chi\in C_0^\infty(\bR^2)$ with
$\supp\chi\subset\{\frac{B}{2}<|x|<2B\}$ (where $B$ is a
sufficiently large  positive number to be fixed later) such that
$$
\chi_0(x)+\sum_{k=1}^\infty \chi(\Delta_{2^{-k}}(x))=1, \quad
\mbox{for every} \quad x\in \bR^2.
$$
Accordingly we decompose the oscillatory integral
$$
J(\la,\tilde\si,\tilde\de,\rho,\si,\de)=\sum_{k=0}^\infty
J_k(\la,\tilde\si,\tilde\de,\rho,\si,\de),
$$
where
$$
J_k(\la,\tilde\si,\tilde\de,\rho,\si,\de):=\rho^{\frac{m}{2(m-1)}}\int_{\bR^2}
e^{i\lambda\rho^{\frac32}
F(x,\tilde\si,\tilde\de,\rho,\si,\de)}\psi(\Delta_\rho(x),\de)\,\chi_k(x)\, dx,
$$
and $\chi_k(x):=\chi(\Delta_{2^{-k}}(x))$ for $k\ge1.$

\medskip
Assume first that $k\ge1$. Then, by using the scaling
$\Delta_{2^k},$ we get
$$
J_k(\la,\tilde\si,\tilde\de,\rho,\si,\de)=(2^k\rho)^
{\frac{m}{2(m-1)}}\int_{\bR^2}
e^{i\la (2^k\rho)^{\frac32}
F_k(x)}\psi(\Delta_{2^k\rho}(x),\de)\,\chi(x)\, dx,
$$
where $F_k(x):=g_1(x_1,\tilde \si_k)+\tilde\si_k\,g_2(x,\tilde \si_k,\tilde\de_{k}, 2^k\rho,\si,\de)$ is
given by
\begin{eqnarray*}
g_1(x_1,\tilde \de_k)&:=&x_1^3+\tilde\de_{1,k}x_1,\\
g_2(x,\tilde \si_k,\tilde\de_{k}, 2^k\rho,\si,\de)&:=&x_2^{m}+\sum_{j=2}^{m-2}\tilde\de_{j,k}x_2^{m-j}+(x_1-\tilde\de_{m-1,k})\,x_2\,b(\Delta_{2^k\rho}(x),\si,\de),
\end{eqnarray*}
with
$$
(\tilde \si_k,\tilde\de_{k}):=(\tilde \si_k,\tilde\de_{1,k},\dots, \tilde \de_{m-1,k})
:=\Delta^*_{2^k}(\tilde\si,\tilde\de)=\Delta^*_{2^k
\rho}(\si,\de).
$$

\medskip

Observe that we may restrict ourselves to those $k$ for which $2^k\rho\lesssim1/B,$ since otherwise  $J_k\equiv 0.$ Consequently, if we  choose $B$ in the definition of $\chi$ sufficiently large, then $2^k\rho<<1,$ and also $|\tilde
\si_k|+|\tilde\de_k|<<1$. We thus see that  there is some positive constant $c>0$  such that if $x\in \supp\chi,$ then either 
$|\pa_1g_1(x_1,\tilde \de_k)|\ge c,$ or
$|\pa_2g_2(x_1,x_2,\tilde \si_k,\tilde\de_{k}, 2^k\rho,\si,\de)|\ge c.$ 
\medskip

Fix a point $x^0=(x_1^0,x_2^0)\in\supp \chi,$ let $\eta$ be  a smooth cut-off function 
supported  in a sufficiently small neighborhood of $x^0,$ and consider the oscillatory integral
$J_k^{\eta}$ defined by
$$
J_k^\eta(\la,\tilde\si,\tilde\de,\rho,\si,\de)=(2^k\rho)^
{\frac{m}{2(m-1)}}\int_{\bR^2}
e^{i\la (2^k\rho)^{\frac32}
F_k(x)}\psi(\Delta_{2^k\rho}(x),\de)\,\chi(x)\eta(x)\, dx.
$$

By using an integration by parts in  $x_1$ in case that $|\pa_1g_1(x^0_1,\tilde \de_k)|\ge c,$ respectively in $x_2$ if $|\pa_2g_2(x_1^0,x_2^0,\tilde \si_k,\tilde\de_{k}, 2^k\rho,\si,\de)|\ge c,$ 
and subsequently applying van der Corput's lemma  to the $x_1$-integration  in the latter case, we then obtain
$$
|J_k^{\eta}|\le
\frac{C (2^k\rho)^{\frac{m}{2(m-1)}}\|\psi(\cdot,\de)\|_{C^3}}
{(1+\lambda(2^k\rho)^{\frac32})^{\frac13}(1+\lambda(2^k\rho)^{\frac32}|\tilde\si_k|)^{\frac23}}
\le \frac{C (2^k\rho)^{\frac{m}{2(m-1)}}\|\psi(\cdot,\de)\|_{C^3}}
{|\lambda(2^k\rho)^{\frac32}|^{\frac12+\varepsilon}|\tilde\si_k|^
{\frac16+\varepsilon}}.
$$
By means of a partion of unity argument this implies the same type of estimate 
\begin{equation}\label{8.15}
|J_k|\le
\frac{C (2^k\rho)^{\frac{m}{2(m-1)}}\|\psi(\cdot,\de)\|_{C^3}}
{|\lambda(2^k\rho)^{\frac32}|^{\frac12+\varepsilon}|\tilde\si_k|^
{\frac16+\varepsilon}}=C (2^k\rho)^{\frac{6-m}{12(m-1)}-\frac {\ve m}{2(m-1)}} 
\frac{\|\psi(\cdot,\de)\|_{C^3}}{\lambda^{\frac{1}{2}+\ve}|\si|^{\frac16+\varepsilon}}
\end{equation}
for $J_k.$

Consider first the case where $m< 6.$  Then clearly
$$
\sum_{k\ge 1}|J_k|=\sum_{2^k\rho\lesssim1}|J_k(\la,\tilde\si,\tilde\de,\rho,\si,\de)|\le
\frac{C\|\psi(\cdot,\de)\|_{C^3}}{\lambda^{\frac{1}{2}+
\ve}|\si|^{\frac{1}{6}+\ve}}.
$$

Assume next that  $m\ge 6$. Then the infinite series $\sum_{k=1}^\infty(2^k)^{\frac{6-m}{12(m-1)}-\frac {\ve m}{2(m-1)}} $converges. Note also that $\rho\ge |\si|^{\frac{2(m-1)}{2m-3}}.$ Summing therefore over all $k\ge 1,$ we obtain from \eqref{8.15} that 
$$
\sum_{k\ge 1}|J_k|\le
\frac{c\|\psi(\cdot,\de)\|_{C^3}}{|\lambda|^{\frac{1}{2}+\ve}|\si|^{l_m+c_m\ve}}.
$$
\medskip

We are thus left with the integral
$$
J_0(\la,\tilde\si,\tilde\de,\rho,\si,\de):=\rho^{\frac{m}{2(m-1)}}\int_{\bR^2}
e^{i\lambda\rho^{\frac32}
F(x,\tilde\si,\tilde\de,\rho,\si,\de)}\psi(\Delta_\rho(x),\de)\,\chi_0(x)\, dx,
$$
where $F(x,\tilde\si,\tilde\de,\rho,\si,\de)$ is given by \eqref{8.14}.
\medskip

Let us fix  a point $(\tilde\si^0,\tilde \de^0)\in \Sigma,$ and a point 
 $x^0=(x_1^0,x_2^0)\in\supp \chi_0,$ and let again $\eta$ be a smooth cut-off function supported near  $x^0.$ $J_0^\eta$ will be defined by introducing $\eta$ into the amplitude of $J_0$ in the same way as before. We shall prove that the oscillatory integral $J_0^\eta$ satisfies the estimate 
 \begin{equation}\label{8.16}
|J_0^\eta|\le
\frac{C\|\psi(\cdot,\de)\|_{C^3}}{\lambda^{\frac
{1}{2}+\ve}|\si|^{(l_m+c_m\ve)}},
\end{equation}
 provided  $\eta$ is  supported in a sufficiently small
neighborhood $U$ of  $x^0$ and 
$(\tilde\si,\tilde\de,\rho,\si,\de)\in V$, where $V$ is a
sufficiently small neighborhood of the point
$(\tilde\si^0,\tilde \de^0,0,0,0).$
By means of a partion of unity argument this will then  imply the same type of estimate for $J_0,$ hence for $J,$ which will conclude the proof of Lemma \ref{s8.6}, hence also of Theorem \ref{s8.3}.
 
 \medskip
 Now, if either  $\pa_1F(x_1^0,x_2^0,\tilde\si^0,\tilde \de^0,0,0,0)\neq0$ or
$\pa_2 f_2(x_1^0,x_2^0,\tilde\si^0,\tilde \de^0,0,0,0)\neq0,$ then we can estimate  $J_0^\eta$ exactly  like the  $J_k^\eta$ and  get the required estimate \eqref{8.16} for $J_0^\eta.$

\medskip
Assume therefore next that
\begin{equation}\label{8.17}
\pa_1F(x_1^0,x_2^0,\tilde\si^0,\tilde \de^0,0,0,0)=0\quad
\mbox{and also}\quad 
\pa_2 f_2(x_1^0,x_2^0,\tilde\si^0,\tilde \de^0,0,0,0)=0.
\end{equation}

We then distinguish the following four cases:
\medskip

{\bf Case 1.} $\tilde\si^0\neq0$ and  $x_1^0\neq0.$ 
\medskip

Then, since $x_1^0\neq0,$ it is easy to see from \eqref{8.14} that $\pa_1^2
F(x_1^0,x_2^0,\tilde\si^0,\tilde \de^0,0,0,0)\neq0$ as
well. Note here that if we write $b(x,\rho,\si,\de):=b(\Delta_\rho(x),\si,\de),$ then
\begin{equation}\label{8.18}
b(x,0,0,0)\equiv b(0,0,0,0)\ne 0.
\end{equation}

We can then argue here in a similar way as in the proof of Proposition \ref{s8.1}, so let us only briefly sketch the argument.
Suppose that $x_1^c(x_2,\tilde\si,\tilde \de,\rho,\si,\de)$ is
a critical point of $F$ with respect to $x_1.$  Then it is a smooth function of its
variables, and if $\rho=\si=\de=0,$ then by \eqref{8.18} 
$$
x_1^c=x_1^c(x_2,\tilde\si,\tilde \de,0,0,0)=\Big(-(\tilde\de_1+\tilde\si x_2b(0,0,0,0)\Big)^{1/2}
$$ 
and
\begin{eqnarray*}
&&F(x_1^c(x_2,\tilde\si,\tilde \de,0,0,0),x_2,\tilde\si,\tilde \de,0,0,0)\\
&=& (x_1^c)^3+\tilde \de_1x_1^c+\tilde\si \Big(
x_2^{m}+\sum_{j=2}^{m-2}\tilde\de_jx_2^{m-j}+(x_1^c-\tilde \de_{m-1})x_2\,
b(0,0,0,0)\Big).
\end{eqnarray*}
If $\phi$ denotes the phase function 
$$\phi(x_2,,\tilde\si,\tilde \de,\rho,\si,\de):=F(x_1^c(x_2,\tilde\si,\tilde \de,\rho,\si,\de),x_2,
\tilde\si,\tilde \de,\rho,\si,\de),
$$
which arises after applying the method of stationary phase to the $x_1$-integration, then 
since $\tilde\si^0\ne0,$ this easily shows that there exists a natural number $N$ such that 
$$
\pa_2^N\phi(x_2^0,\tilde \si^0,\tilde \de^0,0,0,0)\neq0.
$$

Consequently, we can in a second step apply  van der Corput's  lemma to the $x_2$-integration and obtain the estimate  
\begin{equation}\label{8.19}
|J_0^{\eta}|\le
\frac{C \rho^{\frac{m}{2(m-1)}}\|\psi(\cdot,\de)\|_{C^3}}
{|\lambda \rho^{\frac32}|^{\frac12+\varepsilon}|\tilde\si|^
{\frac16+\varepsilon}},
\end{equation}
which implies \eqref{8.16} as before (just put $k=0$ in our previous argument).
\medskip

{\bf Case 2.} $\tilde\si^0\neq0$ and  $x_1^0=0.$  
\medskip

 Then, by \eqref{8.18},  we have $\pa_1^2 F(x_1^0,x_2^0,\tilde\si^0,\tilde \de^0,0,0,0)=0$ as well. 
But, again by \eqref{8.18}, we also  have
$\pa_1\pa_2F(x_1^0,x_2^0,\tilde\si^0,\tilde \de^0,0,0,0)\neq0,$ so that $F$ has a  non-degenerate critical point at $x^0$ as a function of
two variables.  If the neighborhoods $U$ and $V$ are chosen sufficiently small, we can therefore apply the  stationary phase method in
two variables, which leads to an even  stronger estimate than the estimate \eqref{8.19}, since here  $|\tilde\si|\sim 1.$ 

\medskip

{\bf Case 3.}  $\tilde\si^0=0$ and  $\tilde\de_1^0\neq0.$  
\medskip

 In this case we have $x_1^0\neq0,$ because of \eqref{8.17}, and thus $\pa_1^2 F(x_1^0,x_2^0,\tilde\si^0,\tilde \de^0,0,0,0)\ne 0.$ Moreover, in this situation we consider $\tilde\si$ such that  $|\tilde\si|<<1.$ Since we can regard $\tilde\si-\tilde\si^0,\tilde\de-\tilde\de^0$ as small perturbation parameters if the neighborhoods $U$ and $V$ are chosen sufficiently small, we can therefore apply Proposition \ref{s8.1}, with $\si$ in this proposition replaced by $\tilde\si,$  and obtain \eqref{8.19}.

\medskip

{\bf Case 4.}  $\tilde\si^0=0$ and  $\tilde\de_1^0=0.$ 
\medskip

Then, by \eqref{8.17},   $x_1^0=0$ as well. In this case we make use of our induction hypothesis.
Indeed, let us  consider the function
$$
x_2\mapsto f_2(0,x_2,\tilde \si^0,\tilde \de^0,0,0,0)=x_2^{m}+\sum_{j=2}^{m-2}\tilde\de^0_jx_2^{m-j}-\tilde \de^0_{m-1}x_2\,b(0,0,0,0).
$$

\medskip

Now $x_2=x_2^0$ is a critical point, say  of multiplicity $\mu-1,$ of this function, i.e.,  
$ \pa_2^lf_2(0,x_2^0,\tilde \si^0,\tilde \de^0,0,0,0)=0$ for $l=1,\dots,\mu-1$ and 
$ \pa_2^\mu f_2(0,x_2^0,\tilde \si^0,\tilde \de^0,0,0,0)\ne 0.$

 Then $\mu<m$, because at least one of the coefficients
$\tilde \de_j,\ j=2,\dots,m-1,$ does not vanish and $b(0,0,0,0)\ne 0.$ Moreover, at this critical point also the condition 
$\pa_1\pa_2f_2(0,x_2^0,\tilde\si^0,\tilde \de^0,0,0,0)\neq0$ is satisfied.
Therefore, after translating coordinates $x_2$ by $x_2^0,$ by  our hypothesis we may apply the conclusion of Theorem \ref{8.3} for $\mu$ in place of $m$ and obtain the estimate $$
|J_0^\eta|\le \frac{C\rho^\frac{m}{2(m-1)} \|\psi(\cdot,
\de)\|_{C^3}}{|\lambda\rho^\frac32|^{\frac{1}{2}+\ve}|\tilde\si
 |^{l_\mu+c_\mu\ve}},
$$
provided again that $U$ and $V$ are small enough.
Now, if $\mu<6,$ then this estimate agrees with \eqref{8.19}, and we are done. 
\medskip

So, assume finally that  $\mu\ge6.$ Since $l_m$ is increasing in $m,$ we may replace $l_\mu$ by $l_{m-1}$ in this estimate, and clearly we have $c_\mu=c_m=2.$  Recall also that here 
$\tilde\si=\si\rho^\frac{3-2m}{2(m-1)}$ and $\rho\ge |\si|^{\frac{2(m-1)}{2m-3}}.$
Then the total exponent of $\rho$ in this estimate, except for  the terms containing $\ve,$ is 
$\frac{-3}{4(m-1)(2m-5)},$   and $\rho^{\frac{-3}{4(m-1)(2m-5)}}\le |\si|^{\frac{3}{2(2m-5)(2m-3)}}.$ Moreover, one computes that  $|\si|^{\frac{3}{2(2m-5)(2m-3)}-l_{m-1}}=|\si|^{-l_m}.$ In a similar way, if we replace $\rho$ by  $|\si|^{\frac{2(m-1)}{2m-3}}$ in the term $|\rho^\frac32|^{-\ve}|\tilde\si |^{-c_{m-1}\ve},$ we obtain the additional factor $|\si|^{-\frac{3(m-1)}{2m-3}\ve}\le |\si|^{-2\ve}$ in the estimate for $J_0^\eta.$ In combination, we obtain again the estimate \eqref{8.16}.

\medskip
This concludes the proof of the lemma as well as of Theorem \ref{s8.3}.

\qed

\bigskip
\setcounter{equation}{0}
\section{Uniform estimates for oscillatory integrals with finite type phase functions of two variables}\label{uniest}

In this section we shall provide a proof of Theorem \ref{s1.10n}.   We shall closely  follow the proof of Theorem \ref{1.2}, which did already provide uniform estimates for the Fourier transforms of surface carried measures $\widehat{\rho d\si}(\xi)$ for the contribution by the region near  the principal root jet. Notice that the assumption $\rho\ge 0$ that we had made for the estimation of the maximal operator $\M$  had  only been  introduced for convenience and was not needed for the estimations of oscillatory integrals. Without further mentioning, we shall use the same notation as in the various parts of the proof of Theorem \ref{1.2}.

We may assume that $S$ is the graph $S=\{(x_1,x_2,\phi(x_1,x_2)): (x_1,x_2)\in \Om\}$  of  a smooth real valued function of finite type  $\phi\in C^\infty(\Om)$ defined on an open neighborhood $\Om$ of the origin in $\RR^2$ and satisfying 
$$\phi(0,0)=0,\,\nabla\phi(0,0)=0,
$$
where $x^0=(0,0).$  We then have to prove 
\begin{thm}\label{s11.1}
There exists a neighborhood $\Om\subset \RR^2 $ of the origin such
that for every  $\eta\in C_0^\infty(\Om)$ the following estimate holds true for every $\xi\in\RR^3:$
\begin{equation}\label{11.1}
\Big|\int_{\RR^2}e^{i(\xi_3\phi\x+\xi_1x_1+\xi_2x_2)}\eta\x\, dx\Big|
\le C\,||\eta||_{C^3(\RR^2)}\,\log(2+|\xi|)(1+|\xi|)^{-1/h(\phi)}.
\end{equation}
\end{thm}

By decomposing $\RR^2$ into its four quadrants, we may reduce ourselves to the estimation of oscillatory integrals of the form 
$$J(\xi):=\int_{(\RR_+)^2}e^{i(\xi_3\phi\x+\xi_1x_1+\xi_2x_2)}\eta\x\, dx.
$$
Notice also that we may assume in the sequel that 
\begin{equation}\label{11.2}
|\xi_1|+|\xi_2|\le \de |\xi_3|,\quad\mbox{hence} \ |\xi|\sim |\xi_3|,
\end{equation}
where $0<\de<<1$ is a sufficiently small constant, since for $|\xi_1|+|\xi_2|> \de |\xi_3|$ the estimate \eqref{11.1} follows by an integration by parts, if $\Om$ is chosen small enough. 
Of course, we may in addition always assume that $|\xi|\ge 2.$ 

If $\chi$ is any integrable function defined on $\Om,$ we shall put
$$J^\chi(\xi):=\int_{(\RR_+)^2}e^{i(\xi_3\phi\x+\xi_1x_1+\xi_2x_2)}\eta\x\chi(x)\, dx.
$$

\medskip
The case $h(\phi)<2$ is contained in \cite{duistermaat} (here, estimate \eqref{11.1} holds true even without the logarithmic term $\log(2+|\xi|)$), so let us assume from now on that
$$h(\phi)\ge 2.
$$

The following van der Corput type lemma, due to J.~E.~Bj\"ork  (see \cite{domar}) and also G.~I.~Arhipov \cite{arhipov}, will be useful.

\begin{lemma}\label{s11.2}
Assume that $f$ is a smooth real valued function defined on an interval $I\subset \RR$ which is of polynomial type  $m\ge 2\ (m\in\NN)$, i.e., there are positive constants $c_1,c_2>0$ such that
$$
c_1\le\sum^m_{j=2}|f^{(j)}(s)|\le c_2\quad\mbox{for every}\  s\in I.
$$
Then for $\la\in\RR,$ 
$$
\Big|\int_{I}e^{i\la f(s)} g(s)\, ds\Big| \le C||g||_{C^2(I)} (1+|\la|)^{-1/m},
$$
where the constant $C$ depends only on the constants $c_1$ and $c_2.$
\end{lemma}

\medskip
 Following Section \ref{adapted} we  shall  begin with  the easiest case where the coordinates $x$ are adapted to $\phi.$ In analogy with the proof of Proposition \ref{s5.1} we then decompose $J(\xi)=\sum_{k=k_0}^\infty J_k(\xi),$ where 
\begin{eqnarray*}
J_k(\xi)&:=&\int_{(\RR_+)^2}e^{i(\xi_3\phi(x)+\xi_1x_1+\xi_2x_2)}\eta(x)\chi_k(x)\, dx\\
&=& 2^{-k|\ka|} \int_{(\RR_+)^2}e^{i\Big(2^{-k}\xi_3\phi^k(x)+2^{-k\ka_1}\xi_1x_1+2^{-k\ka_2}\xi_2x_2\Big)}\eta(\de_{2^{-k}}(x))\chi(x)\, dx,
\end{eqnarray*}
and where $\chi$ is supported in an annulus $D.$ Moreover, according to the proof of Corollary \ref{s5.2}, we can choose the  weight $\ka$ such that  $0<\ka_1\le\ka_2<1$ and 
$$
\frac1{|\ka|}=d_h(\phi_\ka)\le h(\phi_\ka)=h(\phi).
$$

 Then, as in the proof of Proposition \ref{s5.1}, given any point $x^0\in D,$ we can find a unit vector $e\in\RR^2$ and some $m\in
 \NN$  with $2\le m\le h(\phi_\ka)=h(\phi)$ such that $\pa_e^m\phi_\ka(x^0)\ne 0.$ For $k\ge k_0$ sufficiently large we can thus apply Lemma \ref{s11.2} to the $x_2$-integration in $J_k(\xi)$ near the point $x^0.$ By means of a partition of unity argument, we then get
\begin{eqnarray*}
 |J_k(\xi)|&\le& C||\eta||_{C^3(\RR^2)}\, 2^{-k|\ka|}(1+2^{-k}|\xi_3|)^{-1/m}\\
 &\le& C||\eta||_{C^3(\RR^2)}\, 2^{-\frac k{h(\phi)}}(1+2^{-k}|\xi|)^{-1/h(\phi)}.
\end{eqnarray*}
The estimate \eqref{11.1} then follows by summation in $k.$ 

\bigskip
Assume next that the coordinates $x$ are not adapted to $\phi.$ In a first step,  we then decompose $J(\xi)=J^{1-\rho_1}(\xi)+J^{\rho_1}(\xi),$ where $\rho_1$ is the cut-off function introduced in Subsection \ref{prelim} which localizes to a narrow $\ka$-homogeneous neighborhood 
$$
|x_2-b_1x_1^{m_1}|\le \ve_1 x_1^{m_1}
$$
of the curve $x_2=b_1x_1^{m_1}.$ 

The oscillatory integral $J^{1-\rho_1}(\xi)$ can be estimated in a similar way as in the case of adapted coordinates by means of   Lemma \ref{s11.2} (compare also the proof of Lemma \ref{s6.2}), so that there remains $J^{\rho_1}(\xi)$ to be considered. To this end, we decompose the domain above as in Subsection \ref{fdc} into the  domains $D_l,$  which become $\ka^l$-homogeneous in the coordinates $y$ defined by \eqref{6.3}, and the transition domains $E_l.$ Accordingly, we decompose 
$$J^{\rho_1}(\xi)=\sum_{l=l_0}^{\la} J^{\rho_l}(\xi)+\sum_{l=l_0}^{\la-1} J^{\rho_l}(\xi),
$$
where $\rho_l$ and $\tau_l$ are the cut-off functions defined in that subsection.

\medskip
\noi {\bf Estimation of $J^{\rho_l}(\xi).$} In analogy with the proof of Lemma \ref{s6.4}, after applying the change of coordinates \eqref{6.3} and performing  a dyadic decomposition as before, only with the weight $\ka$ replaced by the weight $\ka^l,$ we find that 
$J^{\rho_l}(\xi)=\sum_{k=k_0}^\infty J_k(\xi),$ where 
\begin{eqnarray*}
J_k(\xi)
= 2^{-k|\ka^l|} \int_{(\RR_+)^2}e^{i\Big(2^{-k}\xi_3\tilde\phi^k(y)+2^{-k\ka^l_1}\xi_1y_1+2^{-k\ka^l_2}\xi_2y_2+2^{-k\ka^l_2}\xi_2\psi^k(y_1)\Big)}\tilde\rho_l(y) \,\tilde\eta(\tilde\de_{2^{-k}}(y))\chi(y)\, dy,
\end{eqnarray*}
with $\psi^k(y_1)$ etc. defined as in Subsection \ref{subs6.4}. In view of \eqref{6.16} for the case $l=\la$ and \eqref{6.17} for the case $l\le \la-1$ we can then again estimate $J_k(\xi)$ by means of Lemma \ref{s11.2} applied to the $y_2$-integration and obtain that 
\begin{eqnarray*}
 |J_k(\xi)|&\le& C||\eta||_{C^3(\RR^2)}\, 2^{-k|\ka^l|}(1+2^{-k}|\xi_3|)^{-1/d_h(\tilde\phi_l)}\\
 &\le& C||\eta||_{C^3(\RR^2)}\, 2^{-\frac k{h(\phi)}}(1+2^{-k}|\xi|)^{-1/h(\phi)},
\end{eqnarray*}
since $d_h(\tilde\phi_l)\le h(\phi),$ except for the case where $l=\la$ and where $\tilde\phi_l=\tilde\phi_p$ is one of the exceptional polynomials $P$ given by \eqref{4.5}, with
 $\la_1+\la_2>0.$  However, the contribution of the ''exceptional domain'' \eqref{6exc} to $J(\xi)$ can be estimated  by Proposition \ref{s7.5} (compare the corresponding discussion in Subsection \ref{exception}). Since $h(\phi)\ge 2,$ the estimate \eqref{7.10} in that proposition is stronger than the desired estimate \eqref{11.1}.
 
 By summing over all $k,$ we see that $J^{\rho_l}(\xi)$ satisfies estimate \eqref{11.1}.

\medskip
\noi {\bf Estimation of $J^{\tau_l}(\xi).$} Following Subsection \ref{tau}, we decompose 
$$
J^{\tau_l}(\xi)=\sum_{j,k} J_{j,k}(\xi),
$$
where summation takes place over all pairs $j,k$ satisfying \eqref{6.18}, i.e., 
\begin{equation}\label{11.3}
a_lj+M\le k\le a_{l+1}j-M,
\end{equation}
with $J_{j,k}(\xi)$ given by
\begin{eqnarray*}
J_{j,k}(\xi)&:=&\int_{\RR^2}e^{i\Big(\xi_3\tilde\phi(y)+\xi_1y_1+\xi_2y_2+\xi_2\psi(y_1)\Big)}\tilde\tau_l(y)\tilde\eta(y)\chi_{j,k}(y)\, dy\\
&=& 2^{-j-k}\int_{\RR^2}e^{i\Big(\xi_3\tilde\phi^{j,k}(y)+2^{-j}\xi_1y_1+2^{-k}\xi_2y_2+\xi_2\psi(2^{-j}y_1)\Big)}\tilde\tau^{j,k}(y)\tilde\eta^{j,k}(y)\chi(y_1) \chi(y_2)\, dy.
\end{eqnarray*}
Here, we have kept the notations from Subsection \ref{tau}. Assume first that $\phi$ is analytic. Then, by  \eqref{6.20}, 
$$
\tp^{j,k}(y)=2^{-(A_l j+B_l k)}\Big(c_ly_1^{A_l} y_2^{B_l}+O(2^{-C M})\Big)
$$
for some constant $C>0,$ where $A_l$ and $B_l$ are given by \eqref{6.12} and $M$ can still be chosen as large as we wish, and where
$$
\pa_2^2(y_1^{A_l} y_2^{B_l})\sim 1.
$$
We can thus again apply Lemma \ref{s11.2}, with $m=2,$ to the $y_2$-integration in $J_{j,k}(\xi)$  and obtain
\begin{eqnarray*}
|J_{j,k}(\xi)|&\le& C||\eta||_{C^3(\RR^2)}\, 2^{-j-k}(1+2^{-(A_lj+B_lk)}|\xi_3|)^{-1/2}\\
&\sim& C||\eta||_{C^3(\RR^2)}\, 2^{-j-k}(1+2^{-(A_lj+B_lk)}|\xi|)^{-1/2}.
\end{eqnarray*}
Then 
$$|J^{\tau_l}(\xi)|\le C||\eta||_{C^3(\RR^2)}\Big(J^{\tau_l}_0(\xi)+J^{\tau_l}_\infty(\xi)\Big),
$$
with
\begin{eqnarray*}
J^{\tau_l}_0(\xi)&:=&\sum_{(j,k)\in I_0} 2^{-(1-\frac{A_l}2)j-(1-\frac{B_l}2)k}|\xi|^{-1/2},\\ 
J^{\tau_l}_\infty(\xi)&:=&\sum_{(j,k)\in I_\infty}2^{-j-k},
\end{eqnarray*}
where $I_0$  and $I_\infty$ denote the index sets
$$I_0:=\{(j,k)\in \NN^2: A_l j+B_lk\le \log |\xi|\quad\mbox{and}\ a_lj\le k\le a_{l+1} j\}$$
and
$$I_\infty:=\{(j,k)\in \NN^2:A_l j+B_lk> \log |\xi|\}.
$$
These estimates can easily be summed in $j$ and $k$  by means of the following auxiliary result.

\begin{lemma}\label{s11.3}
Let $0<a_1<a_2$ and $b_1,b_2\ge 0$ with $b_1+b_2>0$ be given.  For  $\ga>0,$ consider the triangle  $A_\ga:=\{(t_1,t_2)\in(\RR_+)^2: a_1t_1\le t_2\le a_2t_1\ \mbox{ andÊ} \ b_1t_1+b_2t_2\le\ga\},$ and  denote by $(0,0)$ and 
$\ga X_1$ and $\ga X_2,$ with 
$$X_1:=\frac 1{b_1+a_1b_2}(1,a_1)\mbox{ and }\  X_2:=\frac 1{b_1+a_2b_2}(1,a_2),
$$
the three vertices of $A_\ga.$ Assume that $\mu=(\mu_1,\mu_2)\in\RR^2$ is such that 
\begin{equation}\label{11.4}
\mu\cdot X_1< \mu\cdot X_2.
\end{equation}
\bee
\item[(a)] If $\mu\cdot X_2>0,$ then
$$\int_{A_\ga}e^{\mu\cdot t}dt\le C\, e^{\ga\mu\cdot X_2}.$$
\item[(b)] If $\mu\cdot X_2=0,$ then 
$$\int_{A_\ga}e^{\mu\cdot t}dt\le C\, \ga.$$
\item[(c)] If $\mu\cdot X_2< 0,$ then 
$$\int_{A_\ga}e^{\mu\cdot t}dt\le C,$$
\ee
where the constant $C$ in these estimates depends only on the $a_j, b_j$  and $\mu.$

Similarly, if we put  $B_\ga:=\{(t_1,t_2)\in(\RR_+)^2: a_1t_1\le t_2\le a_2t_1\ \mbox{ andÊ} \ b_1t_1+b_2t_2\ge\ga\},$ then the following holds true:
\bee
\item[(d)] If $\mu\cdot X_2< 0,$ then 
$$\int_{B_\ga}e^{\mu\cdot t}dt\le C \, e^{\ga\mu\cdot X_1}.$$
\ee
\end{lemma}

\proof Let us change to the coordinates $\x$ given by 
$$(t_1,t_2)=(x_1+x_2,a_1x_1+a_2x_2).
$$
In these coordinates, $A_\ga$ and $ X_1,  X_2$ correspond to 
$$
\tilde A_\ga:=\{(x_1,x_2)\in(\RR_+)^2: (b_1+a_1b_2)x_1+(b_1+a_2b_2)x_2\le\ga\}
$$
and 
$$\tilde X_1:=(\frac 1{b_1+a_1b_2},0),\ \tilde X_2:=(0,\frac 1{b_1+a_2b_2}),
$$
respectively. Moreover, $\mu\cdot t=\tilde\mu\cdot x,$ where $\tilde\mu\cdot\tilde X_1<\tilde\mu\cdot\tilde X_2,$ i.e., 
\begin{equation}\label{11.5}
\frac{\tilde\mu_1}{b_1+a_1b_2}<\frac{\tilde\mu_2}{b_1+a_2b_2}.
\end{equation}

Now, in case (a) we have  $\tilde\mu_2>0,$ so that because of \eqref{11.5}
\begin{eqnarray*}
\int_{A_\ga}e^{\mu\cdot t}dt&=&C\int_{\tilde A_\ga}e^{\tilde\mu\cdot x}dx
=C\frac 1{\tilde\mu_2}e^{\frac {\tilde\mu_2\ga}{b_1+a_2b_2}}\int_{0}^{\frac{\ga}{b_1+a_1b_2}} e^{\Big(\tilde\mu_1-\frac {\tilde\mu_2(b_1+a_1b_2)}{b_1+a_2b_2}\Big)x_1}\, dx_1\\
&\le&C\frac 1{\tilde\mu_2}e^{\frac {\tilde\mu_2\ga}{b_1+a_2b_2}}=\frac C{\tilde\mu_2}e^{\ga\mu\cdot X_2},
\end{eqnarray*}
where $C$ depends only on $a_1$ and $a_2.$ 
\medskip In case (b), we have $\tilde\mu_2=0$ and  $\tilde\mu_1<0,$ so that a similar estimation as before leads to 
$$\int_{A_\ga}e^{\mu\cdot t}d\le C\frac{\ga\tilde\mu_2}{b_1+a_2b_2},
$$
and the case (c) is obvious, since here $\tilde\mu_1,\tilde\mu_2<0.$

The estimate in (d) is obtained in an analogous way as the one in (a).

\qed

To estimate $J^{\tau_l}_0(\xi),$ we put $\mu:=(\frac{A_l}2-1,\frac{B_l}2-1)$ and $a_1:=a_l, a_2:=a_{l+1},\ b_1:=A_l,b_l:=B_l, \ \ga:=\log |\xi|$ in Lemma \ref{s11.3}. Then 
$$X_1=\frac 1{A_l+a_lB_l}(1,a_l),\ X_2=\frac 1{A_l+a_{l+1}B_l}(1,a_{l+1}),
$$
and  (compare also the discussion in Subsection \ref{roots})
$$\mu\cdot X_1=\frac 12-\frac{1+a_l}{A_l+a_lB_l}=\frac 12-\frac1{d_h(\tilde\phi_{\ka_l})},\quad \mu\cdot X_2=\frac 12-\frac{1+a_{l+1}}{A_l+a_{l+1}B_l}=\frac 12-\frac1{d_h(\tilde\phi_{\ka_{l+1}})}.
$$
Since $d_h(\tilde\phi_{\ka_l})<d_h(\tilde\phi_{\ka_{l+1}}),$ we see that condition \eqref{11.4} is satisfied. Comparing  the sum in $J^{\tau_l}_0(\xi)$ with a corresponding integral and applying Lemma \ref{s11.3} we thus find that 
$$|J^{\tau_l}_0(\xi)|\le C \logÊ|\xi| \, |\xi|^{-1/2}\le C\logÊ(2+|\xi|) \, |\xi|^{-1/h(\phi)},
$$
if $\mu\cdot X_2\le 0,$  and 
$$|J^{\tau_l}_0(\xi)|\le C  |\xi|^{-1/2}\,\exp{\Big(\log|\xi| (1/2-1/d_h(\tilde\phi_{\ka_{l+1}})\Big)}
\le C\, |\xi|^{-1/d_h(\tilde\phi_{\ka_{l+1}})},
$$
if $\mu\cdot X_2> 0.$  
Since $d_h(\tilde\phi_{\ka_{l+1}})\le h(\phi),$ this shows that $J^{\tau_l}_0(\xi)$ satisfies  the estimate \eqref{11.1}.

\medskip
Similarly, in order to estimate $J^{\tau_l}_\infty(\xi),$ we put $\mu:=(-1,-1)$ in Lemma \ref{s11.3} (d). Then $\mu\cdot X_2=-1/d_h(\tilde\phi_{\ka_{l+1}})<0,\ \mu\cdot X_1=-1/d_h(\tilde\phi_{\ka_{l}})<\mu\cdot X_2, $ so that we obtain
$$|J^{\tau_l}_\infty(\xi)|\le C  \exp{\Big(\log|\xi| (-1/d_h(\tilde\phi_{\ka_{l}})\Big)}
\le C\, |\xi|^{-1/h(\phi)}.
$$

In combination, we have seen that all $J^{\tau_l}(\xi)$ satisfy the estimate \eqref{11.1}, at least when $\phi$ is analytic. However, the case of a general finite type function $\phi$ can again be reduced to the analytic case along the lines of Subsection \ref{smoothan}. Notice here that we have only made use of the van der Corput type Lemma \ref{s11.2} in the preceding estimates, and this lemma allows for small perturbations of the phase function. 
\bigskip

What remains to be estimated is the contribution of a small domain of the form \eqref{7.1} to $J(\xi),$ i.e., we are left with the oscillatory integral $J^{\rho_0}(\xi)$ which, after a change of coordinates, is given by \eqref{7.3}. With a slight abuse of notation, we shall therefore adapt the notation from Section \ref{near-root} and write 
$$J(\xi):=J^{\rho_0}(\xi)=\int_{(\RR_+)^2} e^{i\Big(\xi_1 x_1+\xi_2 \psi(x_1)+\xi_2 x_2+\xi_3 \phi(x)\Big)}\rho\Big(\frac{ x_2}{\ve_0x_1^{a}}\Big)\eta(x)dx,$$ 
where here $\phi$ and $\psi$ satisfy the Assumptions \ref{s7.3}.  
We may also in this context assume that condition \eqref{7.4} is satisfied, since otherwise we can again obtain the desired estimate for $J(\xi)$ by means of Lemma \ref{s11.2} applied to the $x_2$-integration in $J(\xi).$ However, under these assumptions we had derived estimates for $J(\xi)$ in Sections \ref{near-root} and  \ref{proof}, and what remains to be shown is that these estimate are sufficient also in order to establish \eqref{11.1}.
\medskip

If $\pa_2\phi_p(1,0)\ne 0,$ then Proposition \ref{s7.5} immediately implies the desired estimate, since $h(\phi)\ge 2.$ 

\medskip 

If $\pa_2\phi_p(1,0)=0,$ we apply the domain decomposition algorithm of Section \ref{near-root} and are left with the estimation  of the oscillatory integrals $J^{\tau_l}$ and 
$J^{\rho_{l+1}}$ defined in that section.

\medskip
We begin with $J^{\tau_l}(\xi)=\sum_{j,k}J_{j,k}(\xi),$ where $J_{j,k}$ is as defined in Subsection \ref{subsec8.1} and where summation takes place again over the set of indices $j,k$ satisfying \eqref{11.3}. Observe that according to our discussion in Subsection \ref{subsec7.3} we have here $\ka_1=1/n,$ $\ka_2/\ka_1\ge 2$ and $\ka_1 A_1+\ka_2B_1=1,$ where $B_1=B\ge 3$ (compare \eqref{7.13n}). This implies that $\ka_2\le 1/3$ and hence 
$$\ka_1\le1/6, \ \ka_2\le 1/3.
$$
From Proposition \eqref{s9.3} we then conclude that 
$$|J_{j,k}(\xi)|\le C ||\eta ||_{C^3(\RR^2)} 2^{-j-k}(1+2^{-n j}|\xi|)^{-\ka_1} (1+2^{-nj}\si_{j,k}|\xi|)^{-\ka_2},
$$
hence
$$|J_{j,k}(\xi)|\le C ||\eta ||_{C^3(\RR^2)} 2^{-j-k}\Big(1+2^{- j}2^{-(A_lj+B_lk)\ka_2}|\xi|^{\ka_1+\ka_2} \Big)^{-1}.
$$ 
Then 
$$|J^{\tau_l}(\xi)|\le C||\eta||_{C^3(\RR^2)}\Big(J^{\tau_l}_0(\xi)+J^{\tau_l}_\infty(\xi)\Big),
$$
where here 
\begin{eqnarray*}
J^{\tau_l}_0(\xi)&:=&\sum_{(j,k)\in I_0} 2^{-k+(A_lj+B_lk)\ka_2}|\xi|^{-|\ka|},\\ 
J^{\tau_l}_\infty(\xi)&:=&\sum_{(j,k)\in I_\infty}2^{-j-k},
\end{eqnarray*}
with index sets
$$I_0:=\{(j,k)\in \NN^2: j+(A_l j+B_lk)\ka_2\le \log( |\xi|^{|\ka|})\ \mbox{and}\ a_lj\le k\le a_{l+1} j\}$$
and
$$I_\infty:=\{(j,k)\in \NN^2:j+(A_l j+B_lk)\ka_2 > \log( |\xi|^{|\ka|})\}.
$$
Since $j\le k/{a_l}$ and $k\le c\log|\xi|$ in $I_0,$ summing first in $j$ and then in $k$ we obtain
\begin{eqnarray*}
|J^{\tau_l}_0(\xi)|&\le & C\sum_{k\le c\log|\xi|} 2^{\Big((\frac{A_l}{a_l}+B_l)k)-1\Big)k}|\xi|^{-|\ka|}C=\sum_{k\le c\log|\xi|} 2^{(\ka_2/\ka_2^l-1)k}|\xi|^{-|\ka|}.
\end{eqnarray*}
But, $\ka_2/\ka_2^l\le 1$ by Lemma \ref{s9.2}, so that $|J^{\tau_l}_0(\xi)|\le C(\log|\xi|)\,|\xi|^{-|\ka|}.$
\medskip

Similarly, since $(A_lj+B_lk)\ka_2\le k$ (compare Lemma \ref{s9.2}), we have 
$j+k> \log( |\xi|^{|\ka|}).$ Putting $r:=j+k,$ we thus see that 
\begin{eqnarray*}
|J^{\tau_l}_\infty(\xi)|&\le & C\sum_{r\ge \log( |\xi|^{|\ka|})} r2^{-r}\le C' (\log|\xi|)\,|\xi|^{-|\ka|}.
\end{eqnarray*}

Since $|\ka|=1/h(\phi),$  we thus see that $J^{\tau_l}(\xi)$ satisfies estimate \eqref{11.1}.

\medskip
What remains are the $J^{\rho_{l+1}}(\xi),$ respectively the oscillatory integrals $J(\xi)$  given by \eqref{9.9}, which we decompose according to \eqref{9.10} into $J(\xi)=\sum_{k=k_0}^\infty J_k(\xi).$ By Proposition \ref{s9.4}, we have 
\begin{eqnarray*}
|J_{k}(\xi)|\le C ||\eta ||_{C^3(\RR^2)} \,2^{-|\ka'|k}\sigma_k^{-(l_m+c\ve)}(2^{-\ka'_1nk}|\xi|)^{-1/2-\ve/2}
\end{eqnarray*}
for every sufficiently small $\ve>0,$ where $l_m<1/4,$ and by  the definition of $J_k(\xi)$ in Subsection \ref{subsec8.2} we also have $|J_k(\xi)|\le C ||\eta ||_{C^3(\RR^2)}\, 2^{-|\ka'|k}.$ Putting in the definition of $\si_k,$  this implies
\begin{eqnarray*}
|J_{k}(\xi)|&\le& C ||\eta ||_{C^3(\RR^2)}\, 2^{-|\ka'|k}\Big(1+\sigma_k^{\frac 14}2^{-\frac{\ka'_1nk}2}|\xi|^{1/2}\Big)^{-1}\\
&\le& C ||\eta ||_{C^3(\RR^2)}\, 2^{-|\ka'|k}\Big(1+2^{-\frac{(1+\ka'_1n)k}2}|\xi|\Big)^{-1/2}\\
&\le& C ||\eta ||_{C^3(\RR^2)}\, 2^{-|\ka'|k}\Big(1+2^{-\frac{(1+\ka'_1n)k}2}|\xi|\Big)^{-|\ka|},
\end{eqnarray*}
because  $\frac 1{|\ka|}=h(\phi)\ge 2.$ Moreover, by \eqref{9.18}, we have $\frac{(1+\ka'_1n)|\ka|}2\le |\ka'|,$ so that 
\begin{eqnarray*}
\sum_{k\lesssim \log |\xi|} 2^{-|\ka'|k}2^{\frac{(1+\ka'_1n)|\ka| k}2}|\xi|^{-|\ka|}\le 
C\,(\log |\xi|) |\xi|^{-|\ka|},
\end{eqnarray*}
and 
$$\sum_{k:\  2^{\frac{(1+\ka'_1n)k}2}>|\xi|} 2^{-|\ka'|k}\le C\, |\xi|^{-\frac{2|\ka'|}{1+\ka'_1n}}
\le |\xi|^{-|\ka|}.
$$
This shows that also $J(\xi)$ given by \eqref{9.9} satisfies estimate \eqref{11.1}, which completes the proofs of Theorem \ref{s11.1} and Theorem \ref{s1.10n}.

\bigskip
\setcounter{equation}{0}
\section{Proof of the remaining statements in the Introduction and refined results}\label{sharpness}
In this section, we shall prove the remaining results and claims that have been stated in the Introduction.

\subsection{Invariance of the notion of height $h(x^0,S)$ under affine transformations}
We assume  that $x^0=(0,0,1)=:e_3 $ and $T_{x^0}=\{x_3=0\}=:V,$ and that our hypersurface $S$ is the graph 
$$
 S=\{(x_1,x_2,1+ \phi\x): \x\in \Om \}
$$
of a smooth function $1+\phi$ defined on an open neighborhood $\Om$ of $0\in\RR^2$ and satisfying the conditions
$$
\phi(0,0)=0,\, \nabla \phi(0,0)=0.
$$
Consider an affine linear change of coordinates $F:u\mapsto w+Au$ of $\RR^3$ which fixes the point $x^0,$ i.e., 
$F(e_3)=e_3,$ and so that the derivative  $DF(x^0)$ leaves the tangent space $T_{x^0}S$ invariant, i.e., $A(V)=V.$ Here, $A\inÊGL(3,\RR)$ and $w\in\RR^3$ is a fixed translation vector. We then denote by $B:=A|_V$ the induced linear isomorphism of $V.$ If we decompose $w=v+\mu e_3,$ with $v\in V$ and $\mu\in\RR,$ and write elements of $\RR^3$ as $(x,x_3),$ with $x\in\RR^2,$ then from $w+Ae_3=e_3$ one computes that 
$$F(x,x_3)=(Bx+(1-x_3)v,\mu+(1-\mu)x_3).
$$
Then 
$$F(S)=\{(Bx-\phi(x)v,1+(1-\mu)\phi(x):\x\in\Om\}.
$$
Notice that $1-\mu\ne 0,$ since $F$ is assumed to be bijective. By our assumptions on $\phi,$ the mapping $\vp:x\mapsto y=Bx-\phi(x)v$ is a local diffeomorphism near the origin with $\vp(0)=0,$ and we can write $F(S)$ locally as the graph of the smooth function
$$1+\tilde\phi(y):=1+(1-\mu)\phi(\vp^{-1}(y)).
$$
Since $h(\phi)=h(\tilde\phi),$ we see that $h(x^0,S)=h(x^0,F(S)),$  which proves the invariance of our notion of height $h(x^0,S)$ under affine linear changes of coordinates.

\subsection{Proof of Proposition \ref{s1.8} and remarks on the critical exponent $p=h(x^0,S)$}

We are first going to prove Proposition \ref{s1.8}. As outlined in the Introduction, we may assume without loss of generality that  the hypersurface  $S$ is given as the  graph
$$
S=\{(x_1, x_2,1+ \phi\x): (x_1,x_2)\in \Om \}
$$
of a smooth function $1+\phi$ defined on an open neighborhood $\Om$ of $(0,0)\in\RR^2$ and satisfying the conditions
$$
\phi(0,0)=0,\, \nabla \phi(0,0)=0,
$$
and that $x^0=(0,0,1),$ so that the affine tangent plane $x^0+T_{x^0}S$ is the plane $\{x_3=1\}.$ Then $d_{T,x^0}(x)=|\phi\x|,$ so that we have to show that for every neighborhood $\Om$ of the origin
\begin{equation}\label{10.1}
\int_\Om|\phi(x)|^{-1/p}\, dx=\infty
\end{equation}
whenever  $p< h(\phi)$. Moreover, if $\phi$ is analytic, then we need to show that \eqref{10.1} holds also for the critical exponent  $p=h(\phi).$

\medskip
To this end, observe first that we may reduce ourselves to the case where the coordinates $x$ are adapted to $\phi$ by applying  the change of coordinates \eqref{6.1} (compare \cite{varchenko}  and \cite{ikromov-m}) to the integral in \eqref{10.1}. Recall that then one of the following three cases applies:

\bee
\item[(a)] $\pi(\phi)$ is a compact edge, and either
$ \frac {\ka_2}{\ka_1} \notin\bN,$
or  $\frac {\ka_2}{\ka_1}\in \bN$  and $m(\phi_p)\le d(\phi).$
\item[(b)] $\pi(\phi)$ consists of  a vertex. 
\item[(c)] $\pi(\phi)$ is unbounded. 
\ee
Moreover, in this case we have $h(\phi)=d(\phi_p).$
\medskip

First, we consider the cases (a) and  (b), where the principal face of the Newton
polyhedron of $\phi$  is a compact set.  
\begin{prop}\label{s10.1}
If the principal face $\pi(\phi)$ of the Newton polyhedron of the function $\phi,$ when expressed in adapted coordinates, is compact, then \eqref{10.1} holds for every $p\le h(\phi).$
\end{prop}

\proof As in the proof of Corollary \ref{s5.2}, we can in this situation choose a weight $\ka=(\ka_1,\ka_2)$ such that  $h(\phi)=\frac{1}{|\ka|}=\frac{1}{\ka_1+\ka_2},$ where $0<\ka_1\le \ka_2$ without loss of generality. Then the $\ka$-principal part $\phi_\ka$  of the function $\phi$ is a weighted $\ka$-homogeneous polynomial  of degree $1.$    

We may also assume that  $\ka_1$ and $\ka_2$ are rational numbers. Then  we can find even positive integers  $q_1,q_2$ and a positive integer $r$ such that 
$\ka_1=\frac{r}{q_1}, \ka_2=\frac{r}{q_2}.$

The quasi-norm $N(x):=(x_1^{q_1}+x_2^{q_2})^{1/r}$ is then $\ka$-homogeneous of degree $1$ and smooth away from the origin. Denote by $\Sigma:=\{(y_1,y_2): \rho(y_1,y_2)=1\}$ the associated  ''unit circle," and let $(y_1(\theta),y_2(\theta)),\ 0\le \theta<1,$ be  a smooth parametrization of  $\Sigma.$ We can then introduce generalized polar coordinates 
$(\rho, \theta)$ for  $\RR^2\setminus\{0\}$ by writing 
$$
x_1:=\rho^{\ka_1}y_1(\theta), \quad
x_2:=\rho^{\ka_2}y_2(\theta),\quad \rho>0.
$$
It is well-known and easy to see that the Lebesgue measure on $\RR^2$ then decomposes 
as 
$$dx_1dx_2=\rho^{|\kappa|-1}\,d\ga(\theta),
$$ where $d\ga(\theta)$ is a positive  Radon measure such that $\int_\Sigma d\ga(\th)>0.$ Let us also assume without loss of generality  that   $\Om=\{\x:\, \rho\x<\ve \},$ where $\ve >0.$ 

If we now decompose  $\phi=\phi_\ka+\phi_r$ as before into the $\ka$-principal part $\phi_\ka$ and the remainder term $\phi_r,$ and express $\phi$ in polar coordinates
$\tilde\phi(\rho,\theta):=\phi(\rho^{\ka_1}y_1(\theta),\rho^{\ka_2}y_2(\theta)),$ then 
$$\tilde\phi(\rho,\theta)=\rho\,\Big(\tilde\phi(1,\theta)+\tilde\phi_r(\rho,\theta)\Big),
$$
where $\tilde\phi_r(\rho,\theta)=O(\rho^\de)$ for some $\de>0$ as $\rho\to 0.$  In particular, also $\tilde\phi_r(\rho,\theta)$ is bounded, which is all that we need. By passing to these polar coordinates, we obtain 
$$
\int_\Om|\phi(x)|^{-1/h(\phi)}\, dx=\int_0^\ve\frac{d\rho}{\rho}
\int_\Sigma\Big|\tilde\phi(1,\theta)+\tilde\phi_r(\rho,\theta)\Big|^{-1/h(\phi)} d\ga(\theta)\ge c\int_0^\ve\frac{d\rho}{\rho}.
$$
In the last inequality $c$ is a positive constant and therefore the
integral diverges. This proves the proposition.

\qed

There remains the case (c) where  the principal face is unbounded.
\begin{prop}\label{s10.2}
Assume that  the principal face $\pi(\phi)$ of the Newton polyhedron of the function $\phi,$ when expressed in adapted coordinates, is unbounded.
\bee
\item[(i)] Then  \eqref{10.1} holds for every $p< h(\phi).$
\item[(ii)] If $\phi$ is assumed to be analytic, then \eqref{10.1} holds also for $p= h(\phi).$
\ee
\end{prop}

\proof We first prove (i), so assume that $p< h(\phi).$
Here, we can apply a similar reasoning as in the proof of case (c) in Corollary \ref{s5.2}.   The principal face $\pi(\phi)$ is a horizontal half-line, with left endpoint $(\nu_1,N),$ where $\nu_1<N=h(\phi).$ Notice that $N\ge 2,$ since for $N=1$ we had $\nu_1=0,$ which is not possible given our assumption $\nabla\phi(0,0)=0.$ We can then choose $\ka$ with $0<\ka_1<\ka_2$ so that  the line  $\ka_1t_1+\ka_2t_2=1$ is a supporting line to the Newton polyhedron of $\phi$ and that the point $(\nu_1,N)$ is the only point of $\N(\phi)$ on this line. Moreover, we can choose $\ka_2/\ka_1$ as large as we wish, so that we may  assume that 
$$p<|\ka|^{-1}<h(\phi).
$$
Then the $\ka$-principal part $\phi_{\ka}$ of $\phi$ is of the form $\phi_{\ka}(x)=cx_1^{\nu_1}x_2^N, $ with $c\ne 0,$ and it is $\ka$-homogeneous of degree $1.$  

By passing to generalized polar coordinates as in the proof of Proposition \ref{s10.1} we then see that 
$$
\int_\Om|\phi(x)|^{-1/p}\, dx=\int_0^\ve\frac{d\rho}{\rho^{\frac 1p-|\ka|+1}}
\int_\Sigma\Big|\tilde\phi(1,\theta)+\tilde\phi_r(\rho,\theta)\Big|^{-1/p}d\ga(\theta),
$$
where again $\tilde\phi_r(\rho,\theta)$ is bounded. Since $\frac1p-|\ka|>0,$ we conclude that the last integral diverges.

\medskip
In order to prove (ii), observe that if $\phi$ is analytic, then  there
exists a non-trivial  analytic function $f$ near the origin so that $\phi\x=x_2^Nf\x,$ where again $N=h(\phi).$ 
Then, for sufficiently small $\ve >0,$ we have 
$$
\int_\Om \frac{dx_1dx_2}{|\phi\x|^{1/h(\phi)}}\ge
\int_{-\varepsilon}^\varepsilon
\frac{dx_2}{|x_2|}\int_{-\varepsilon}^\varepsilon
\frac{dx_1}{|f\x|^{\frac1N}}.
$$
Obviously the last integral diverges.
\qed

\begin{remark}\label{s10.3}
If $\phi$ is a finite type smooth function and the principal face is a 
noncompact set then the integral  $\int_\Om|\phi(x)|^{-1/h(\phi)}\, dx$ may be  convergent.

 An example is  given by the function $\phi \x=x_2^2+e^{-x_1^{-\al}}$ considered by A.~ Iosevich and E.~ Sawyer in \cite{iosevich-sawyer2}. Here we have  $h(\phi)=2,$ and the associated
integral converges whenever $0<\al<1.$ Correspondingly, it has been shown in \cite{iosevich-sawyer2} that the maximal operator associated to the hypersurface $x_3=1+x_2^2+e^{-x_1^{-\al}}$
is $L^2$ bounded whenever $0<\al<1$ and unbounded for $p<2$ (the latter statement follows of course also from  Proposition \ref{s10.2}).  However, if $\al\ge1,$
then it is unbounded whenever $p\le2.$
\end{remark}

We have thus obtained a  confirmation of Iosevich- Sawyer's   conjecture
for analytic hypersurfaces \cite{iosevich-sawyer2}, and for smooth finite
type hypersurfaces we have a partial confirmation of the conjecture.
The conjecture remains open when $p=h(\phi)$ in the case where the  principal face of $\phi$  is
unbounded  in an adapted coordinate system.

\medskip

\medskip
\subsection{Proof of Theorem \ref{s1.10}  }
By means of a smooth partition of unity consisting of non-negative functions, we may reduce ourselves to the situation where  $\rho$ is supported in a sufficiently small neighborhood of some given point $z\in S.$ 
Without loss of generality  we may then assume that $z=0, $ and that our hypersurface $S$ is the graph 
$$
 S=\{(x,\phi(x)): x\in \Om \}
$$
of a smooth function $\phi$ defined on an open neighborhood $\Om$ of $0\in\RR^{n-1}$ and satisfying the conditions
$$
\phi(0)=0,\, \nabla \phi(0)=0.
$$
Then the Fourier transform  $\widehat{\rho d\si}(0,\dots,0,\la)$ of the superficial measure $\rho d\si$ in direction of the unit normal to $S$ at $z=0$ is an  oscillatory integral of the form
$$
J(\la)=\int_{\RR^{n-1}} e^{-i\la \phi(x)}\eta(x)\, dx,
$$
where $0\le\eta\in C_0^\infty(\Om).$  By \eqref{1.5}, we have in particular that 
\begin{equation}\label{10.2}
|J(\la)|\le C_\be\,(1+|\la|)^{-\be} \mbox{ for every } \la \in\RR,
\end{equation}
where $\be>0.$

\begin{lemma}\label{s10.4}
If \eqref{10.2} holds true, then 
\begin{equation}\label{10.3}
\int_{\RR^{n-1}}|\phi(x)|^{-\ga}\eta(x)\, dx<\inftyâ 
\end{equation}
for every $\ga<1$ such that $\ga <\be.$
\end{lemma}

\proof We  choose a  sequence of smooth even functions $\chi^\nu\in C_0^\infty(\RR), \ \nu\ge 1,$ with compact
support such that $0\le \chi^\nu(\la)\le 1,$
$\supp \chi^\nu\subset  [-\nu-1, -1]\cup[1, \nu+1]$ and
$\chi^\nu(\la)=1$ for any $\la\in [-\nu, -2]\cup[2, \nu] .$ We  may clearly choose the $\chi^\nu$ so that 
for any  $k\in\NN$ their $C^k(\RR)$ norms are uniformly bounded with
respect to $\nu$. 

If $\ga<1$  is given such that $\ga <\be,$ then we define 
the Schwartz functions   $\varphi_\nu\in \S(\RR)$
by their Fourier transforms 
$$
\widehat{\varphi_\nu}(s):=\frac{\chi^\nu(\nu s)}{|s|^{\ga}}.
$$
Then, by standard scaling and integration by parts arguments, one easily finds that
\begin{equation}\label{10.4}
|\varphi_\nu(\la)|\le C(1+|\la|)^{\ga-1}.
\end{equation}

Consider next  the integral
$$
I_\nu:=\int_\RR \vp_\nu(\la)J(\la)\, d\la=\int_{\RR}\int_{\RR^{n-1}}\varphi_\nu(\la)
e^{-i\la\phi(x)}\eta(x)\,dxd\la.
$$
Due to our assumptions on $\ga$ and the estimates \eqref{10.2} and \eqref{10.4} these integrals are uniformly bounded  with respect to $\nu.$ 

On the other hand since $\varphi_\nu$  and $\eta$ both  belong to the 
Schwartz class,  we can apply Fubini's theorem and obtain 
$$
I_\nu=\int _{\RR^{n-1}}\eta(x)\widehat{\varphi_\nu}(\phi(x))\,dx.
$$
Since the integrand is non-negative, this implies the lower estimate 
$$
C\ge |I_\nu|\ge \int _{\frac 2\nu\le |\phi(x)|\le 1} |\phi(x)|^{-\ga}\eta(x)\, dx
$$
for every $\nu,$ where $C$ is a fixed constant. The estimate \eqref{10.3} now follows if we let $\nu$ tend to infinity.

\qed 

Theorem \ref{s1.10} is now an easy consequence of Lemma \ref{s10.4}. Indeed, by Remark \ref{s1.7} it suffices to prove the estimate \eqref{1.6} only for affine tangent planes $H=z+T_{z}S$ to $S,$ where $z\in S$ is sufficiently close to the support of $\rho.$ For these, the previous reasoning applies, and since then $d_H(x)=|\phi(x)|,$  we see that \eqref{1.6} is an immediate consequence of Lemma \ref{s10.4}.

\begin{remark}\label{s10.5}
By the same reasoning, Lemma \ref{s10.4} also shows that if $z\in S$ and if $0<\beta\in\frak B(z,S)$ and  $\ga<\min\{1,\be\},$ then $\ga\in\frak C(z,S).$ 
\end{remark}

\medskip
\subsection{Proof of  Corollary \ref{s1.11} }
Note first that always $h(x^0,S)\ge 1.$ We first assume that $h(x^0,S)> 1.$ If we had $\be>1/h(x^0,S),$ then we could choose some $p>1$ in this case such that    $\be>1/p>1/h(x^0,S).$  Then Theorem \ref{s1.10} in combination with Proposition \ref{s1.8} 
 would  imply that $p\le 1/\be,$  a contradiction. 

There remains the case where $h(x^0,S)=1.$ We may again assume that $S$ is given as the graph of a smooth function $\phi,$  with $\phi$ satisfying \eqref{1.2} and $x^0=(0,0,0).$ Assuming without loss of generality that the coordinates are adapted to $\phi,$ it is then easy to see that the Hessian matrix $D^2\phi(0,0)$ is non-degenerate. The asymptotic form of the method of stationary phase then shows that $\ga\le 1=1/h(\phi)=1/h(x^0,S).$

\qed

\medskip
\subsection{Proof of  Theorem \ref{s1.12} } Let $S$ be a smooth, finite type hypersurface in $\RR^3,$ and let $x^0\in S$ be given.  Notice first that Theorem \ref{s1.10n} implies that $$\be_u(x^0,S)\ge1/h(x^0,S).
$$

Moreover, by  Corollary \ref{s1.11} we have $\be_u(x^0,S)\le 1/h(x^0,S).$ Indeed, since its proof was based on Proposition \ref{s1.8}, which made only use of the affine tangent hyperplane at the point $x^0,$ with the same arguments restricted to these tangent hyperplane we even obtain 
$$\be(x^0,S)\le 1/h(x^0,S).
$$
In combination with \eqref{s1.10} these estimates imply 
\begin{equation}\label{10.5}
\be_u(x^0,S)=\be(x^0,S)=1/h(x^0,S)\le 1.
\end{equation}

Observe next that if $ \be\in\frak B_u(x^0,S),$ then by Theorem \ref{s1.10} and \eqref{10.5} we have $\be\le 1$ and then $\be-\ve\in \frak C_u(x^0,S)$ for every sufficiently small $\ve>0.$ This implies 
$$\be_u(x^0,S)\le \ga_u(x^0,S), $$
hence by \eqref{10.5} and \eqref{1.10}
\begin{equation}\label{10.6}
1/h(x^0,S)\le \ga_u(x^0,S)\le \ga(x^0,S).
\end{equation}

Finally, if $\ga\in\frak C(x^0,S),$ then putting $p:=1/\ga$ in Proposition \ref{s1.8} we see that 
$1/\ga\ge h(x^0,S),$ hence $\ga\le 1/h(x^0,S).$ This implies $\ga(x^0,S)\le 1/h(x^0,S),$ and in combination with \eqref{10.6} we also get 
$$\ga(x^0,S)=\ga_u(x^0,S)=1/h(x^0,S).
$$
This concludes the proof of Theorem \ref{s1.12}

\qed

\bigskip

\bibliographystyle{plain}



\def\cprime{$'$} \def\cprime{$'$} \def\cprime{$'$} \def\cprime{$'$}

\end{document}